\documentclass{amsart}
\numberwithin{figure}{section}
\numberwithin{equation}{section}
\usepackage[all]{xy}
\usepackage{tikz}
\usepackage{color}
\usepackage{amssymb}
\usepackage{comment}
\usepackage{float}
\usepackage[normalem]{ulem}
\let\cal\mathcal
\def\Ascr{{\cal A}}
\def\Bscr{{\cal B}}
\def\Cscr{{\cal C}}

\def\Lscr{{\cal L}}

\def\Oscr{{\cal O}}

\def\Wscr{{\cal W}}

\let\blb\mathbb

\def \ZZ{{\blb Z}}
\def \TT{{\blb T}}
\def \NN{{\blb N}}
\def \RR{{\blb R}}

\def\id{\text{id}}
\def\Id{\operatorname{id}}

\def\Lotimes{\overset{L}{\otimes}}
\def\quot{/\!\!/}

\def\mod{\operatorname{mod}}

\def\Lie{\mathop{\text{Lie}}}

\def\coh{\mathop{\text{\upshape{coh}}}}

\def\Spec{\operatorname {Spec}}

\def\GL{\operatorname {GL}}
\def\PGL{\operatorname {PGL}}

\def\diag{\operatorname {diag}}
\def\Ext{\operatorname {Ext}}
\def\Hom{\operatorname {Hom}}
\def\uHom{\operatorname {\mathcal{H}\mathit{om}}}

\def\End{\operatorname {End}}

\def\Sl{\operatorname {SL}}

\def\relint{\operatorname {relint}}
\def\im{\operatorname {im}}

\def\Tr{\operatorname {Tr}}

\def\ker{\operatorname {ker}}

\def\End{\operatorname {End}}

\def\id{{\operatorname {id}}}

\def\add{\operatorname {add}}
\def\rk{\operatorname {rk}}

\def\Tot{\operatorname {Tot}}

\def\gldim{\operatorname {gl\,dim}}

\def\r{\rightarrow}

\DeclareMathOperator{\Proj}{Proj}

\DeclareMathOperator{\Ind}{Ind}

\DeclareMathOperator{\Aut}{Aut}

\newcommand{\ol}{\bar}

\newtheorem{lemma}{Lemma}[section]
\newtheorem{proposition}[lemma]{Proposition}
\newtheorem{theorem}[lemma]{Theorem}
\newtheorem{corollary}[lemma]{Corollary}

\newtheorem{lemmas}{Lemma}[subsection]
\newtheorem{propositions}[lemmas]{Proposition}
\newtheorem{theorems}[lemmas]{Theorem}
\newtheorem{corollarys}[lemmas]{Corollary}

\theoremstyle{definition}

\newtheorem{definition}[lemma]{Definition}

\newtheorem{definitions}[lemmas]{Definition}

{

}

\theoremstyle{remark}

\newtheorem{remark}[lemma]{Remark}
\newtheorem{remarks}[lemmas]{Remark}

\newdimen\uboxsep \uboxsep=1ex
\def\uboxn#1{\vtop to 0pt{\hrule height 0pt depth 0pt\vskip\uboxsep
\hbox to 0pt{\hss #1\hss}\vss}}

\def\uboxs#1{\vbox to 0pt{\vss\hbox to 0pt{\hss #1\hss}
\vskip\uboxsep\hrule height 0pt depth 0pt}}

\def\codim{\operatorname{codim}}
\def\Ref{\operatorname{ref}}
\def\ind{\operatorname{ind}}
\def\cone{\operatorname{cone}}
\def\Spin{\operatorname{Spin}}
\def\SO{\operatorname{SO}}
\def\Pin{\operatorname{Pin}}
\def\Sp{\operatorname{Sp}}
\def\Xs{X^{\mathbf{s}}}
\def\Ob{\operatorname{Ob}}
\def\pdim{\operatorname{pdim}}
\def\Tot{\operatorname{Tot}}
\def\cone{\operatorname{cone}}
\def\Cl{\operatorname{Cl}}

\title[Non-commutative resolutions of quotient singularities]{Non-commutative resolutions of quotient singularities for reductive groups}
\author{\v{S}pela \v{S}penko}
\email[\v{S}pela \v{S}penko]{Spela.Spenko@ed.ac.uk}
\address{School of Mathematics\\
  The University of Edinburgh\\
  James Clerk Maxwell Building\\
  The King's Buildings\\
  Peter Guthrie Tait Road\\
  EDINBURGH\\
  EH9 3FD \\
  Scotland, UK} 
\author{Michel Van den Bergh}
\email[Michel Van den Bergh]{michel.vandenbergh@uhasselt.be}
\address{Universiteit Hasselt\\ Universitaire Campus\\ 3590 Diepenbeek\\Belgium}
\thanks{The second author is a senior researcher at the FWO}
\thanks{This research was carried out while the first author was visiting the University of Hasselt
supported by the Slovenian Research Agency and in part by the Slovene Human Resources Development and Scholarship Fund. The second author was supported by the FWO grant 1503512N}


\begin{document}
\begin{abstract}
  In this paper we generalize standard results about non-commu\-tative
  resolutions of quotient singularities for finite groups to arbitrary reductive groups. We show
  in particular that quotient singularities for reductive groups
  always have non-commutative resolutions in an appropriate
  sense. Moreover we exhibit a large class of such singularities which
  have (twisted) non-commutative \emph{crepant} resolutions.

 We discuss
  a number of examples, both new and old, that can be treated using our
  methods. Notably we prove that  twisted non-commutative crepant resolutions exist  in previously unknown cases for determinantal varieties
of symmetric and skew-symmetric matrices. 

In contrast to almost all prior results in this area our techniques
are algebraic and do not depend on knowing a commutative resolution
of the singularity. 
\end{abstract}
\maketitle

\long\def\flash#1{{\color{blue}#1}}

\let\oldmarginpar\marginpar
\def\marginpar#1{\oldmarginpar{\tiny #1}}
\setcounter{tocdepth}{1}
\tableofcontents

\section{Introduction}
In this paper we generalize standard results about non-commutative
resolutions of quotient singularities for finite groups to arbitrary reductive groups.  Some
basic definitions are given in~\S\ref{ref-1.1-0}. Our main technical results
are stated in \S\ref{ref-1.5-23}-\S\ref{ref-1.6-26}.  A number of
applications are listed in \S\ref{ref-1.3-4}-\S\ref{ref-1.4-11}.

\medskip

Throughout $k$ is an algebraically closed field of characteristic
zero. All rings will be $k$-algebras. If $G$ is an algebraic group
then we denote the set of isomorphism classes of irreducible
$G$-representations by $\widehat{G}$.
\subsection{Preliminaries on non-commutative resolutions}
\label{ref-1.1-0}
We first recall the 
definition of a non-com\-mutative (crepant) resolution  (abbreviated as  NC(C)R).
\begin{definitions} \cite{DaoIyama,Leuschke,VdB32,Wemyss1}
\label{ref-1.1.1-1}
Assume that $S$ is a normal noetherian domain.
Then a \emph{non-commutative resolution}  of $S$ is
  an $S$-algebra of finite global
  dimension  of the form $\Lambda=\End_S(M)$ where $M$ is a non-zero finitely generated reflexive $S$-module.
The resolution is \emph{crepant} if $S$ is in addition Gorenstein\footnote{Sometimes NCCR's are defined assuming only that $S$ is a normal Cohen-Macaulay ring. E.g.\ \cite{IW}. However in
this paper we will only use NCCR's over Gorenstein rings.}  and $\Lambda$ is a maximal Cohen-Macaulay $S$-module.
\end{definitions}
Among other things NCRs are a crucial part of the Kuznetsov's
Homological Projective Duality program \cite{Kuznetsov4,Kuznetsov3}. The HPD dual of a smooth
projective variety is often a  NCR (possibly twisted, see below) of a
singular variety. See \cite{Kuznetsov2} for many examples.

\medskip

For the rationale behind the definition of a NCCR see \cite{VdB32}.
As shown in \cite{IW} NCCRs are optimal among the NCRs in the sense that they have excellent
homological properties which are lost when you enlarge or shrink
them. See Remark \ref{rem:nccropt} for a precise statement.

 In general  \mbox{NCCRs} are close cousins of (commutative)
crepant resolutions as defined in algebraic geometry (see e.g.\ \cite{VdBSt1}). 
They are so close that one may often pass from the commutative to the non-commutative context and vice versa. Given a
NCCR one may sometimes produce a crepant resolution as a GIT-moduli space
of representations
and conversely a crepant resolution may give rise to a NCCR obtained as the
endomorphism ring of a tilting bundle. Deep examples include the three dimensional McKay correspondence by Bridgeland-King-Reid \cite{BKR}
and Bezrukavnikov's ``non-commutative Springer resolution'' \cite{Bezrukavnikov}. For other examples see \cite{BezKal2,VdB100,VdB34,SmithGordon,VdB32}.  

\medskip

For three dimensional terminal Gorenstein singularities 
 the existence of commutative and non-commutative
 crepant resolutions are even equivalent \cite{VdB32}  and
this is part of the motivation for the algebraic approach to the three
dimensional minimal model program by Iyama and Wemyss
\cite{IW,IW1,Wemyss}.

\medskip

Under mild conditions a NC(C)R yields a ``categorical (crepant) resolution of singularities'' \cite{Kuznetsov,LuntsKuznetsov,Lunts}
as shown by the following trivial lemma which is an extension of \cite[Example 5.3]{Lunts}.
\begin{lemmas} 
\label{ref-1.1.2-2} Assume that $S$ is a finitely generated integrally closed $k$-algebra and that $\Lambda=\End_S(M)$ is a NCR for $S$. Then $\Lambda$
is smooth as a DG-algebra. Assume in addition that $M$ is a projective
$\Lambda$-module (so in particular: (1) if $S$ is a summand of $M$ or (2) if $S$ is Gorenstein, $\Lambda$ is a NCCR of $S$ and $M$ is Cohen-Macaulay). Then the functor
\[
D(S)\r D(\Lambda): N\mapsto M\Lotimes_S N
\]
is fully faithful and hence yields a categorical resolution of singularities of $S$ in the sense of \cite{Kuznetsov,LuntsKuznetsov,Lunts}. Moreover if $\Lambda$ is a NCCR then the categorical resolution is crepant in the sense of
\cite{Kuznetsov}.
\end{lemmas}
Note that the existence of categorical resolutions of singularities
(satisfying even stronger conditions) has been shown in complete
generality in \cite{LuntsKuznetsov}. However we believe that, given
their simplicity, categorical resolutions given by NCRs are of
independent interest. 

\medskip

For use below we note that a ``twisted''
NC(C)R of index $u$ is just like a NC(C)R except that it is generically a central simple algebra of index $u$
rather than a matrix ring. See \S\ref{ref-3-32} below for more precise definitions.  A NC(C)R is a twisted NC(C)R of index 1. 

\subsection{Non-commutative resolutions by algebras of covariants}
\label{ref-1.2-3}
As an introduction to the rest of this paper we will first recall some
facts in the well understood finite group case. This material is classical.
See e.g.\ \cite{Auslander}, \cite[\S J]{Leuschke}, \cite[\S1.4]{Wemyss}.

Let $G$ be a finite
group which acts on a smooth affine variety $X$.  For $U$ a finite
dimensional $G$-representation let
$M(U)\overset{\text{def}}{=}(U\otimes k[X])^G$ be the corresponding
$k[X]^G$-\emph{module of covariants} \cite{Brion,VdB9}.

Let $\mod(G,k[X])$ be the abelian category of $G$-equivariant finitely generated $k[X]$-modules. This category has a projective generator $U\otimes k[X]$ where
$U=\bigoplus_{V\in \widehat{G}} V$. 
It follows that $\mod(G,k[X])$ is equivalent to $\mod(\Lambda)$ where
\[
\Lambda=\End_{G,k[X]}(U\otimes k[X])=(\End(U)\otimes k[X])^G=M(\End_k(U)).
\]
Hence $\Lambda$ is a $k[X]^G$-``algebra of covariants'' (as $\End_k(U)$ is a $G$-equivariant $k$-algebra). Since $\mod(G,k[X])$ has finite global dimension we also
obtain
$
\gldim \Lambda<\infty
$.
We may think of $\Lambda$ as a non-commutative resolution of $k[X]^G=k[X\quot G]$ in a
weak sense.

If no element of $G$ fixes a divisor then $k[X]/k[X]^G$ is \'etale in codimension one and
from this one easily obtains
\[
\Lambda=\End_{k[X]^G}(M(U)).
\]
Hence in this case $\Lambda$ is a true NCR of $k[X]^G$ in the sense of Definition \ref{ref-1.1.1-1}. If in addition $X=\Spec SW$ for a representation $W$ and
$G\subset \Sl(W)$ then $k[X]^G$ is Gorenstein and
$\Lambda$ is
a NCCR of~$k[X]^G$.

\medskip

For general reductive groups one may attempt to construct non-commutative resolutions in a similar way
 using the properties of the category $\mod(G,k[X])$. However if $G$
is not finite  the analysis is more complicated because
of two  non-trivial issues:
\begin{enumerate}
\item as $G$ has infinitely many non-isomorphic irreducible representations  the category $\mod(G,k[X])$ does not have a 
projective generator;
\item modules of covariants are usually not Cohen-Macaulay.
\end{enumerate}
The first issue will be handled in \S\ref{ref-11.2-105} where we construct certain nice complexes
  which relate different projectives in $\mod(G,k[X])$. The second
issue is handled using the results in \cite{VdB3,Vdb1,Vdb7,VdB4}. See \S\ref{ref-4.4-57} below.

\medskip

Other
papers which discuss the homological properties of categories of $G$-equi\-variant modules and more
generally coherent sheaves are \cite{BFK,DHL}.  To the best of our
understanding the results in those papers are complementary to ours.

\subsection{General results}
\label{ref-1.3-4}
In the rest of this paper we will look for non-commutative resolutions
of quotient singularities given by algebras of covariants.  
We first state some general results. 
\begin{theorems}  \label{ref-1.3.1-5} (See \S\ref{ref-11.4-116} below.) Assume that $G$ is a reductive group acting on a smooth affine  variety $X$. 
Then there exists a finite dimensional $G$-representation $U$ containing the trivial representation
such that $\Lambda=M(\End(U))=(\End(U)\otimes k[X])^G$ satisfies $\gldim \Lambda<\infty$. 
\end{theorems} 
Since $\Lambda=M(\End(U))$ is a $k[X]^G$-algebra which is
finitely generated as $k[X]^G$-module and furthermore $k[X]^G\subset
Z(\Lambda)$ we may view $\Lambda$ as a kind of
``weak'' non-commutative resolution of $k[X]^G=k[X\quot G]$. Moreover since $U$ contains a trivial direct summand
$M(U)$ is a projective left $\Lambda$-module such that in addition one has $\End_\Lambda(M(U))=k[X]^G$. Hence
by a variant of Lemma \ref{ref-1.1.2-2} we obtain.
\begin{corollarys} 
\label{ref-1.3.2-6} Assume that $G$ is a reductive group acting on a smooth affine  variety $X$.  Then $k[X]^G$ has
a categorical resolution of singularities in the sense of \cite{Lunts} given by an algebra of covariants.
\end{corollarys}

 If $G$ is a reductive group acting on a smooth projective variety $X$ 
linearized by an ample line bundle $\Lscr$ and $U$ is a finite dimensional $G$-representation then we may define
an associated ``module'' of covariants $M^{ss}(U)$ on $X^{ss}\quot G$ which is a coherent sheaf of $\Oscr_{X^{ss}\quot G}$-modules
whose sections on $V\quot G$ for a $G$-invariant saturated\footnote{``Saturated'' means that $V$ is the inverse image of its image in $X^{ss}\quot G$.}  affine open $V\subset X^{ss}$ are given by $(U\otimes \Oscr(V))^G$.

We may prove a GIT-version of Theorem \ref{ref-1.3.1-5} when $X$ is projective.
\begin{theorems} (See \S\ref{ref-11.5-117} below.)
\label{ref-1.3.3-7} Let $G,X,\Lscr$ be as in the previous paragraph.
 Then  there exists a finite dimensional $G$-representation $U$ containing the trivial representation 
such that the coherent sheaf of algebras $\Lambda=M^{ss}(\End(U))$ on $X^{ss}\quot G$  has finite global dimension when restricted to affine opens.
\end{theorems}
Again the sheaf of algebras $M^{ss}(\End(U))$ yields a categorical
resolution of singularities of $X^{ss}\quot G$ in the sense of
\cite{Lunts} as in Corollary \ref{ref-1.3.2-6}. It would be interesting to
compare these categorical resolutions to the canonical partial
resolutions of $X^{ss}\quot G$ constructed by Kirwan in \cite{Ki1}.

\medskip

For simplicity of exposition we return to the case that $X$ is affine. The generalization
to $X$ projective are routine.
To obtain a genuine non-commutative resolution for $k[X]$ from Theorem \ref{ref-1.3.1-5}
we have to restrict $(G,X)$ as in the finite group
case. 
\begin{definitions}
\label{ref-1.3.4-8}
Below we will say that $G$ acts \emph{generically} on a smooth affine  variety $X$ if
\begin{enumerate}
\item  $X$ contains a point with closed orbit and trivial stabilizer.
\item If $X^{\mathbf{s}}\subset X$ is the locus of points that satisfy (1) then $\codim
  (X-\Xs) \ge 2$.
\end{enumerate}
If $W$ is a $G$-representation then we will say that $W$ is \emph{generic} if $G$ acts generically on
$\Spec SW\cong W^\ast$.
\end{definitions}
\begin{corollarys}
\label{ref-1.3.5-9}
Assume that $G$ acts generically on $X$.  Then there exists a finite
dimensional $G$-representation $U$ containing the trivial
representation such that $\Lambda=\End_{k[X]^G}(M(U))$ is a NCR for
$k[X]^G$.
\end{corollarys}
\begin{proof} This follows from Theorem \ref{ref-1.3.1-5} together with the fact that if the action is generic then  (see Lemma \ref{ref-4.1.3-38} below)
\[
 M(\End(U))=\End_{k[X]^G}(M(U))
\]
and furthermore $M(U)$ is reflexive.
\end{proof}
The following result was originally stated in~\cite{Yasuda} but the
proof was later retracted. We can now prove it by applying our
techniques to abelian reductive groups.
\begin{propositions} \label{ref-1.3.6-10} (See
  \S\ref{ref-11.6-118} below.)  Assume that $S\subset \ZZ^n$ is a
  finitely generated commutative positive (no units) normal
  semigroup. Then for $n\in\NN$, $n\gg 0$ the $k[S]$-module
  $M=k[\frac{1}{n}S]$ defines a NCR for $k[S]$.
\end{propositions}
In characteristic $p>0$ the $k[S]$-modules $k[\frac{1}{p^e}S]$ arise as Frobenius twists of $k[S]$.
Proposition \ref{ref-1.3.6-10} then becomes  a confirmation of the idea that in some cases large Frobenius twists provide  canonical NCRs. See e.g.\ \cite[Thm 1.6]{YasudaToda} for 
other instances of this principle.

The proof of Proposition \ref{ref-1.3.6-10} is based on the fact that $k[S]$ may be canonically written \cite{BrGu} as $R^G$ where $R=SW$ and $G$ is a (generally non-connected) abelian
reductive group for which $W$ is generic. 

\subsection{Some applications for specific quotient singularities}
\label{ref-1.4-11}
Before stating our main technical results (see \S\ref{ref-1.5-23} below) we give
some applications for specific quotient singularities.
\subsubsection{Determinantal varieties}
We obtain a new proof for the following result from~\cite{VdB100}.
\begin{theorems} (See \S\ref{ref-5-68} below.) \label{ref-1.4.1-12}
For $n< h$ let $Y_{n,h}$ be the variety
of $h\times h$-matrices of rank $\le n$.
The $k$-algebra
$k[Y_{n,h}]$ has a NCCR.
\end{theorems}
We will prove in \S\ref{ref-5.3-74} below that the NCCR we obtain in this paper is the same
as the one constructed in \cite{VdB100}.
\subsubsection{Pfaffian varieties}
\begin{theorems} (See \S\ref{ref-6-77} below.)
\label{ref-1.4.3-14} 
For $2n<h$ let $Y_{2n,h}^-$ be  the variety of skew-symmetric $h\times h$ matrices of rank $\le 2n$.
If $h$ is odd then $k[Y_{2n,h}^-]$ has a NCCR.
\end{theorems}
\begin{remarks}
If $h$ is even then we show that $k[Y_{2n,h}^-]$ has a NCR which is very similar to the NCCR which exists in the odd case.
\end{remarks}
The existence of NCCRs for $k[Y^-_{2n,h}]$ when $h$ is odd seems to be
completely new.  NCRs for $k[Y^-_{2n,h}]$ were constructed by Weyman
and Zhao in \cite{WeymanZhao} but they are not NCCRs \cite[Prop.\ 7.4]{WeymanZhao} and they are
larger than ours also in the even case (see Remark \ref{ref-6.1.3-82}
below).  NC(C)Rs for $Y^-_{4,h}-\{0\}$ were constructed in
\cite{Kuznetsov} but they do not extend to NC(C)Rs for $Y^-_{4,h}$
(see Remark \ref{ref-6.1.4-84} below).
\subsubsection{Determinantal varieties for symmetric matrices}
\begin{theorems} (See \S\ref{ref-7-85} below.)
\label{ref-1.4.5-15} For $t<h$ let $Y^+_{t,h}$ be the the variety of symmetric $h\times h$ matrices of rank $\le t$.
If $t$ and $h$ have opposite parity then $k[Y^+_{t,h}]$ has a NCCR.
If $t$ and $h$ have the same parity then
$k[Y_{t,h}^+]$ has a twisted NCCR of index $2^{\lfloor h/2\rfloor}$. 
\end{theorems}
Again the existence of twisted NCCRs for $k[Y^+_{t,h}]$ 
seems to be completely new.
NCRs for $k[Y^+_{t,h}]$ were
  constructed by Weyman and Zhao in \cite{WeymanZhao} but they are  only NCCRs when $t=h-1$ \cite[Prop.\ 6.6]{WeymanZhao}.
If $t=h-1$ then the NCCR constructed in~\cite{WeymanZhao} coincides with ours
(see Remark \ref{ref-7.5.1-90} below). 
\subsubsection{Non-commutative resolutions for $\Sl_2$-invariants}
\begin{theorems}
  \label{ref-1.4.7-16} (See \S\ref{ref-8-92} below.)  Let $V$ be a
  vector space of dimension two and\footnote{We denote the group by
    $H$ here since in \S\ref{ref-8-92} below $G$ will stand for either
    $\PGL_2(k)$ or $\Sl_2(k)$ depending on whether all $(d_i)_i$ are
    even or not.} put $H=\Sl(V)$. Put $ W=\bigoplus_{i=1}^{d_i} S^{d_i}V $
  and assume in addition that $W$ is not a sum of $k^c$ and one of the
  following special representations
\begin{equation}
\label{ref-1.1-17}
V,S^2V,V\oplus V,V\oplus S^2V, S^2V\oplus S^2V, S^3V, S^4V.
\end{equation}
Put 
\[
s^{(n)}=\begin{cases}
n+(n-2)+\cdots+1=\dfrac{(n+1)^2}{4}&\text{if $n$ is odd}\\
n+(n-2)+\cdots+2=\dfrac{n(n+2)}{4}&\text{if $n$ is even}
\end{cases}
\]
and $s=\sum_i s^{(d_i)}$. Put $R=SW$.
If not all  $d_i$ are even then $R^H$ has a NCR given by 
\begin{equation}
\label{ref-1.2-18}
M=\bigoplus_{0\le i\le s/2-1} M(S^i V).
\end{equation}
It is an NCCR if $s$ is odd. If all $d_i$ are even (and hence $s$ is even) then~$R^H$ has a NCR given by
\begin{equation}
\label{ref-1.3-19}
M=\bigoplus_{0\le 2i\le s/2-1} M(S^{2i} V)
\end{equation}
 and it is an NCCR if $s/2$ is even. In the case that $s/2$ is odd   $R^H$ has a twisted NCCR of index $2$
given by
\begin{equation}
\label{ref-1.4-20}
M=\bigoplus_{0\le 2i+1\le s/2-1} M(S^{2i+1} V).
\end{equation}
\end{theorems}
\begin{remarks}
\begin{enumerate}
\item If $W$ is one of the special representations \eqref{ref-1.1-17} then it is classical that $R^H$ is a polynomial
ring.
\item Theorem \ref{ref-1.4.7-16} implies in particular that if all $d_i$ are even then $R^H$ always has a twisted NCCR.
\item For an application to ``trace rings'' see \S\ref{ref-1.4.5-21} below.
\end{enumerate}
\end{remarks}
\subsubsection{Trace rings}
\label{ref-1.4.5-21}
Let $n\ge 2$, $m\ge 2$ and let $V$ be a vector space of dimension $n$, $G=\PGL_n$
  and $W=\End(V)^{\oplus m}$. Put $R=SW$. Then $Z_{m,n}\overset{\text{def}}{=}R^G$
  is the so-called (commutative) \emph{trace ring} of $m$, $n\times
  n$-matrices.  The reason for this terminology is as follows.
One has $X=\Spec SW=\End(V)^{\oplus m}$. For $x\in X$ let
  $x_i\in \End(V)$ be the $i$'th component of $x$. Then a famous
  result, conjectured by Artin \cite{Ar3} and proved by Procesi \cite{Procesi} (see
  also \cite{WeymanDerksen,DomokosZubkov,Donkin,Raz,SchVdB3}), asserts that $Z_{m,n}$ is
  generated by the traces $\Tr(x_{i_1}\cdots x_{i_t})$ of products of the
  $x_i$ and moreover Procesi also proves that all relations between
  these traces are derivable from the Cayley-Hamilton identity. If $(m,n)=(2,2)$
then $Z_{m,n}$ is a polynomial ring in 5 variables given by the traces $\Tr(x_i)_{i}$, $\Tr(x_ix_j)_{i\ge j}$ \cite[p.20]{Herstein}. In
all other cases $Z_{m,n}$ is singular \cite[Prop.\ II.3.1]{LBP}.

The ring $Z_{m,n}$ is the center of the ``non-commutative'' trace ring
$\TT_{m,n}$ which is the module of covariants $M(\End(V))$. The ring
$\TT_{m,n}$ is a free algebra in a suitable category of algebras with
trace \cite{Procesi1} and hence, from a non-commutative geometry standpoint,
it may be regarded as an analogue of a polynomial ring. This
makes it interesting to understand the homological
properties of $\TT_{m,n}$.

It turns out that the homological properties of $\TT_{m,n}$ are a bit
better than those of $Z_{m,n}$, but not much. The ring $\TT_{m,n}$ has
finite global dimension if and only if $(m,n)=(2,2),(3,2),(2,3)$ (see
\cite{VdB15} for the if direction and \cite{LBP} for the only if direction).

The proofs in \cite{VdB15} that $\gldim \TT_{m,n}<\infty$ in the indicated cases are rather adhoc but using the methods in this paper one may give a more systematic analysis.
The case $(m,n)=(2,2)$ may be deduced from Theorem \ref{ref-1.5.1-24}. The case $(m,n)=(3,2)$ is a special case
of \eqref{ref-1.4-20}. In particular $\TT_{3,2}$ is a twisted NCCR of $Z_{3,2}$. It also follows from Theorem \ref{ref-1.4.7-16}
that for $m\ge 3$, $Z_{m,2}$ has a twisted NCCR of index two if $m$ is odd and a NCCR if $m$ is even. However this (twisted) NCCR
is not given by $\TT_{m,2}$ for $m>3$ since the latter has infinite global dimension.

In \S\ref{ref-9-93} below we will show that $\TT_{2,3}$ is a twisted NCCR of its center $Z_{2,3}$ using our methods. In particular we will recover that it has
finite global dimension. We will also show the following general result:
\begin{theorems} \label{ref-1.4.9-22}
 Assume $m\ge 2$, $n\ge 2$. Then $Z_{m,n}$ has a twisted NCCR.
\end{theorems}
Note that no ``nice'' commutative resolution of singularities  is known  for $Z_{m,n}$ and it seems very doubtful that it exists.
\subsection{Non-commutative resolutions for quotient singularities}

\label{ref-1.5-23}

Below we state a number of easy to verify sufficient conditions for the existence of NC(C)Rs (possibly twisted) for quotient singularities. 
Theorem \ref{ref-1.3.1-5} will be a corollary of these more refined results.
The criteria we state are special cases of a general combinatorial
method for verifying whether a module of covariants yields a NC(C)R (see Remark \ref{ref-11.3.2-114} below). 
For simplicity of exposition we will state our criteria for $G$ connected. In \S\ref{ref-4.5-62} below we explain how one may
handle the non-connected case.

\medskip

We first introduce some notations which will remain in force for the rest of the paper except when overruled locally.
Let $G$ be a reductive group and denote its identity connected component by $G_e$. Let
$T\subset B\subset G_e$ be respectively a maximal torus and a Borel
subgroup of $G_e$ with $\Wscr=N(T)/T$ being the corresponding Weyl group. Put
$X(T)=\Hom(T,G_m)$ and let $\Phi\subset X(T)$ be the roots of $G$.  
By convention
the roots of~$B$ are the negative roots $\Phi^-$ and $\Phi^+=\Phi-\Phi^-$ is the set
of  positive roots. 
We write $\bar{\rho}\in X(T)_\RR$ for half the sum of the positive roots.
 Let $X(T)^+_\RR$ be the dominant cone in $X(T)_\RR$ and let $X(T)^+=X(T)^+_\RR\cap X(T)$ be the set of dominant
weights. For $\chi\in X(T)^+$ we denote the simple
$G_e$-representation with highest weight~$\chi$ by~$V(\chi)$.

Let $W$ be a finite dimensional $G$-representation of dimension $d$ and put $R=SW$, $X=\Spec SW=W^\ast$.
Let 
$(\beta_i)_{i=1}^d\in X(T)$ be the $T$-weights of $W$. 

 Put
\begin{align*}
\Sigma&=\left\{\sum_i a_i \beta_i\mid a_i\in ]-1,0]\right\}\subset X(T)_\RR.
\end{align*}
\begin{theorems} \label{ref-1.5.1-24} (See \S\ref{ref-11.3-109} below.) Assume $G$ is connected and let $\Delta$ be a $\Wscr$-invariant bounded closed convex subset of $X(T)_\RR$. Let 
\begin{align*}
{\Lscr}&=X(T)^+\cap (-\bar{\rho}+\Sigma+\Delta),\\
U&=\bigoplus_{\chi\in \Lscr} V(\chi).
\end{align*}
Then one has 
$
\gldim M(\End(U)) <\infty
$.
\end{theorems}
Note that we may always take $\Delta$ in such a way that $0\in {\Lscr}$ (e.g. let $\Delta$ be the convex
hull of $\Wscr\cdot \bar{\rho}$). In that case $U\neq 0$.
We obtain as in Corollary \ref{ref-1.3.5-9}.
\begin{corollarys}
\label{ref-1.5.2-25}
Assume $G$ is connected.
If $W$ is generic then for $\Delta$ such that $U\neq 0$ in 
Theorem~\ref{ref-1.5.1-24} one has that $M(U)$ yields a NCR of $R^G$.
\end{corollarys}
\subsection{Non-commutative crepant resolutions for quotient
  singularities}
\label{ref-1.6-26}
Let the notations be as in the previous section.  We will say that $W$
is \emph{quasi-symmetric} if for every line $\ell\subset X(T)_\RR$ through
the origin we have 
\[
\sum_{\beta_i\in\ell}\beta_i=0.
\]
This implies in particular that $W$ is unimodular (i.e.\ $\wedge^d W\cong k$)
and hence $R^G$ is Gorenstein if $W$ is generic by a result of Knop (see Theorem \ref{ref-4.1.7-41} below).

 The following result strengthens 
Theorem \ref{ref-1.5.1-24} in the quasi-symmetric case.
\begin{theorems} 
\label{ref-1.6.1-27} (See \S\ref{ref-12.1-123} below.) Let $G$ be connected and assume $W$ is quasi-symmetric. Let $\Delta$ be a $\Wscr$-invariant bounded closed convex subset of $X(T)_\RR$.

Put
\begin{align*}
{\Lscr}&=X(T)^+\cap (-\bar{\rho}+(1/2)\bar{\Sigma}+\Delta),\\
U&=\bigoplus_{\chi\in {\Lscr}} V(\chi).
\end{align*}
 Then one has
$
\gldim M(\End(U))<\infty
$. 
\end{theorems}
If $W$ is generic and ${\Lscr}\neq \emptyset$ then this yields again a NCR as in Corollary \ref{ref-1.5.2-25}. However our main
concern in the quasi-symmetric case will be the construction of NCCRs rather than just NCRs.
For this we need the concept of a \emph{half open polygon} which generalizes the notion of a half open
interval. Let $\Delta\subset\RR^n$ be a bounded closed convex polygon. For $\varepsilon\in \RR^n$ parallel to the linear space spanned by $\Delta$ put
\[
\begin{aligned}
\Delta_\varepsilon&=\bigcup_{r>0} \Delta\cap (r\varepsilon+\Delta),\\
\Delta_{\pm\varepsilon}&=\Delta_\varepsilon\cap \Delta_{-\varepsilon}.
\end{aligned}
\]
So $\Delta_\varepsilon$ is obtained from $\Delta$ by removing the boundary faces which are
moved inwards by $\varepsilon$ and $\Delta_{\pm\varepsilon}$ is obtained from $\Delta$ by removing the boundary faces not parallel to~$\varepsilon$.

We will say that
 $\varepsilon\in \RR^n$ is \emph{generic} for $\Delta$ if it is a
non-zero vector which is parallel to  $\Delta$, but not parallel
to any of its boundary faces. In that case $\Delta_{\pm \varepsilon}$ is the relative interior of $\Delta$.

Our main result concerning NCCRs is Theorem \ref{ref-1.6.4-30} below, but we will first 
state it in torus case.
\begin{theorems} (See \S\ref{ref-12.4-129} below.) 
\label{ref-1.6.2-28} 
Assume that $G=T$ is a torus and that $W$ is quasi-symmetric and generic. 
Fix any $\varepsilon\in X(R)_\RR$ which is
generic for $\bar{\Sigma}$ and put
\begin{align*}
{\Lscr}&=X(T)\cap (1/2)\bar{\Sigma}_{\varepsilon},\\
U&=\bigoplus_{\chi\in {\Lscr}}V(\chi).
\end{align*}
Put $M=M(U)$. Then
$
\End_{R^T}(M)
$
is a NCCR for $R^T$.
\end{theorems}
In \S\ref{ref-10.1-99} below we give an example of a
non-quasi-symmetric $W$ such that $R^T$ has no NCCR given by a module
of covariants. On the other hand if $R^T$ is Gorenstein and $\dim
\Spec R^T=3$ then a NCCR given by a module of covariants always
exists~\cite{Broomhead}.

\medskip

It turns out that in order to make the proof of Theorem \ref{ref-1.6.2-28}
work for more general connected reductive groups $\varepsilon$ needs to be generic \emph{and}
$\Wscr$-equivariant. Unfortunately it is usually not possible to
satisfy both conditions at once. Indeed if~$G$ is for example
semi-simple then $\Wscr$-invariance enforces $\varepsilon=0$ which is
in some sense the opposite to generic.  In \S\ref{ref-10.2-100} we will give an example of a
quotient singularity for a quasi-symmetric representation which does
not have a NCCR. See however \S\ref{ref-5-68} for a positive
example.

\medskip

In
order to state a more generally applicable version of Theorem \ref{ref-1.6.2-28} it will be convenient to expand our setting slightly.
So we assume that in addition to the connected $G$ that there is a surjective morphism
$\phi:\bar{G}\r G$ where $\bar{G}$ is a connected reductive group with
$\dim\bar{G}=\dim G$. Then $A\overset{\text{def}}{=}\ker\phi$ is a finite subgroup of the center of $\bar{G}$. Let $\bar{T}\subset \bar{G}$ be the inverse image of $T$ in $G$. 
Then~$\bar{T}$ is still a maximal torus and $A\subset \bar{T}$.
We have an exact sequence
\[
0\r X(T)\r X(\bar{T})\r X(A)\r 0
\]
and hence a corresponding coset decomposition of lattices inside $X(T)_\RR=X(\bar{T})_\RR$
\[
X(\bar{T})=\bigcup_{\bar{\mu}\in X(A)} X(T)_{\bar{\mu}},
\]
where $X(T)_{\bar{\mu}}=X(T)+\mu$. We set $X(T)^+_{\bar{\mu}}=X(T)^+_\RR\cap X(T)_{\bar{\mu}}$ and define \label{index}
\[{}
i_{\bar{\mu}}(\bar{G})=\gcd_{\chi\in X(T)^+_{\bar{\mu}}} \dim V(\chi).
\]
Clearly $i_{\bar{0}}(\bar{G})=1$. Our first result is a strengthening of Theorem \ref{ref-1.6.1-27}.
\begin{theorems}
\label{ref-1.6.3-29} (See \S\ref{ref-12.2-127} below.) Let $G$ be connected and let the other notations be as above. Assume~$W$ is quasi-symmetric. Let $\varepsilon\in X(T)_\RR$ be $\Wscr$-invariant
and $\bar{\mu}\in A$.
Put
\begin{align*}
{\Lscr}&=X(T)_{\bar{\mu}}^+\cap (-\bar{\rho}+(1/2)\bar{\Sigma}_\varepsilon),\\
U&=\bigoplus_{\chi\in {\Lscr}} V(\chi),\\
\Lambda&= M(\End(U)).
\end{align*}
 Then one has
$
\gldim \Lambda<\infty
$. If $W$ is in addition generic and ${\Lscr}\neq \emptyset$ then  $\Lambda$ is  a twisted NCR for $R^G$ of index $i_{\bar{\mu}}(\bar{G})$. 
\end{theorems}
We now give the criterion for the existence of (twisted) NCCRs we will use.
\begin{theorems}
\label{ref-1.6.4-30} 
(See \S\ref{ref-12.3-128} below.) Let $G$ be connected. Assume that $W$ is quasi-symmetric and generic.
Assume  that $\varepsilon\in X(T)_\RR$ is $\Wscr$-invariant and let $\bar{\mu}$ be such that
\begin{equation}
\label{ref-1.5-31}
X(T)_{\bar{\mu}}\cap (-\bar{\rho}+(1/2)(\bar{\Sigma}_{\pm\varepsilon}-\Sigma))=\emptyset.
\end{equation}
Put
\begin{align*}
{\Lscr}&=
X(T)^+_{\bar{\mu}}\cap 
(-\bar{\rho}+(1/2)\bar{\Sigma}_\varepsilon),\\
U&=\bigoplus_{\chi\in {\Lscr}}V(\chi),\\
\Lambda&=M(\End(U)).
\end{align*}
If ${\Lscr}\neq \emptyset$ then $\Lambda$
 is a twisted NCCR of index $i_{\bar{\mu}}(\bar{G})$ for $R^G$. 
\end{theorems}

\begin{remarks}
In Example \ref{ref-10.3-101} below we give an example of a pair $(G,W)$ such that $R^G$ has a twisted NCCR but it has no  NCCR.
\end{remarks}

\section{Acknowledgement}
The authors thank Roland Abuaf, Michel Brion, Hailong Dao, Johan de Jong, Craig Huneke, Jean Michel, Michael
Wemyss and Ga\v sper Zadnik for interesting discussions. The first
author also thanks the University of Hasselt for its hospitality. 
In addition, the authors thank the referee for his careful reading 
of the manuscript and his helpful comments.

\section{Non-commutative resolutions}
\label{ref-3-32}
Let $S$ be a normal noetherian domain with quotient field $K$. A finitely generated~$S$-module $M$ is said to be \emph{reflexive}
if the canonical map $M\mapsto M^{\ast\ast}$ is an isomorphism. This implies in particular that $M$ is torsion free.
Reflexive modules are not affected by codimension two phenomena. For example a morphism $\phi:M\r N$ between reflexive modules is an isomorphism if this is the case for all $\phi_P:M_P\r N_P$
where $P$ runs through the height one primes in~$S$.

The category $\Ref(S)$ of reflexive $S$-modules is a rigid symmetric
monoidal category  with the tensor product being given by $(M,N)\mapsto (M\otimes_S N)^{\ast\ast}$.  This implies that many
concepts for $S$-modules and $S$-algebras have a natural ``reflexive'' analogue. For example a
\emph{reflexive Azumaya algebra} \cite{Leb99} $A$ is a non-zero $S$-algebra $A$
which is a reflexive $S$-module such that the natural map $ A\otimes_S
A^\circ\r \End_S(A):a\otimes b\mapsto (x\mapsto axb) $ becomes an
isomorphism after applying $(-)^{\ast\ast}$. Such
reflexive notions will be used without further comment below. A reflexive Azumaya algebra~$A$ is said to be trivial if it is of the form $\End_S(M)$ for $M$ a
reflexive $S$-module. In that case $\Ref(S)$ and $\Ref(A)$ are
equivalent. This is a particular case of ``reflexive Morita
equivalence'' which is defined in the obvious way.

The \emph{index} $\ind(A)$ of a reflexive Azumaya algebra $A$ is the index of the central
simple $K$-algebra $K\otimes_S A$. If $\ind(A)=1$ then it is easy to see that $A$ is trivial.
\begin{definition} A \emph{twisted non-commutative resolution} of $S$ of index $m$ is a reflexive Azumaya
algebra $A$ of index $m$ over $S$ such that $\gldim A<\infty$. If $m=1$ then 
$A$ is said to be a \emph{non-commutative resolution (NCR)} of $S$.
\end{definition}
\begin{definition} Assume that $S$ is Gorenstein
A \emph{twisted non-commutative crepant resolution} $A$ of $S$ of index $m$ 
is a twisted NCR of $S$ of index $m$ which is in addition a Cohen-Macaulay $S$-module.
If $m=1$ then
such $A$ is said to be a \emph{non-commutative crepant resolution (NCCR)} of $S$.
\end{definition}
The notion of a twisted NC(C)R is obviously local. We will use it in the following sense.
\begin{proposition} \label{ref-3.3-33} Assume that $S$ is the coordinate ring of a normal affine algebraic
variety over $k$. For $m$ a maximal ideal in $S$ let $\hat{S}_m$ be the completion
at $m$. Let $A$ be an $S$-algebra which is finitely generated as an $S$-module. Then
$A$ is a twisted NC(C)R for $S$ if and only if for each maximal ideal in $S$ one has that $\hat{A}_m=\hat{S}_m\otimes_S A$ is
a twisted NC(C)R for~$\hat{S}_m$.
\end{proposition}
Note that the notion of a NC(C)R is not fully 
local in the sense of Proposition \ref{ref-3.3-33}. The index may go down under completion. 

The following result which is proved using basic homological algebra
is of similar nature.
\begin{proposition}
\label{ref-3.4-34} Assume that $S$ is a finitely generated commutative normal
  $\NN$-graded algebra which is connected (i.e.\ $S=k\oplus S_1\oplus
  S_2\oplus\cdots$) and let $m=S_1\oplus S_2\oplus\cdots$ be the
  augmentation ideal. Let $A$ be a graded $S$-algebra which is
  finitely generated as an $S$-module. Then $\gldim \hat{A}_m=\gldim A$ for
$\hat{A}_m=\hat{S}_m\otimes_S A$.
Moreover
$A$ is a twisted NC(C)R
  for $S$ if and only if $\hat{A}_m$ is a
  twisted NC(C)R for~$\hat{S}_m$.
\end{proposition}
A highly useful criterion for recognizing NCCRs that does not refer to
finite global dimension has been given by Iyama and Wemyss
\cite{IW,Wemyss1}. 
If $B$ is a ring and $M$ is a left
$B$-module then we write $\add M$ for the category of left $B$
modules which are direct summands of some $M^{\oplus n}$.
\begin{proposition} \label{ref-3.5-35} \cite[Ex.\ 4.34]{Wemyss1}
  Assume that $S$ is a local normal
  Gorenstein
 ring. Let $M$ be a reflexive $S$-module such that
  $\Lambda=\End_S(M)$ is a Cohen-Macaulay $S$-module. Then $\Lambda$
  is a NCCR over $S$ if and only if for every reflexive $S$-module
  which has the property that $\Hom_S(M,N)$ is a Cohen-Macaulay
  $S$-module, we have that $N\in \add M$.
\end{proposition}
Of course this result has a graded version for a connected graded ring $S=k\oplus S_1\oplus S_2\oplus \cdots$. We will use this without further comment.
\begin{remark} 
\label{rem:nccropt} It easily follows from Proposition \ref{ref-3.5-35} that NCCRs
  are optimal among the NCRs in the following sense: if $\Lambda=\End_S(M)$ is a NCCR as in
  the statement of the proposition and $M'$ is a reflexive $S$-module
  such that $\add(M')\subsetneq \add(M)$ then $\gldim \End_S(M')=\infty$. On the other hand if 
$\add(M)\subsetneq \add(M')$ then $\End_S(M')$ is not Cohen-Macaulay. Informally: NCCRs
cannot be enlarged or shrunk without losing some of their good properties.
\end{remark}

\section{Modules of covariants}

\subsection{Preliminaries}
Let $G$ be a reductive group 
and let $R$ be the coordinate ring of a smooth connected affine $G$-variety $X$ with function field $K$. 
Let $\mod(G,R)$ be the category of $G$-equivariant finitely generated $R$-modules. The following is standard.
\begin{lemmas}\label{ref-4.1.1-36} \begin{enumerate}
\item The objects $U\otimes R$ with $U\in \hat{G}$ form a family of projective generators for $\mod(G,R)$.
\item A projective object in $\mod(G,R)$ is projective as an $R$-module.
\item If $P\in \mod(G,R)$ is projective as an $R$-module then it is a projective object in $\mod(G,R)$.
\item $\gldim \mod(G,R)\le \gldim R$.
\item If $X$ has a fixed point then  $\gldim \mod(G,R)= \gldim R$.
\end{enumerate}
\end{lemmas}
\begin{proof} 
  The first statement is clear. The second statement is true because it is true for
the generators $U\otimes R$.
 The fourth statement follows from
  the first three combined with the fact that $R$ has finite
  global dimension.  For the third statement choose a $G$-equivariant
  surjection $\phi:U\otimes R\r P$ with $U$ a representation of $G$ and an $R$-linear splitting
  $\theta:P\r U\otimes R$ for $\phi$.  Applying the Reynolds operator
  $R$ to the identity $\phi\theta=\Id$ yields $\phi
  R(\theta)=R(\phi\theta)=R(\id)=\id$. Hence $R(\theta)$ is a
  $G$-equivariant splitting for $\phi$ and thus $P$ is projective in
  $\mod(G,R)$. Finally to prove (5) let $x$ be the fixed point. Using (2) we see that a
  $G$-equivariant projective resolution of $k(x)$ must have at least
  length $\dim X$ since this is true if we forget the $G$-action. This
  proves $\gldim \mod(G,R)\ge \dim X=\gldim R$.
\end{proof}
Now let $\bar{G}\r G$ be a central extension of $G$ with kernel $A$ where $\bar{G}$ is also reductive. Put $X(A)=\Hom(A,G_m)$. 
For $\chi\in X(A)$ let $\mod(\bar{G},R)_{\chi}$ be the abelian category of $\bar{G}$-equivariant finitely generated $R$-modules on which $A$
acts through the character $\chi$. Note $\mod(\bar{G},R)_0=\mod(G,R)$ and furthermore $\mod(\bar{G},R)=\bigoplus_{\chi\in X(A)}\mod(\bar{G},R)_\chi$. Clearly Lemma \ref{ref-4.1.1-36} extends
to $\mod(\bar{G},R)_\chi$. We let $i_\chi(\bar{G})$ be the greatest common divisor of the dimensions of the objects in $\mod(\bar{G},k)_\chi$. 
Obviously $i_0(\bar{G})=1$ and $i_{\chi}(\bar{G})=i_{\chi^{-1}}(\bar{G})$.

\begin{lemmas} 
\label{ref-4.1.2-37}  Assume that the $G$-action on $X$ has a fixed point. For every $M\in \mod(\bar{G},R)_\chi$ one has that $\rk M\overset{\text{def}}{=}\dim_K(K\otimes_R M)$ is divisible
by $i_\chi(\bar{G})$.
\end{lemmas}
\begin{proof} By taking the fiber in the fixed point we see that the claim is true for projectives. The general case follows by replacing
$M$ with a projective resolution in $\mod(G,R)_\chi$.
\end{proof}
\begin{lemmas} 
\label{ref-4.1.3-38} Assume that $G$ acts generically on $X$ (see \S\ref{ref-1.3-4}).
Let $\Ref(G,R)$ be the category of $G$-equivariant $R$-modules which are reflexive as $R$-modules. Then
the  functors
\begin{gather*}
\Ref(G,R)\mapsto \Ref(R^G):M\mapsto M^G,\\
\Ref(R^G)\mapsto \Ref(G,R):N\mapsto (R\otimes_{R^G}N)^{\ast\ast}
\end{gather*}
are inverse equivalences between the symmetric monoidal categories $\Ref(G,R)$ and $\Ref(R^G)$.
\end{lemmas}
\begin{proof}
The hypotheses imply that $X\r X\quot G$ contracts no divisor. This
implies that $(-)^G$ preserves reflexive modules by  \cite[Prop.\ 1.3]{Brion}.

For the remaining part of the lemma we have to show that for $M\in \Ref(G,R)$
the map $(R\otimes_{R^G} M^G)^{\ast\ast}\r M$ is an isomorphism, or equivalently
that $\phi:R\otimes_{R^G} M^G\r M$ is an isomorphism outside a closed subset of $X$ of codimension two. 
We will show that $\phi$ is an isomorphism in a neighborhood of any point $x$ of $X^{\mathsf{s}}$ (see Definition \ref{ref-1.3.4-8}).

Up to restricting to a suitable $G$-invariant affine \'etale neighborhood of $x$ 
we may assume by  the Luna slice theorem \cite{Luna} that $X=G\times S$. But then the result is clear by descent. 
\end{proof}

Let $U$ be finite dimensional
$G$-representation. Then  
\[
M(U)\overset{\text{def}}{=} (U\otimes R)^G
\]
is a finitely generated $R^G$-module which is called the \emph{module of covariants} associated to $U$. Sometimes we use additional decorations such as $M_G(U)$, $M_{G,R}(U)$, $M_{G,X}(U)$ to indicate context.

\begin{corollarys}
\label{ref-4.1.4-39} Assume that $G$ acts generically on $X$. Then modules of covariants are reflexive and moreover
\[
\mod(G,k)\mapsto \Ref(R^G):U\mapsto M(U)
\]
is a symmetric monoidal functor. 
\end{corollarys}

\begin{lemmas} Assume that $G$ acts generically on $X$ and that
  $X$ contains a fixed point. If $M(U)\cong M(U')$ as
  $R^G$-modules then $U\cong U'$ as $G$-representations.
\end{lemmas}
\begin{proof} By Lemma \ref{ref-4.1.3-38} $U\otimes R\cong U'\otimes R$ in $\Ref(G,R)$. Specializing at the fixed point yields what we want. 
\end{proof}
Since $\End(U)$ is a $G$-equivariant $k$-algebra we obtain that $M(\End(U))$ is an $R^G$-algebra. We will call it an \emph{algebra of covariants}.  From Lemma \ref{ref-4.1.3-38} we obtain that if $G$ acts generically then
\[
M(\End(U))=\End_{R^G}(M(U)).
\]
Now assume that $U\in \mod(\bar{G},k)_\chi$. If $\chi\neq 0$ then $(U\otimes R)^{\bar{G}}=0$. However $A$ acts trivially on $\End(U)$ so the algebra of covariants $M(\End(U))$ is still 
interesting. We will  use the following result.  
\begin{propositions}
\label{ref-4.1.6-40} Assume that $G$ acts generically on $X$ and $0\neq U\in \mod(\bar{G},k)_\chi$. Then $\Lambda=M(\End(U))$ is a reflexive Azumaya algebra over $R^G$ and the 
reflexive Morita equivalence class of $\Lambda$ depends only on $\chi$. In addition one has $\ind(\Lambda)\mid i_\chi(\bar{G})$.
Moreover if $X$ contains a fixed point
then $\ind(\Lambda)=i_\chi(\bar{G})$. 
\end{propositions}
\begin{proof} The fact that $\Lambda$ is reflexive Azumaya follows immediately from Lemma \ref{ref-4.1.3-38}. Similarly let $U'$ be another object in $\mod(\bar{G},k)_\chi$. 
Then the bimodules $\Hom(U,U')\otimes R$ and $\Hom(U',U)\otimes R$ define a $G$-equivariant Morita context between $\End(U)\otimes R$ and $\End(U')\otimes R$.
Using Lemma \ref{ref-4.1.3-38} this descends to a reflexive Morita context between $M(\End(U))$ and $M(\End(U'))$. We deduce in particular $\ind(\Lambda)\mid \dim U'$ for any $U'$.
Hence $\ind(\Lambda)\mid i_\chi(\bar{G})$.

For the next claim assume that the $G$-action on $X$ has a fixed point. Let
$M$ be non-zero reflexive left $\Lambda$-module of minimal rank. Put
$\Gamma=\End_\Lambda(M)$.  Then $\Gamma$ is a reflexive Azumaya algebra over $R^G$
which is reflexive Morita equivalent to~$\Lambda$. Put
$\tilde{M}=(R\otimes_{R^G} M)^{\ast\ast}$, $\tilde{\Gamma}=(R\otimes_{R^G}
\Gamma)^{\ast\ast}$, $\tilde{\Lambda}=\End(U)\otimes
R$. Then by Lemma \ref{ref-4.1.3-38} $\tilde{M}$ is a
$G$-equivariant reflexive $\tilde{\Lambda}$-module and $\End_{\tilde{\Lambda}}(\tilde{M})=\tilde{\Gamma}$. Put $N=\Hom_{\tilde{\Lambda}}(U\otimes R,\tilde{M})$. Then
$N$ is a reflexive object in $\mod(\bar{G},R)_{\chi^{-1}}$ and moreover by (reflexive) Morita theory the $\tilde{\Gamma}$-action on $\tilde{M}$
induces a $G$-equivariant isomorphism $\tilde{\Gamma}\r \End_R(N)$.
By Lemma \ref{ref-4.1.2-37} we know
that $\rk N$ is divisible by $i_{\chi^{-1}}(\bar{G})$. Hence $i_{\chi^{-1}}(\bar{G})^2\mid \rk_{R}(\tilde{\Gamma})=\rk_{R^G}(\Gamma)$. It follows that $i_{\chi^{-1}}(\bar{G})\mid \ind(\Gamma)=\ind(\Lambda)$.
Finally we use $i_{\chi^{-1}}(\bar{G})=i_\chi(\bar{G})$.
\end{proof}
For further reference recall the following result.
\begin{theorems} \cite{Knop3} \label{ref-4.1.7-41} Assume that $W$ is a generic unimodular ($\det W\cong k$) $G$-representation and $R=SW$. Then $R^G$ is Gorenstein.
\end{theorems}
\subsection{Modules of covariants and the Luna slice theorem}
\label{ref-4.2-42}
Let $G$ be a reductive group and let $X$ be a smooth affine $G$-variety. Recall that the inverse image of every point in $X\quot G$ contains
a unique closed orbit.
This may be used to analyze questions
that are local on $X\quot G$. The facts we recall in this section will
 be used in \S\ref{ref-4.3-47} and Remark \ref{ref-6.1.4-84} below.

Let $x\in X$ be a point with closed orbit and let $G_x\subset G$ be the stabilizer
of $x$. Then $G_x$ is a reductive subgroup of $G$. Choose a $G_x$-invariant complement~$N_x$ to the inclusion 
of $G_x$-representations $T_x(G)/T_x(G_x)\subset T_x(X)$. 
For a finite dimensional $G$-representation~$U$ put
$M_{G,X}(U)=(U\otimes k[X])^G$. Write $M_{G,X}(U)\,\hat{}_{\bar{x}}$ for the completion of $M_{G,X}(U)$ at the image of $x$ in $X\quot G$. The following
well-known lemma is a direct consequence of the Luna slice theorem (see e.g.\ \cite{LBP}). Since it is not explicitly stated in loc.\ cit.\ 
we include the short proof for the benefit of the reader.
\begin{lemmas} 
\label{ref-4.2.1-43} Let $\bar{x}$ be the image of $x$ in $X\quot G$ and similarly for $\bar{0}\in N_x\quot G_x$. One has compatible isomorphisms
\begin{equation}
\label{ref-4.1-44}
\begin{aligned}
k[X\quot G]\,\hat{}_{\bar{x}}&\cong k[N_x\quot G_x]\,\hat{}_{\bar{0}}&& \text{as $k$-algebras,}\\
M_{G,X}(U)\,\hat{}_{\bar{x}}&\cong M_{G_x,N_x}(U)\,\hat{}_{\bar{0}} && \text{as modules over these algebras.}
\end{aligned}
\end{equation}
\end{lemmas}
\begin{proof}
The Luna slice theorem \cite{Luna} (see also \cite[App.\ D to Ch.\ 1]{Mumford})
asserts the existence of an affine  $G_x$-invariant ``slice'' $S$ to the $G$-orbit of $x$ and a
$G_x$-invariant \'etale map $S\r N_x$ sending $x$ to $0$ such that there are Cartesian diagrams
\[
\xymatrix{
N_x\ar[d]&&S\ar[d]\ar[ll]_{\text{\'etale}}\\
N_x\quot G_x&& S\quot G_x\ar[ll]^{\text{\'etale}} 
}
\qquad \xymatrix{
G\times_{G_x} S\ar[d]\ar[rr]^{\text{\'etale}}&& X\ar[d]\\
S\quot G_x \ar[rr]_{\text{\'etale}} && X\quot G
}
\]
We find
\[
M_{G,G\times_{G_x}S}(U)=k[S\quot G_x]\otimes_{k[X\quot G]}M_{G,X}(U)
\]
and using descent $M_{G,G\times_{G_x}S}(U)=M_{G_x,S}(U)$. Finally we also have
\[
M_{G_x,S}(U)=k[S\quot G_x]\otimes_{k[N_x\quot G_x]}M_{G_x,N_x}(U).
\]
To finish one uses  the fact that \'etale morphisms induce isomorphisms on completions.
\end{proof}
For further reference we note that the Luna slice theorem is particularly easy to apply
if $X$ is a representation. In that case there is a natural embedding $x+N_x\subset X$ and  the morphisms in \eqref{ref-4.1-44} may be taken to be the restriction
morphisms. Furthermore
let $\mathfrak{g}=\Lie(G)$, $\mathfrak{g}_x=\Lie(G_x)$.
Then we have
\begin{align}
\label{ref-4.2-45}
\frak{g}_x&=\{v\in \frak{g}\mid v\cdot x=0\},\\
\label{ref-4.3-46}
T_x(X)&=X=\mathfrak{g}/\mathfrak{g}_x\oplus N_x.
\end{align}
Since $G_x$ is reductive we see that the $G_x$-representation $N_x$ is uniquely determined by $X$ and $\mathfrak{g}/\mathfrak{g}_x$.

\subsection{Restricting to locally closed embeddings}
\label{ref-4.3-47}
The Luna slice theorem allows us to reduce certain questions to the linear case. Here we discuss an example of  such a reduction.
We use similar notations as in \S\ref{ref-4.2-42}.
\begin{theorems}
\label{ref-4.3.1-48}
Let $G$ be a reductive group and let $Y\hookrightarrow X$ be a $G$-equivariant locally closed embedding of smooth affine
$G$-varieties, such that closed orbits in $Y$ remain closed in $X$.
Let $U$ be a finite dimensional $G$-repre\-sentation. Then we have
\begin{equation}
\label{ref-4.4-49}
\gldim M_{G,Y}(\End(U))\le \gldim M_{G,X}(\End(U)).
\end{equation}
\end{theorems}
\begin{proof}
We
have
\begin{equation}
\label{ref-4.5-50}
g:=\gldim M_{G,X}(\End(U))=\max_{x\in X,Gx\text{ closed }}  \gldim M_{G,X}(\End(U))\,\hat{}_{\bar{x}}
\end{equation}
and a similar identity for $\gldim M_{G,Y}(\End(U))$.
Assume $y\in Y$ has closed orbit in~$Y$ and hence in $X$. We will show  that
\begin{equation}
\label{ref-4.6-51}
\gldim M_{G,Y}(\End(U))\,\hat{}_{\bar{y}}\le g,
\end{equation}
which then implies \eqref{ref-4.4-49}.

Choose $G_y$-equivariant splittings $T_yX=T_y Y\oplus N$,
$T_yY=N_{y,Y}\oplus T_y(Gy)$. By Lemma \ref{ref-4.2.1-43} we have
\begin{align}
\label{ref-4.7-52}
 M_{G,Y}(\End(U))\,\hat{}_{\bar{y}}&\cong M_{G_y,N_{y,Y}}(\End(U))\,\hat{}_{\bar{0}}\\
\label{ref-4.8-53}
 M_{G,X}(\End(U))\,\hat{}_{\bar{y}}&\cong M_{G_y,N_{y,Y}\oplus N}(\End(U))\,\hat{}_{\bar{0}}
\end{align}
Giving $k[N_{y,Y}\oplus N]$ the standard connected grading by putting $N_{y,Y}\oplus N$ in degree one we obtain by Proposition \ref{ref-3.4-34}, \eqref{ref-4.8-53} and \eqref{ref-4.5-50}.
\begin{equation}
\label{ref-4.9-54}
\gldim  M_{G_y,N_{y,Y}\oplus N}(\End(U))=\gldim M_{G_y,N_{y,Y}\oplus N}(\End(U))\,\hat{}_{\bar{0}}\le g
\end{equation}
We now give $k[N_{y,Y}\oplus N]$ a different $\NN$-grading by putting
$N_{y,Y}$, $N$ respectively in degrees $0,1$. By Lemma
\ref{ref-4.3.2-56} below and  Proposition \ref{ref-3.4-34} 
(for the standard grading) we find
\begin{equation}
\label{ref-4.10-55}
\gldim M_{G_y,N_{y,Y}}(\End(U))\,\hat{}_{\bar{0}}=\gldim  M_{G_y,N_{y,Y}}(\End(U))\le \gldim  M_{G_y,N_{y,Y}\oplus N}(\End(U))
\end{equation}
Combining \eqref{ref-4.10-55} with \eqref{ref-4.9-54}
and \eqref{ref-4.7-52} yields \eqref{ref-4.6-51}. 
\end{proof}
\begin{lemmas}
\label{ref-4.3.2-56}
Assume that $\Lambda=\Lambda_0+\Lambda_1+\cdots$ is a $\NN$-graded ring. Then
\[
\gldim \Lambda_0\le \gldim \Lambda.
\]
\end{lemmas}
\begin{proof} Let $M$ be a $\Lambda_0$-module and consider a graded projective resolution
$P^\bullet$ of $\Lambda\otimes_{\Lambda_0}M$. Restricting this resolution to degree zero yields
a projective resolution of $M$, finishing the proof. 
\end{proof}

\subsection{Cohen-Macaulayness of modules of covariants}
\label{ref-4.4-57}
We let the notations be as in \S\ref{ref-1.5-23} in the introduction. 
We will be interested in sufficient criteria for $M(U)$ to be Cohen-Macaulay for $U$ a finite dimensional representation
of $G$. If $M_{G_e}(U)$ is Cohen-Macaulay
then so is $M_G(U)$ so from now on we will restrict ourselves to the case that $G=G_e$ is connected.

A relevant conjecture in the connected case was stated in \cite{St} and this conjecture was almost completely
proved in \cite{VdB3}. Below we use those results to obtain  easy to verify criteria
for Cohen-Macaulayness in the cases that interests us. 
\begin{definitions}
\label{ref-4.4.1-58}
The elements of the intersection 
\[
X(T)^+\cap(-2\bar{\rho}+\Sigma)
\] 
are called \emph{strongly critical (dominant) weights} for $G$.
\end{definitions}
\begin{lemmas} \label{ref-4.4.2-59}
Assume that $\chi$ is strongly critical for $G$ and let the $T$-weights of $U=V(\chi)$ be given by $(\chi_i)_i$. Then for any $S\subset \Phi$
and for any $i$ we have that $\chi_i+\sum_{\rho\in S}\rho$ is strongly critical for $T$.
\end{lemmas}
\begin{proof} Let $\Gamma$ be the convex polygon
\[
\Gamma=\{\sum_{\rho\in \Phi} u_\rho\rho\mid u_\rho\in [0,1]\}.
\]
For every $S\subset\Phi$ we have $\sum_{\rho\in S}\rho\in \Gamma$ and moreover by Corollary \ref{ref-B.3-144} below $\Gamma$ is the convex hull of the $\Wscr$-orbit of $2\bar{\rho}$.

Similarly all $\chi_i$ are contained in the convex hull of the $\Wscr$-orbit of $\chi$ by \cite[Thm 14.18]{FH}. Hence we have to prove that for all $v,w\in \Wscr$ one has $v\chi+w(2\bar{\rho})\in \Sigma$.
This follows from Lemma \ref{ref-D.1-148} below since $\Sigma$ is convex and $\Wscr$-invariant,  $2\bar{\rho}$ and $\chi$ are dominant and finally by hypothesis $\chi+2\bar{\rho}\in \Sigma$.
\end{proof}
Recall that a stable point is a point with closed orbit and finite stabilizer.
\begin{theorems} \label{ref-4.4.3-60}
Assume $X$ contains a stable  point. Let $\chi\in X(T)^+$ be a strongly critical weight and $U=V(\chi)$. Then 
$M(U^\ast)$ is a Cohen-Macaulay $R^G$-module.
\end{theorems}
\begin{proof} Let $R_{U}$ be the isotypical component of $R$
  corresponding to $U$. I.e.\ $R_{U}$ is the sum of all
  subrepresentations of $R$ isomorphic to $U$. One has
  $R_{U}=U\otimes M(U^\ast)$. Hence $M(U^\ast)$ is Cohen-Macaulay if
  and only if $R_{U}$ is Cohen-Macaulay.

Let $(\chi_i)_i$ be the $T$-weights of $U$. 
According to \cite[Thm 1.3]{VdB3} $R_{U}$ will be Cohen-Macaulay if for every $i$
and for every $S\subset \Phi$ one has that $\chi_i+\sum_{\rho\in S}\rho\in \Sigma$, or
equivalently if $\chi_i+\sum_{\rho\in S}\rho$ is strongly critical for $T$. This 
condition holds by Lemma \ref{ref-4.4.2-59}.
\end{proof}
\begin{propositions}
\label{ref-4.4.4-61}
Assume that $X$ contains a stable point. 
Let $\chi_1,\chi_2\in X(T)^+$. If $\chi_1+\chi_2$ is strongly critical then $M((V(\chi_1)\otimes V(\chi_2))^\ast)$ is Cohen-Macaualay.
\end{propositions}
\begin{proof} Assume that $V(\chi)$ with $\chi\in X(T)^+$ is a summand of $V(\chi_1)\otimes V(\chi_2)$. Then by 
\cite[Ex.\ 25.33]{FH}
 $\chi=\chi_1+\chi'$ with $\chi'$ a weight
of $V(\chi_2)$. Hence by \cite[Thm 14.18]{FH} $\chi'$ is in the convex hull of the $\Wscr$-orbit of $\chi_2$. Hence by Theorem \ref{ref-4.4.3-60} we have to prove that $\chi_1+w\chi_2+2\bar{\rho}\in \Sigma$ for $w\in \Wscr$.
This follows again from Lemma \ref{ref-D.1-148} below since $\chi_1+2\bar{\rho}$ and $\chi_2$ are dominant and  by hypothesis $\chi_1+\chi_2+2\bar{\rho}\in \Sigma$.
\end{proof}
\subsection{Non-connected groups}
\label{ref-4.5-62}
Assume that $G$ is a reductive group and $H$ is a normal
subgroup of finite index.
We assume that $G$ acts on a smooth affine variety
with coordinate ring $R$.

If $g\in G$ then we write $\sigma_g=g\cdot g^{-1}\in \Aut(H)$. We will say that a finite
dimensional representation $U$ of $H$ is $G$-invariant if we have that for every $g\in G$
the $\sigma_g$-twisted $H$-representation ${}_{\sigma_g} U$ is isomorphic to $U$.
Note that if $U$ is the restriction of a $G$-representation then $U$ is automatically $G$-invariant. 

To relate NC(C)R's for $R^G$ and $R^{H}$ we
will use the following trivial result. 
\begin{lemmas}
\label{ref-4.5.1-63}
Assume that $U$ is  a finite dimensional $G$-invariant $H$-representation. Then
\[
\gldim M_{H}(\End(U))=\gldim M_G(\End(\Ind^G_{H}(U))).
\]
Moreover if $M_{H}(\End(U))$ is Cohen-Macaulay then so is  $M_G(\End(\Ind^G_{H}(U)))$.
\end{lemmas}
\begin{proof}
Let $U$ be a finite dimensional $G$-representation, not necessarily $G$-invariant. Put $V=\Ind^G_{H}U$ and $\Lambda=M_G(\End(V))$. We have
\[
V=\Gamma_{H}(G,U)\overset{\text{def}}{=}\{f:G\r U\mid \forall h\in H, g\in G: f(hg)=hf(g)\}.
\]
with $G$-action: $(g\cdot f)(g')=f(g'g)$.
For $\bar{g}\in G/H$ let 
\[
U_{\bar{g}}=\{f\in \Gamma_{H}(G,U)\mid \forall g'\in G: f\mid Hg'=0\text{ if 
$Hg'\neq Hg$}\}.
\]
Then 
\[
V=\bigoplus_{\bar{g}\in G/H} U_{\bar{g}}.
\]
We have for $g'\in G$: 
\begin{equation}
\label{ref-4.11-64}
g' U_{\bar{g}}\subset U_{\bar{g}\bar{g}'}.
\end{equation}
Moreover there is an isomorphism of $H$-representations
\[
U_{\bar{g}}\r {}_{\sigma_{g}} U:f\mapsto f(g)
\]
Now we define a $G/H$-grading on $\Lambda$ as follows
\[
\Lambda_{\bar{g}}=\{f\in \Lambda\mid \forall g'\in G:f(U_{\bar{g}'}\otimes R)\subset
U_{\bar{g}\bar{g}'}\otimes R
\}.
\]
By \eqref{ref-4.11-64} one has
\begin{equation}
\label{ref-4.12-65}
\Lambda_{\bar{g}}\cong M_{H}(\Hom(U,U_{\bar{g}})).
\end{equation}
Assume now that $U$ is $G$-invariant. Then \eqref{ref-4.12-65} implies 
the Cohen-Macaulayness claim since $U_{\bar{g}}\cong {}_{\sigma_{g}} U
\cong U$ as $H$-representations.

Choose a $H$-linear isomorphism $\theta_g:U\r U_{\bar{g}}$. Then it is easy to see that under the isomorphism
\eqref{ref-4.12-65} $\theta_g$ corresponds to an element of $\Lambda_{\bar{g}}$ which is a unit in $\Lambda$. 
Recall that a $L$-graded ring $\Gamma$ for a group $L$ is said to be \emph{strongly
  graded} if $\Gamma_{e}=\Gamma_u\Gamma_{u^{-1}}$ for all $u\in L$, (see \cite[\S I.3]{NVO}). 
  In particular
$\Lambda$ is a strongly $G/H$-graded ring.  
Hence $\gldim \Lambda = \gldim \Lambda_e=\gldim M_{H}(\End(U))$ follows from \cite[Cor. 7.6.18]{MR}
 combined with the fact that the categories of $\Lambda$-modules and graded
$\Lambda_e$-modules are equivalent.
\end{proof}
\emph{Now we restrict to the case that $H$ is the identity component $G_e$ of $G$.}
To check whether a $G_e$-representation is $G$-invariant we proceed as follows. Let $\alpha$ be an automorphism of
$G_e$. Then $\alpha(T)\subset \alpha(B)$ are respectively a maximal torus and a Borel subgroup and there exists $g_0\in G_e$ such that $g_0 \alpha(B) g_0^{-1}=B$,
$g_0 \alpha(T) g_0^{-1}=T$. Put $\tilde{\alpha}=(g_0\cdot g_0^{-1})\circ \alpha$. Then $\tilde{\alpha}$ preserves $(T,B)$ and hence it acts on $X(T)_\RR$
preserving the (positive) roots and the dominant cone. From this we obtain for $\chi\in X(T)^+$:
\begin{equation}
\label{ref-4.13-66}
{}_\alpha V(\chi)\cong {}_{\tilde{\alpha}}V(\chi)=V(\tilde{\alpha}^\ast(\chi))
\end{equation}
where $\tilde{\alpha}^\ast(\chi)=\chi\circ \tilde{\alpha}$.
We may use this observation in conjunction with the following lemma.
\begin{lemmas} \label{ref-4.5.2-67}
Assume ${}_\alpha W\cong W$ as $G_e$-representations (hence in particular if $\alpha=\sigma_g$ as above). Then for
$u,v\in \RR$, 
\[
-u\bar{\rho}+v\Sigma\qquad 
-u\bar{\rho}+v\bar{\Sigma}
\]
are stable under $\tilde{\alpha}$. The same is true for
\[
-u\bar{\rho}+v\bar{\Sigma}_{\varepsilon}
\]
provided $\varepsilon$ is stable under $\tilde{\alpha}$.
\end{lemmas}
\begin{proof} Since $W\cong {}_\alpha W\cong {}_{\tilde{\alpha}} W$ and $\tilde{\alpha}$ preserves $T$ we see that $\tilde{\alpha}$
permutes the weights of~$W$. Hence $\tilde{\alpha}$ preserves $\Sigma$ and its closure $\bar{\Sigma}$. Since $\tilde{\alpha}$
preserves positive roots, it also preserves $\bar{\rho}$. The lemma is now obvious. 
\end{proof}
\section{Determinantal varieties}
\label{ref-5-68}
In the sections \S\ref{ref-5-68}-\S\ref{ref-9-93} below we will apply the
technical results stated in \S\ref{ref-1.5-23}-\S\ref{ref-1.6-26}. Proofs 
of these technical results will be given in  \S\ref{ref-11-102}-\S\ref{ref-12-122}.
\subsection{Preliminaries}
\label{ref-5.1-69}
  In the examples below we will have to verify the genericity condition (Definition \ref{ref-1.3.5-9}). This is routine but sometimes a bit
messy so we will leave it to the reader. Here we outline how one may do the verification in general
but in practice there are usually short cuts. Assume that $X=W^\ast$ is a~$G$-representation. Let $X^s$ be
  the locus of points in $X$ which are \emph{stable}, i.e.\ which have
  closed orbit and finite stabilizer. Then $X-X^s$ may be described by
  the usual numerical criterion in terms of one-parameter subgroups \cite{Mumford}. This
  may be used to bound the dimension of $X-X^s$. 

  The locus $X_n$ of points in $X$ which are stabilized by an element
  of order $n$ may also be described numerically using one parameter
  subgroups. Indeed assume that $g\in G$ has order $n$ and $x\in X$ is
  such that $gx=x$. The element~$g$ is in particular semi-simple and
  hence it is contained in a maximal torus and therefore in the image
  of an injective one-parameter subgroup $\lambda:G_m\r G$. Let
  $(\mu_i)_i\in \ZZ$ be the weights for a diagonalization of the
  $\lambda$-action on $X$. Then since $x$ is a fixed point for
  $\lambda^{-1}(g)$ which has order $n$ we see that if $n\nmid \mu_i$
  we must have $x_i=0$. In this way one may bound the dimension of
  $X_n$. Since $X^{\mathbf{s}}=X^s-\bigcup_{p\text{ prime}} X_p$ we are done.
\subsection{Proof of Theorem \ref{ref-1.4.1-12}}
In this section we will prove Theorem \ref{ref-1.4.1-12}. Let $Y=(y_{ij})$ be a generic $h\times h$ matrix. Put
$S_{n,h}:=k[Y_{n,h}]=k[Y]/(I_{n})$ where $I_{n}$ is the ideal generated by the minors of size $n+1$. 

Let~$V$ be a vector space of dimension $n$ and let $G=\GL(V)$. 
Put $W=V^h\oplus (V^\ast)^h $, $R=SW$. Then $X=\Spec R=W^\ast\cong (V^\ast)^h\oplus V^h$. Let $y_{ij}$ be the functions
on $X$ obtained by pairing the $i$'th copy of $V^\ast$ and the $j$'th copy of $V$ in $X$. Then the first and second
fundamental theorems for the general linear group state: 
\begin{theorems} \label{ref-5.2.1-70} \cite{weyl} The morphism $S_{n,h}\r R^G:y_{ij}\mapsto y_{ij}$ is an isomorphism.
\end{theorems}
Hence to prove Theorem \ref{ref-1.4.1-12} it suffices to construct the
corresponding resolutions for $R^G$. We will do this using Theorem \ref{ref-1.6.4-30}  taking into account that the condition $h\ge n+1$
ensures that $W$ is generic (cfr.\ \S\ref{ref-5.1-69}). Choose a basis for~$V$ and let $T$ be the standard maximal torus $\{\diag(z_1,\ldots,z_n)\}$.
We let $L_i\in X(T)$ be given by $(z_1,\ldots,z_n)\mapsto z_i$.  Hence 
$
X(T)=\{\sum_i a_iL_i\mid a_i\in \ZZ\}
$.
For~$\lambda$ a partition let $S^\lambda$ be the corresponding Schur
functor. If $\chi=\sum_i \lambda_i L_i\in X(T)$ is the weight
corresponding to a partition $\lambda$ 
then $V(\chi)=S^\lambda V$. For $a,b$ let
$B_{a,b}$ be the set of partitions fitting in a box of size $a\times
b$.  We prove the following result.
\begin{propositions}
\label{ref-5.2.2-71}
Put
\begin{equation}
\label{ref-5.1-72}
M=\bigoplus_{\lambda\in B_{n,h-n}} M(S^\lambda V).
\end{equation}
Then $\End_{R^G}(M)$ is a NCCR for ${R^G}$.
\end{propositions}
In the next section we will show that this NCCR is the same NCCR as in \cite{VdB100}.
\begin{proof}[Proof of Proposition \ref{ref-5.2.2-71}]
The weights of  $V$ are $(L_i)_i$.  A system
of positive roots is given by $(L_i-L_j)_{i>j}$. From this we compute
\begin{equation}
\label{ref-5.2-73}
\bar{\rho}=(n-1)/2 L_1+(n-3)/2 L_2+\cdots+(-n+1)/2 L_n.
\end{equation}
The dominant cone is given by
\[
X(T)_\RR^+=\{\sum_i a_iL_i\mid a_i\in \RR, a_1\ge \cdots\ge a_n\}.
\]
The weights of $W$ are $(\pm L_i)_i$, each weight occurring with multiplicity $h$. Hence $W$ is quasi-symmetric. Furthermore
\[
\Sigma=\{\sum_i a_iL_i\mid a_i\in ]-h,h[\}
\]
and so
\[
-\bar{\rho}+(1/2)\bar{\Sigma}=\{\sum_i a_iL_i\mid a_i\in [-h/2-(n+1)/2+i,h/2-(n+1)/2+i]\}.
\]
The Weyl group permutes the $L_i$. From this we see that if we choose the $\Wscr$-invariant $\varepsilon=\epsilon(L_1+\cdots+L_n)$ with $\epsilon<0$
we find
\begin{align*}
\bar{\Sigma}_{\varepsilon}&=\{\sum_i a_iL_i\mid a_i\in [-h,h[\},\\
\bar{\Sigma}_{\pm\varepsilon}&=\{\sum_i a_iL_i\mid a_i\in ]-h,h[\}=\Sigma.
\end{align*}
Applying Theorem \ref{ref-1.6.4-30} we conclude that $R^G$ has a NCCR given by
\begin{align*}
{\Lscr}=X(T)^+\cap (-\bar{\rho}+(1/2)\bar{\Sigma}_\varepsilon)&=\{\sum_i a_iL_i\mid a_i\in [-h/2-(n+1)/2+i,h/2-(n+1)/2+i[\cap \ZZ,a_1\ge \cdots \ge a_n\}\\
&=\{\sum_i a_iL_i\mid a_i\in \ZZ, h/2-(n-1)/2>  a_1\ge \cdots \ge a_n\ge -h/2+(n-1)/2\}.
\end{align*}
Since $W$ is generic it follows from Corollary \ref{ref-4.1.4-39} that we get the same NCCR by translating ${\Lscr}$ by a character of~$G$. Using as character a suitable power of the determinant
 we see that the following set of weights 
\[
{\Lscr}'=\{\sum_i a_iL_i\mid a_i\in \ZZ, h-n\ge a_1\ge \cdots \ge a_n\ge 0\}
\]
works equally well,
from which one obtains the statement of the proposition.
\end{proof}
\begin{remarks} The results in \cite{VdB100} are stated for non-necessarily square matrices. In the non-square 
case the $k$-algebra $S_{n,h}$ is not Gorenstein and in that case one obtains a NCR instead of a NCCR. One may still obtain this NCR from
the methods in this paper by applying the combinatorial algorithm given in Remark \ref{ref-11.3.2-114} below  directly.
\end{remarks}

\makeatletter
\def\@cite@ofmt{\bfseries\hbox}
\makeatother
\subsection{Comparison with 
\cite{VdB100}}
\label{ref-5.3-74}
\strut
\makeatletter
\def\@cite@ofmt{\hbox}
\makeatother
The NCCR constructed
in \cite{VdB100} is obtained from
 a tilting bundle on a Springer type crepant resolution of $\Spec S_{n,h}$. This is the main geometric 
method for constructing such resolutions. However it is also shown that there is an equivalent algebraic construction (see \cite[Prop.\ 3.5]{VdB100}) which we now describe.

As in \cite{VdB100} it will be convenient to introduce auxiliary vector spaces $F_1,{F_2}=k^h$ and to put
$W={F_1}^\ast\otimes V\oplus {F_2}\otimes V^\ast$. We may then identify $k[Y]$ with the coordinate ring $S$
of the vector space $\Hom_k({F_2},{F_1})$ and $\Spec S_{n,h}$ is the locus in $\Hom_k({F_2},{F_1})$ of the
matrices which have rank $\le n$. The tautological map of $S$-modules
\[
\phi:{F_2}\otimes_k S\r {F_1}\otimes_k S
\]
restricts to a map
\[
\phi:{F_2}\otimes_k S_{n,h}\r {F_1}\otimes_k S_{n,h},
\]
which has generically rank $n$.  
The NCCR constructed in \cite{VdB100} is of the form
$\End_{S_{n,h}}(N)$ where $N=\bigoplus_{\lambda\in B_{n,h-n}} N_{\lambda}$ with
\[
N_{\lambda}=\text{image}\, (S^\lambda {F_1}^\ast\otimes S_{n,h}
\xrightarrow{S^\lambda \phi^\ast} S^\lambda {F_2}^\ast\otimes S_{n,h}).
\]
\begin{lemmas} We have an isomorphism of $S_{n,h}$-algebras
\[
\End_{S_{n,h}}(N)\cong \End_{S_{n,h}}(M)
\]
where $M$ is as in \eqref{ref-5.1-72}.
\end{lemmas}
\begin{proof} We identify $S_{n,h}$ with $R^G$ as in Theorem \ref{ref-5.2.1-70}.
By Lemma
\ref{ref-5.3.2-75} below we have $N_\lambda=M(S^\lambda V^\ast)$. Hence $\End_{R^G}(N)=\End_{R^G}(N')$ with $N'=\bigoplus_{\lambda\in B_{n,h-n}}M(S^\lambda V^\ast)$.
However we also have for $\lambda\in B_{n,h-n}$: $S^\lambda
V^\ast=S^{\lambda^!} V\otimes (\det V)^{-h+n}$ with
$\lambda^!=(h-n-\lambda_n,\cdots,h-n-\lambda_1)\in B_{n,h-n}$ from which we easily deduce that
$\End_{R^G}(M)=\End_{R^G}(N')$.
\end{proof}
We have used the following result
\begin{lemmas} 
\label{ref-5.3.2-75} Let $\lambda$ be a partition. There is an isomorphism of
$R^G$-modules $N_\lambda\cong M(S^\lambda V^\ast)$.
\end{lemmas}
\begin{proof}
With our current conventions $R$ is the coordinate ring of $X=W^\ast=\Hom({F_2},V)\oplus \Hom(V,{F_1})$. 
The tautological map $\phi:{F_2}\otimes S_{n,h}\rightarrow {F_1}\otimes S_{n,h}$ is obtained by taking invariants
of the composition of the following tautological maps
\[
{F_2}\otimes R\xrightarrow{\phi_1} V\otimes R\xrightarrow{\phi_2} {F_1}\otimes R
\]
and hence the dual map $\phi^\ast$ decomposes as 
\[
{F_1}^\ast\otimes R\xrightarrow{\phi_2^\ast} V^\ast\otimes R\xrightarrow{\phi_1^\ast} {F_2}^\ast\otimes R.
\]
Here $\phi_1^\ast$ is injective and the support of the cokernel of $\phi_2^\ast$
 has codimension $\ge 2$.  
Let $\lambda\in B_{n,h-n}$. Applying Schur functors and taking invariants  we get that
$S^\lambda \phi^\ast$ decomposes as 
\[
S^\lambda {F_1}^\ast \otimes R^G\xrightarrow{S^\lambda \phi_2^\ast} M(S^\lambda V^\ast)
\xrightarrow{S^\lambda \phi_1^\ast} S^\lambda {F_2}^\ast\otimes R^G.
\]
Thus
\begin{equation}
\label{ref-5.3-76}
 M(S^\lambda V^\ast)\supset \im S^\lambda\phi^\ast=N_\lambda.
\end{equation}
It is shown in \cite[Cor.\ 3.6]{VdB100} that $N_{\lambda}$ is
Cohen-Macaulay for $\lambda\in B_{n,h}$.  In fact a similar argument
shows that $N_\lambda$ is reflexive in general.  

The morphism \eqref{ref-5.3-76} is an inclusion of reflexive $R^G$ modules which differ on a subset of codimension $\ge 2$. Hence
it is in fact an equality. This finishes the proof.
\end{proof}

\section{Pfaffian varieties}
\label{ref-6-77}
\subsection{Preliminaries}
\label{ref-6.1-78}
In this section we will prove Theorem \ref{ref-1.4.3-14}. 
Let $Y=(y_{ij})$ be a generic skew-symmetric $h\times h$ matrix  and put
$S_{2n,h}^-:=k[Y_{2n,h}^-]=k[Y]/(I^-_{2n})$ where $I^-_{2n}$ is the ideal generated by principal Pfaffians  of size $2n+2$. 

Let $V$ be a vector space of dimension $2n$ equipped with a non-degenerate skew-symmetric bilinear form $\langle-,-\rangle$ and let $G=\Sp_{2n}(k)$
be the corresponding symplectic group. Put $W=V^h$, $R=SW$. Then $X=\Spec R=W^\ast\cong V^h$. Let $y_{ij}$ be the functions
on $X$ obtained by pairing the $i$'th and $j$'th copy of $V$ in $X$ using $\langle-,-\rangle$. Then the first and second
fundamental theorems for the symplectic group state:
\begin{theorems} \cite[Thm 6.7]{DeConciniProcesi} The morphism $S^-_{2n,h}\r R^G:y_{ij}\mapsto y_{ij}$ is an isomorphism.
\end{theorems}
Hence to prove Theorem \ref{ref-1.4.3-14} it suffices to construct the
corresponding resolutions for $R^G$. We will do this using the results in \S\ref{ref-1.6-26} taking into account that the condition $h\ge 2n+1$
ensures that $W$ is generic (cfr.\ \S\ref{ref-5.1-69}).
Let $v_1,\ldots,v_{2n}$ be a basis for~$V$ such that $\langle-,-\rangle$ is given by
 $\langle v_i,v_{i+n}\rangle=1$, $\langle v_i,v_j\rangle=0$
for $j\neq i\pm n$.  $G$ is a simply connected semi-simple algebraic group which contains a standard
maximal torus given by
$
T=\{\diag(z_1,\ldots,z_n,z_1^{-1},\ldots,z_n^{-1})\}
$.
We let $L_i\in X(T)$ be given by $(z_1,\ldots,z_n)\mapsto z_i$.  Hence 
$
X(T)=\{\sum_i a_iL_i\mid a_i\in \ZZ\}
$.
If~$\lambda$ is partition of length $n$ with corresponding weight $\chi=\sum_i\lambda_i L_i$ write $S^{\langle \lambda\rangle}V$ for $V(\chi)$
(such $\chi$ is dominant, see below).
We prove the following result.
\begin{propositions}
Put
\begin{equation}
\label{ref-6.1-79}
M=\bigoplus_{\lambda\in B_{n,\lfloor h/2\rfloor-n}} M(S^{\langle \lambda\rangle} V).
\end{equation}
Then $\End_{S^-_{2n,h}}(M)$ is an NCR for $S^-_{2n,h}$. It is a NCCR if $h$ is odd.
\end{propositions}
\begin{proof}
We use some fragments of the representation theory of the symplectic group, following \cite[Ch.\ 16]{FH}.
The weights of  $V$ are $(\pm L_i)_i$. A system
of positive roots is given by $(L_i\pm L_j)_{i>j}$, $(2L_i)_i$. From this we compute
\[
\bar{\rho}=nL_1+(n-1)L_2+\cdots+L_n.
\]
The dominant cone is given by
\[
X(T)_\RR^+=\{\sum_i a_iL_i\mid a_i\in \RR, a_1\ge \cdots\ge a_n\ge 0\}.
\]
The weights of $W$ are $(\pm L_i)_i$, each weight occurring with multiplicity $h$. Hence $W$ is quasi-symmetric. Furthermore
\begin{equation}
\label{ref-6.2-80}
\Sigma=\{\sum_i a_iL_i\mid a_i\in ]-h,h[\}
\end{equation}
and so
\[
-\bar{\rho}+(1/2)\bar{\Sigma}=\{\sum_i a_iL_i\mid a_i\in [-h/2-n-1+i,h/2-n-1+i]\}.
\]
The only way we can apply Theorem \ref{ref-1.6.3-29}  is with $A=0$ ($G$ is simply connected) and $\varepsilon=0$ ($G$ is semi-simple). Thus $\bar{\Sigma}_{\epsilon}=\bar{\Sigma}$. 
We conclude that 
$R^G$ has a NCR given by
\begin{align}
{\Lscr}=X(T)^+\cap (-\bar{\rho}+(1/2)\bar{\Sigma})&= \{\sum_i a_iL_i\mid a_i\in [-h/2-n-1+i,h/2-n-1+i]\cap\ZZ, a_1\ge \cdots \ge a_n\ge 0\}\\
&=\{\sum_i a_iL_i\mid a_i\in \ZZ,  h/2-n\ge a_1\ge \cdots \ge a_n\ge 0\}.\label{ref-6.4-81}
\end{align}
If $h$ is odd
we have 
\[
X(T)\cap (-\bar{\rho}+(1/2)(\bar{\Sigma}-\Sigma))=\emptyset.
\]
In this case Theorem \ref{ref-1.6.4-30} implies that ${\Lscr}$ as in \eqref{ref-6.4-81} yields a NCCR provided $h$ is odd. This finishes the proof of
Theorem \ref{ref-1.4.3-14}.
\end{proof}
\begin{remarks} 
\label{ref-6.1.3-82}
In the same way as in \S\ref{ref-5.3-74} one may check that the NCR for  $S^{-}_{2n,h}$ constructed by Weyman and Zhao in \cite{WeymanZhao} is given by the module of covariants
\[
\bigoplus_{\lambda\in B_{2n,h-2n}}M(S^\lambda V).
\]
This resolution is not a NCCR if $h$ is odd \cite[Prop.\
7.4]{WeymanZhao} and also in the even case it is larger than ours. Indeed
$B_{2h,h-2n}$ contains the partition $(h-2n)$ and the restriction of
$S^{h-2n}V$ to the symplectic group is $S^{\langle h-2n\rangle} V$
(e.g. by \cite[(17.10)]{FH}).  However it is clear that
$M(S^{\langle h-2n\rangle}V)$ will not be a summand of
\eqref{ref-6.1-79}.
\end{remarks}
\begin{remarks} \label{ref-6.1.4-84} 
The notion of a NC(C)R   may be extended in an obvious way to schemes (replacing modules by coherent sheaves and $\Hom$ by $\uHom$).
The $k$-algebra~$R$ carries  an $\NN$-grading by giving the elements of $W$ degree one. With this convention $S^-_{2n,h}$ is generated in
degree two (and hence lives in even degree). Since gradings are equivalent to $G_m$-actions we obtain an equivalence of symmetric monoidal categories
\[
\coh(\Proj S^-_{2n,h})\cong \coh(G_m,\Spec S^-_{2n,h}-\{0\})_{(2)}
\]
where the subscript $(2)$ indicates that   $\ker(G_m\xrightarrow{(-)^2} G_m)$-act trivially.

Since the notion of a NC(C)R is 
compatible with restriction we obtain from Theorem \ref{ref-1.4.3-14} a NC(C)R for $\Spec S^-_{2n,h}-\{0\}$. This NC(C)R can be graded
in a naive way but then it does not live on in even degree, so it does not descend to
a NC(C)R for $\Proj  S^-_{2n,h}$. We may fix this by defining the graded NCCR to be obtained from
\[
M=\bigoplus_{\lambda\in B_{n,\lfloor h/2\rfloor-n}} M(S^{\langle \lambda\rangle} V)(|\lambda|)
\]
where the functor $(n)$ shifts a graded object $n$ places to the left. Looking at the action of $Z(\Sp_{2n}(k))\cong \ZZ/2\ZZ$ we see
that with this choice $\End_{R^G}(M)$ only lives in even degree and hence it defines a NC(C)R for $\Proj S^-_{2n,h}$.

In \cite[Lemma 8.6]{Kuznetsov} a much smaller NC(C)R for $\Proj  S^-_{4,h}$ or equivalently for $\Spec S^-_{4,h}-\{0\}$ is constructed. 
With a similar method as \S\ref{ref-5.3-74} one may check that 
this resolution yields a NCR for  $\Spec S^-_{4,h}-\{0\}$ given by
a module of covariants (dropping shifts and hence the grading information)
\[
N=M(k)\oplus M(V)\oplus\cdots \oplus M(S^{\lfloor h/2\rfloor-2} V)
\]
(where in this case $n=2$ and hence $\dim V=4$). 
Note that $S^tV$ is an irreducible
representation for $\Sp_{2n}(k)$ (e.g. by \cite[(17.10)]{FH}) with highest weight $tL_1$ whereas ${\Lscr}$ in \eqref{ref-6.4-81} contains $L_1+L_2$. 
Thus the representations occurring in the construction of $N$ are a strict subset of those given by 
\eqref{ref-6.4-81}. 
If $h$ is odd then this immediately implies $N$ cannot yield a NCCR for $S^-_{4,h}$ by Proposition \ref{ref-3.5-35}. 

\medskip

It is possible to confirm using our methods that $N$ does indeed give a NCR for $\Spec S^-_{4,h}-\{0\}$ and a NCCR if $h$ is odd.
By  Lemma \ref{ref-4.2.1-43} it is sufficient to look at the slice representations. There is only
one non-trivial case to consider, which is when $x\in X-\{0\}$ has closed orbit and non-trivial stabilizer. In that case there must be 
a one-parameter subgroup $\lambda$ of $G$ which fixes $x$. 
Let $(x_i)_{i=1}^h\in V$ be the components of $x$ in the decomposition $X=V^h$. Computing $X^\lambda$ and changing $\lambda$ in its $G$-orbit (in particular
making it a one-parameter subgroup of $T$) we see that
we may assume $(x_i)_i\in kv_2+kv_4$. Moreover the $(x_i)_i$ must span $ kv_2+kv_4$ since otherwise $0\in \overline{Gx}$. Let $\mathfrak{g}=\Lie(G)$. 
According to \cite[\S16.1]{FH} we have 
\[
\mathfrak{g}=
\begin{pmatrix}
A&B\\
C&-A^t
\end{pmatrix}
\]
where $B,C$ are symmetric $2\times 2$-matrices and $A$ is an arbitrary $2\times 2$-matrix.
By \eqref{ref-4.2-45}  we find that $\mathfrak{g}_x$ is given by the matrices
\[
\begin{pmatrix}
a&0&b&0\\
0&0&0&0\\
c&0&-a&0\\
0&0&0&0
\end{pmatrix}.
\]
Hence $\mathfrak{g}_x=\mathfrak{sl}_2$. Let $V'$ be the standard representation of $\mathfrak{sl}_2$. Direct
inspection shows $V=V'\oplus k^{\oplus 2}$ and $\mathfrak{g}/\mathfrak{g}_x=(V')^{\oplus 2}\oplus k^{\oplus 3}$ as $\Sl_2$-representations.  Hence according to \eqref{ref-4.3-46}  we find that the slice representation at $x$ is given by
\[
\mathsf N_x=(V')^{\oplus h-2}\oplus k^{\oplus 2h-3}.
\]
We conclude that $\widehat{N}_{\bar{x}}$ is a sum of summands of $\widehat{N}'_{\bar{0}}$ where $N'$ is the $k[\mathsf{N}_x]^{\Sl_2}$-module of covariants given by
\[
N'=M(k)\oplus M(V')\oplus \cdots \oplus M(S^{\lfloor h/2\rfloor-2}V')
\]
with each summand appearing at least once. It now follows from Theorem \ref{ref-1.4.7-16} that $N'$ does indeed yield a NCR and a NCCR when $h$ is odd.
\end{remarks}

\section{Determinantal varieties: symmetric matrices}
\label{ref-7-85}
\subsection{Preliminaries}
In the next few sections we prove Theorem \ref{ref-1.4.5-15}. Let $Y=(y_{ij})$ be a generic symmetric $h\times h$ matrix and put 
$S_{t,h}^+:=k[Y_{t,h}^+]=k[Y]/(I^+_{t})$ where $I^+_{t}$ is the ideal generated by minors of size $t+1$.

Let $V$ be a vector space of dimension $t$ equipped with a non-degenerate symmetric  bilinear form $(-,-)$ and let $H=O_t(k)$
we the corresponding orthogonal group. Put $W=V^h$, $R=SW$. Then $X=\Spec R=W^\ast\cong V^h$. Let $y_{ij}$ be the functions
on $X$ obtained by pairing the $i$'th and $j$'th copy of $V$ in $X$ using $(-,-)$. Then the first and second
fundamental theorems for the orthogonal group state:
\begin{theorems} \cite{weyl} The morphism $S^+_{t,h}\r R^H:y_{ij}\mapsto y_{ij}$ is an isomorphism.
\end{theorems}
Hence to prove Theorem \ref{ref-1.4.3-14} it suffices to construct the
corresponding resolutions for $R^H$. 
We will do this using Theorem \ref{ref-1.6.4-30} taking into account that the condition $h\ge t+1$
ensures that $W$ is generic (cfr.\ \S\ref{ref-5.1-69}). Note however that $H=O_t(k)$ is not connected so Theorem \ref{ref-1.6.4-30} does not immediately apply.
Instead we will apply Theorem \ref{ref-1.6.4-30} with $G=\SO_t(k)$ and its double cover $\bar{G}=\Spin_t(k)$ and we will then apply Lemma \ref{ref-4.5.1-63}
together with Lemma \ref{ref-4.5.2-67}.
\subsection{Some facts about the orthogonal group}
For the benefit of the reader we recall some facts about the orthogonal group. 
 For the
representation theory of the special orthogonal group, and more
generally the spin group we follow \cite[Ch.\ 19,20]{FH}.

The orthogonal group fits in a diagram consisting of four related groups. 
\[
\xymatrix@1{
&&0\ar[d]&0\ar[d]\\
0\ar[r]&\ZZ/2\ZZ\ar[r]\ar@{=}[d]&\Spin_t(k)\ar[d]\ar[r]&\SO_t(k)\ar[d]\ar[r]&0\\
0\ar[r]&\ZZ/2\ZZ\ar[r]&\Pin_t(k)\ar[d]\ar[r]& O_t(k)\ar[d]\ar[r]&0\\
&&\ZZ/2\ZZ\ar[d]\ar@{=}[r]&\ZZ/2\ZZ\ar[d]\\
&& 0 &0
}
\]
where the horizontal arrows are central extensions. 
If $t=2n$ is even let $v_1,\ldots,v_{2n}$ be a basis for
$V$ such that $(-,-)$ is given by $(v_i,v_{i+n})=1$, $(v_i,v_j)=0$ for
$j\neq i\pm n$.    Then $\SO_t(k)$ contains a standard maximal torus given by
$T=
\diag(z_1,\ldots,z_n,z_1^{-1},\ldots,z_n^{-1})
$.

If $t=2n+1$ is odd let $v_1,\ldots,v_{2n},v_{2n+1}$ be a basis for $V$ such that $(-,-)$ is given by
 $(v_i,v_{i+n})=1$ for $i=1,\ldots,n$, $(v_{2n+1},v_{2n+1})=1$, 
$(v_i,v_j)=0$
for $(i,j)\neq (2n+1,2n+1)$ and $j\neq i\pm n$.  Then $\SO_t(k)$ contains a standard maximal torus given by
$T=
\diag(z_1,\ldots,z_n,z_1^{-1},\ldots,z_n^{-1},1)$.

We let $\bar{T}$ be the inverse image of~$T$ in $\Spin_{t}(k)$. It is still a maximal torus.
Put $A=\ker(\bar{T}\r T)\cong \ZZ/2\ZZ$. We let $L_i\in X(T)$ be given by $(z_1,\ldots,z_n)\mapsto z_i$.  Hence 
$
X(T)=\{\sum_i a_iL_i\mid a_i\in \ZZ\}
$. Inside $X(T)_\RR=X(\bar{T})_\RR$ we have the coset decomposition 
\[
X(\bar{T})=X(T)\coprod X(T)_{1/2}
\]
where
\[
X(T)_{1/2}\overset{\text{def}}{=}\{\sum_i a_iL_i\mid a_i\in 1/2+\ZZ\}.
\]
The dominant cone is given by 
\[
X(T)_\RR^+=\{\sum_i a_iL_i\mid a_i\in \RR, a_1\ge \cdots\ge a_{n-1}\ge |a_n|\ge 0\}
\]
if $t=2n$ is even and by
\[
X(T)_\RR^+=\{\sum_i a_iL_i\mid a_i\in \RR, a_1\ge \cdots\ge a_{n-1}\ge a_n\ge 0\}
\]
if $t=2n+1$ is odd. 
It is well known how the representation theory of $O_t(k)$ and $\SO_t(k)$ are related. The same
relation holds between $\Pin_t(k)$, $\Spin_t(k)$ but we did not find an explicit reference where this is stated.
Hence the following lemma, which will in fact only be used for the computation of $i_{1/2}(\Pin_t(k))$ (see p.\pageref{index})  in Lemma \ref{ref-7.2.2-87}
below.
\begin{lemmas} {}
\label{ref-7.2.1-86}
  Assume that $\chi=a_1L_1+\cdots+a_nL_n\in X(\bar{T})^+$ and let
  $V(\chi)$ be the corresponding representation of $\Spin_t(k)$.  If
  $t=2n+1$ is odd or if $t=2n$ is even and $a_n=0$ then the action of
  $\Spin_t(k)$ on $V(\chi)$ can be lifted in two distinct ways, up to isomorphism, to an
  action of $\Pin_t(k)$.  Denote the resulting irreducible $\Pin_t(k)$
  representations by $V^+(\chi)$ and $V^-(\chi)$.

If $t=2n$ is even and $a_n\neq 0$ then put $\chi^+=\chi$, $\chi^-=a_1L_1+\cdots+a_{n-1}L_{n-1}-a_nL_n$. Then the $\Spin_t(k)$-action on $V(\chi^+)\oplus V(\chi^-)$ extends
uniquely to a $\Pin_t(k)$-action, up to isomorphism. The resulting $\Pin_t(k)$-representation is irreducible. Denote it by $\tilde{V}(\chi)$.

The representations $V^{\pm}(\chi)$ and $\tilde{V}(\chi)$ (with $a_n>0$) form a complete list of irreducible $\Pin_t(k)$-representations. 
\end{lemmas}
\begin{proof}
Let $\bar{y}\in \Pin_t(k)-\Spin_t(k)$ and let $\bar{\sigma}\in \Aut(\Spin_t(k))$ be given by conjugation by $\bar{y}$.
As explained for example in \cite[\S5.8)]{weyl} 
the irreducible $\Pin_t(k)$-representations
are determined by the orbits of the $\ZZ/2\ZZ$-action $V(\chi)\mapsto {}_{\bar{\sigma}}V(\chi)$ on the irreducible
$\Spin_t(k)$-representations. 
Hence it is sufficient to describe this action. To do this we will choose $\bar{y}$ in a special way.
If $t$ is even let $y$ be the reflection $v_n\leftrightarrow v_{2n}$, if $t$ is odd
let $y$ be the reflection $v_{2n+1}\mapsto -v_{2n+1}$ (all other basis vectors remaining fixed). Let $\sigma$ be the automorphism of
$\SO_t(k)$ given by conjugation by~$y$. Then clearly $\sigma$ preserves $T$  and moreover $\sigma^\ast:X(T)\r X(T):\chi\mapsto\chi\circ \sigma$
has the following form: if $t$ is odd then it is the identity and if $t$ is even then 
it sends $L_n$ to $-L_n$ and preserves the other $L_i$.

Let $\bar{y}\in \Pin_t(k)$ be a lift of $y$ and as above let $\bar{\sigma}$ be the automorphism of $\Spin_t(k)$ given by conjugation by $\bar{y}$.
 Then $\bar{\sigma}$ preserves $\bar{T}$ and the
action of $\bar{\sigma}^\ast$ on $X(\bar{T})$ is simply the extension of the action of $\sigma^\ast$ on $X(T)$. Hence $\bar{\sigma}^\ast$ is given by
the same formula as $\sigma^\ast$.

It is clear that $\bar{\sigma}$ preserves dominant weights, and moreover it also preserves a system of positive roots (see \eqref{ref-7.1-88} and \eqref{ref-7.2-89} below).
In particular it preserves the ordering on weights.
As in \eqref{ref-4.13-66} we find  ${}_{\bar{\sigma}} V(\chi)=V(\bar{\sigma}^\ast(\chi))$.
This finishes the proof.
\end{proof}
For use below we state the following lemma.
\begin{lemmas} 
\label{ref-7.2.2-87}
Denote the unique non-trivial character of $A$ by $1/2$. Then we have 
\begin{align*}
i_{1/2}(\Spin_{t}(k))&=
\begin{cases}
2^{n-1}&\text{if $t=2n$ is even}\\
2^n&\text{if $t=2n+1$ is odd}
\end{cases}\\
i_{1/2}(\Pin_t(k))&=
\begin{cases}
2^n&\text{if $t=2n$ is even}\\
2^n&\text{if $t=2n+1$ is odd}
\end{cases}
\end{align*}
\end{lemmas}
\begin{proof}
  The group $\Spin_{t}(k)$ has one or two distinguished irreducible
  representations which are called \emph{spin representations}. If
  $t=2n$ is even there are two spin representations $S^{\pm}$ of
  dimension $2^{n-1}$ with highest weight given by
  $(1/2)(L_1+\cdots+L_{n-1}\pm L_n)$. If $t=2n+1$ is odd then there is
  a unique spin representation $S$ of dimension $2^n$. 

To prove the lemma for $\Spin_t(k)$ we clearly need that the dimension of every
finite dimensional representation of $\Spin_t(k)$ with non-trivial central character is divisible by the dimension 
of the spin representations. This follows for example from the structure of
the representation ring of $\Spin_t(k)$. See \cite[\S23.2]{FH}. The result
for $\Pin_t(k)$ immediately follows from Lemma \ref{ref-7.2.1-86}.
\end{proof}
We now discuss non-commutative resolutions for invariants for special
orthogonal groups. Since these are connected we may apply Theorem \ref{ref-1.6.4-30}.
As usual the cases where $t$ is even and $t$ is odd are a bit
different, so to keep things clear we discuss them separately.
\subsection{The case \mathversion{bold} $G=\SO_{2n}(k)$}
Below we assume $t=2n$. We put $G=\SO_{2n}(k)$ and $\bar{G}=\Spin_{2n}(k)$.
We prove the following result.
\begin{propositions}
Let $\theta=0$ if $h$ is odd and $\theta=1/2$ if $h$ is even. Put
\[
{\Lscr}^e=\{ \sum_i a_i L_i\mid  a_i\in \theta+\ZZ,  h/2-(n-1)>a_1\ge a_2\ge\cdots \ge a_{n-1}\ge |a_n|\}
\]
and let $U^e$ be the $\Spin_{2n}(k)$-representation given by
$
U^e=\bigoplus_{\chi\in {\Lscr}^e} V(\chi)
$.
If $h$ is even then
$
M(\End(U^e))
$
is a twisted NCCR of index $2^{n-1}$ for $R^G$. If $h$ is odd then it is a NCCR.
\end{propositions}
\begin{proof}
A system
of positive roots is given by $(L_i\pm L_j)_{i>j}$. From this we compute
\begin{equation}
\label{ref-7.1-88}
\bar{\rho}=(n-1)L_1+(n-2)L_2+\cdots+L_{n-1}.
\end{equation}
The weights of $V$ are $\pm L_i$ and hence $W$ has the same weights occurring with multiplicity $h$. In particular $W$ is quasi-symmetric. Furthermore
\[
\Sigma=\{\sum_i a_iL_i\mid a_i\in ]-h,h[\}
\]
and so
\[
-\bar{\rho}+(1/2)\bar{\Sigma}=\{ a_i L_i\mid -h/2-(n-i)\le a_i \le h/2-(n-i)\}.
\]
We can only apply Theorem \ref{ref-1.6.4-30} with $\varepsilon=0$ ($G$ is semi-simple). Hence $\bar{\Sigma}_{\varepsilon}=\bar{\Sigma}$.
If $h$ is odd we have
\[
X(T)\cap (-\bar{\rho}+(1/2)(\bar{\Sigma}-\Sigma))=\emptyset
\]
and if $h$ is even then
\[
X(T)_{1/2}\cap (-\bar{\rho}+(1/2)(\bar{\Sigma}-\Sigma))=\emptyset.
\]
The statement of the proposition now follows by inspecting
$
X(T)_{\theta}^+\cap (-\bar{\rho}+(1/2)\bar{\Sigma})
$ and using Theorem \ref{ref-1.6.4-30} together with Lemma \ref{ref-7.2.2-87}.
\end{proof}
\subsection{The case \mathversion{bold} $G=\SO_{2n+1}(k)$}
Below we assume $t=2n+1$. We put $G=\SO_{2n+1}(k)$ and $\bar{G}=\Spin_{2n+1}(k)$.
We prove the following result.
\begin{propositions}
Let $\theta=1/2$ if $h$ is odd and $\theta=0$ if $h$ is even. Put
\[
{\Lscr}^o=\{ \sum_i a_i L_i\mid  a_i\in \theta+\ZZ,  h/2-(n-1/2)>a_1\ge a_2\ge\cdots \ge a_{n-1}\ge a_n\ge 0\}
\]
and let $U^o$ be the $\Spin_{2n+1}(k)$-representation given by
$
U^o=\bigoplus_{\chi\in {\Lscr}^\circ} V(\chi)
$.
Then if $h$ is odd
$
M(\End(U^o))
$
is a twisted NCCR for $R^G$ of index $2^{n}$. If $h$ is even then it is a NCCR.
\end{propositions}
\begin{proof}
 A system
of positive roots is given by $(L_i\pm L_j)_{i>j}$, $(L_i)_i$. From this we compute
\begin{equation}
\label{ref-7.2-89}
\bar{\rho}=(n-1/2)L_1+(n-3/2)L_2+\cdots+(3/2)L_{n-1}+(1/2)L_n.
\end{equation}
The weights of $V$ are $0,\pm L_i$.  Hence $W$ as the same weights with multiplicity $h$. In particular $W$ is quasi-symmetric. Furthermore
\[
\Sigma=\{\sum_i a_iL_i\mid a_i\in ]-h,h[\}
\]
and so
\[
-\bar{\rho}+(1/2)\bar{\Sigma}=\{ a_i L_i\mid -h/2-(n-i+1/2)\le a_i \le h/2-(n-i+1/2)\}.
\]
Reasoning as in the even case we note that for
$h$  even
\[
X(T)\cap (-\bar{\rho}+(1/2)(\bar{\Sigma}-\Sigma))=\emptyset
\]
and for $h$ odd
\[
X(T)_{1/2}\cap (-\bar{\rho}+(1/2)(\bar{\Sigma}-\Sigma))=\emptyset.
\]
The statement of the proposition now follows by inspecting
$
X(T)_{\theta}^+\cap (-\bar{\rho}+(1/2)\bar{\Sigma})
$ and using Theorem \ref{ref-1.6.4-30} together with Lemma \ref{ref-7.2.2-87}.
\end{proof}
\subsection{Proof of Theorem \ref{ref-1.4.5-15}}
It follows from Lemma \ref{ref-7.2.1-86} that both $U^e$ and $U^o$ may be
lifted to representations of $\Pin_t(k)$ (alternatively one may see
that they are $\Pin_t(k)$-invariant by the general Lemma
\ref{ref-4.5.2-67}). It now suffices to use Lemma \ref{ref-4.5.1-63} together with
Lemma \ref{ref-7.2.2-87}.
\begin{remarks} 
\label{ref-7.5.1-90}
In the same way as in \S\ref{ref-5.3-74} one may check that the NCR for  $S^{+}_{t,h}$ constructed by Weyman and Zhao in \cite{WeymanZhao} is given by the module of covariants
\[
N=\bigoplus_{\lambda\in B_{t,h-t}}M(S^\lambda V).
\]
This is  not a NCCR unless $t=h-1$ \cite[Prop 6.6]{WeymanZhao}. In the latter case the construction by Weymann and Zhao gives the same result as ours. 
Indeed now $N$ is given by 
\[
N=\bigoplus_{i\le t}M(\wedge^i V),
\]
whereas for $t$ being respectively equal to $2n$ and $2n+1$ one has
\begin{align*}
\Lscr^{e}&=\{\sum_i a_i L_i\mid a_i\in \ZZ, 1\ge a_1\ge a_2\ge \cdots \ge a_{n-1}\ge |a_n|\},\\
\Lscr^{o}&=\{\sum_i a_i L_i\mid a_i\in \ZZ, 1\ge a_1\ge a_2\ge \cdots \ge a_{n-1}\ge a_n\ge 0\}.
\end{align*}
By the results in \cite[\S 19.5]{FH} the $(\wedge^i V)_i$ are irreducible $O_t(k)$-representations
and they are precisely the irreducible $O_t(k)$ representations whose restrictions to $SO_t(k)$
have highest weights in $\Lscr^e$ or $\Lscr^o$. This proves our claim.
\end{remarks}
\section{Non-commutative resolutions for $\Sl_2$-invariants}
\label{ref-8-92}
Here we prove Theorem \ref{ref-1.4.7-16}. We use the fact that $\Sl_2$-representations are always quasi-symmetric. If all $d_i$ are even we put $G=\PGL_2(k)$ and $\bar{G}=\Sl_2(k)$. If not all $d_i$ are even then we put $G=\Sl_2(k)$.
Note that excluding the special representations \eqref{ref-1.1-17} ensures that $W$ is generic in all cases (cfr \S\ref{ref-5.1-69}).
\subsection{The case \mathversion{bold} $G=\Sl_2(k)$}
Let $T$ be $\{\diag(z,z^{-1})\}$. We have $X(T)=\ZZ$ where $1$ corresponds to the character $z\mapsto z$. We put\footnote{We follow the convention that $0\in\NN$.} $X(T)^+=\NN$. With this identification  one has
$
\Sigma=]-s,s[
$
and
$\bar{\rho}=1$. So the claim that \eqref{ref-1.2-18} yields an NCR follows from Theorem \ref{ref-1.6.3-29} with $\bar{\mu}=0,\varepsilon=0$.

We can only apply Theorem \ref{ref-1.6.4-30} with $A=0$ and $\varepsilon=0$. 
We find that if $s$ is odd $X(T)\cap -\bar{\rho}+(1/2)(\bar{\Sigma}-\Sigma)=\emptyset$. Hence in this case \eqref{ref-1.2-18} yields a NCCR by Theorem \ref{ref-1.6.4-30}.
\subsection{The case \mathversion{bold} $G=\PGL_2(k)$}
Let $\bar{T}=\diag(z,z^{-1})\subset \bar{G}$ and let $T$ be the image of $\bar{T}$ in $\bar{G}$. We have $X(\bar{T})=\ZZ$ where $1$ corresponds to the character $z\mapsto z$. We put $X(\bar{T})^+=\NN$. One still has $
\Sigma=]-s,s[
$
and
$\bar{\rho}=1$. 
We now have  $A=\ZZ/2\ZZ$ and $X(T)=2\ZZ$.
So the claim that \eqref{ref-1.2-18} yields an NCR follows from Theorem \ref{ref-1.6.3-29}.

 If $\bar{1}$ denotes the unique non-trivial character of $A$ then $X(T)_{\bar{1}}=1+2\ZZ$.
For  Theorem \ref{ref-1.6.4-30} we still have $\varepsilon=0$. Then
we have if $s/2$ is even
\[
X(T)\cap -\bar{\rho}+(1/2)(\bar{\Sigma}-\Sigma)=\emptyset ,
\]
which implies that \eqref{ref-1.3-19} is a NCCR
and if $s/2$ is odd
\[
X(T)_{\bar{1}}\cap -\bar{\rho}+(1/2)(\bar{\Sigma}-\Sigma)=\emptyset, 
\]
which implies that \eqref{ref-1.4-20} is a twisted NCCR. The statement about the index follows from the fact that $i_{\bar{1}}(\Sl_2)=2$ and Proposition \ref{ref-4.1.6-40}.

\section{Trace rings}
\label{ref-9-93}
In this section we prove Theorem \ref{ref-1.4.9-22}. To
illustrate the proof goes we will first handle the case $(m,n)=(2,3)$
graphically. From this discussion we will obtain a new proof for the fact
that $\TT_{2,3}$ has finite global dimension (see \cite{VdB15}).

Put $G=\PGL_3(k)$, $\bar{G}=\Sl_3(k)$, $V=k^3$. Put $W=\End(V)^{\oplus 2}$. It is easy to see that $W$ is
generic (cfr. \S\ref{ref-5.1-69}).
Let $\bar{T}$ be the diagonal maximal torus in $\bar{G}$ and let $T$
be its image in $\PGL_3(k)$. We have $A=\ker(\bar{T}\r T)\cong
\ZZ/3\ZZ$. In Figure \ref{ref-9.1-97} we have drawn the elements of
$X(\bar{T})$, color coded according the character of $A$
they correspond to. The red weights are those in $X(T)$. We have also indicated the simple
roots $\alpha_{1,2}$ and the fundamental weights  $\phi_{1,2}$ for $\bar{G}$.

The dominant cone has been colored yellow. Finally we have also drawn part of the open regular hexagon 
$-\bar{\rho}+(1/2)\Sigma$.
From
this picture we see that $-\bar{\rho}+(1/2)(\bar{\Sigma}-\Sigma)$ contains no green or blue weights. Hence
either of these colors yields a twisted NCCR by Theorem \ref{ref-1.6.4-30}. The picture shows that  $-\bar{\rho}+(1/2)\bar{\Sigma}$ contains a single
dominant green and a single dominant blue weight. These are precisely the fundamental weights so they correspond to $V$ and $\wedge^2 V=V^\ast$ respectively. 
The corresponding twisted NCCRs are $M(\End(V))\cong \TT_{2,3}$ and $M(\End(V^\ast))\cong \TT_{2,3}^\circ$.
\begin{proof}[Proof of Theorem \ref{ref-1.4.9-22}]
  $Z_{2,2}$ is its own NCCR because it is a polynomial ring \cite[p.20]{Herstein}.
  Hence without
  loss of generality we may assume $(m,n)\neq (2,2)$ which is the same
  as saying that $W=\End(V)^{\oplus m}$ is generic. 

  Inspecting Figure \ref{ref-9.1-97} for the case $(m,n)=(2,3)$ we see
  that in fact the green and blue weights are missed by the entire
  boundary lines of $-\bar{\rho}+(1/2)\bar{\Sigma}$ and not just by
  the faces.  We will show that this pattern persists for $(m,n)$
  arbitrary: the boundary hyperplanes of
  $-\bar{\rho}+(1/2)\bar{\Sigma}$ will always miss certain ``colors''
  and these colors then yield a twisted NCCR by Theorem
  \ref{ref-1.6.4-30}.

Put $G=\PGL_n(k)$,
  $\bar{G}=\Sl_n(k)$. We use the standard associated notations (see \S\ref{ref-1.5-23}). We let
  $\bar{T}$ be the standard maximal torus in $\bar{G}$ given by
  $\{\diag(z_1,\ldots,z_n)\mid z_1\cdots z_n=1\}$ and $T$  its image in $G$.
We let $L_i\in X(\bar{T})$ be the projection $\bar{T}\mapsto G_m:(z_1,\ldots,z_n)\mapsto
z_i$. Then $\sum_i L_i=0$ and in fact $X(\bar{T})=\bigoplus_i \ZZ L_i/(\sum_i L_i)$.

The group $A=\ker(\bar{T}\r T)=Z(\bar{G})$ consists of the diagonal matrices $\diag(\xi,\ldots,\xi)$
where $\xi^n=1$. Hence canonically $X(A)\cong \ZZ/n\ZZ$ where $\bar{1}$ corresponds to the 
inclusion of the $n$-th roots of unity in $G_m$. Keeping the above terminology we will refer
to the image of a weight $\chi$ in $X(A)$ as its ``color'' and denote it by $c(\chi)$. Of course color is preserved by the Weyl group action.
One finds that the color of $L_i$ 
is equal to $\bar{1}$.

As in \eqref{ref-5.2-73} we have
\begin{align*}
\bar{\rho}&=(n-1)/2 L_1+(n-3)/2 L_2+\cdots+(-n+1)/2 L_n\\
&=n/2 L_1+(n-2)/2 L_2+\cdots+(-n+2)/2 L_n
\end{align*}
We conclude 
\begin{equation}
\label{ref-9.1-94}
c(\bar{\rho})=
\begin{cases}
\bar{0}&\text{if $n$ is odd}\\
\overline{n/2}&\text{if $n$ is even}
\end{cases}
\end{equation}
The non-zero weights of $W$ are 
given by the root system $\Phi=A_{n-1}$. I.e.\ they are of the form $(L_i-L_j)_{i\neq j}$ where each such weight
occurs with multiplicity $m$.
Using the description of $\partial\overline{\Sigma}$ in Lemma
\ref{ref-B.1-142} below one sees that the boundary hyperplanes of
$\bar{\Sigma}$ correspond to maximal subroot systems of $\Phi$ and
these are given by $A_{p-1}\times A_{q-1}$ for a partition $n=p+q$
(with $A_0=0$). More precisely, one verifies that up to the action of the Weyl group 
the boundary hyperplanes of $(1/2)\bar{\Sigma}$ are given
by
\[
H_{p,q}= m\rho_{p,q}+\sum_{1\le i<j\le p}\RR (L_i-L_j)+\sum_{p+1\le i<j\le n}\RR (L_i-L_j)
\]
where
\begin{align*}
\rho_{p,q}&=\frac{1}{2}\sum_{i=1,\ldots,p,j=p+1,\ldots,n} (L_i-L_j)\\
&=\frac{q}{2}\sum_{i=1,\ldots,p}L_i-\frac{p}{2}\sum_{i=p+1,\ldots,n}L_i
\end{align*}
from which we compute in the same way as for $\bar{\rho}$:
\begin{equation}
\label{ref-9.2-95}
c(m\rho_{p,q})=
\begin{cases}
\bar{0}&\text{if $m$ even}\\
\bar{0}&\text{if $m$ odd, $p,q$ even}\\
\overline{n/2}&\text{if $m$ odd, $p,q$ odd}\\
\text{undefined}&\text{if $m$ odd, $n$ odd}\\
\end{cases}
\end{equation}
The last possibility is because if both $m,n$ are odd then $m\rho_{p,q}\not\in X(\bar{T})$. 

We have
\begin{equation}
\label{ref-9.3-96}
H_{p,q}=\{\sum_i a_i L_i\mid \sum_{i=1}^p a_i=mq/2,\sum_{i=p+1}^n a_i=- mp/2\}.
\end{equation}
Assume
\[
\chi=\sum_i n_i L_i\in X(\bar{T})\cap H_{p,q}
\]
for $n_i\in \ZZ$. Then there exists $\epsilon\in \RR$ such that for $a_i=n_i-\epsilon$  the conditions on $(a_i)_i$ on the righthand side of \eqref{ref-9.3-96} are true. Hence in particular
\begin{align*}
mq/2+p\epsilon\in \ZZ\\
- mp/2+q\epsilon\in \ZZ
\end{align*}
which implies $m(p^2+q^2)/2\in\ZZ$ which is impossible if $m,n$ are both odd. Hence in this case we are done. So below we assume that $m,n$ are not both odd. So in particular $m\rho_{p,q}\in X(\bar{T})$ and hence
\[
X(\bar{T})\cap H_{p,q}=m\rho_{p,q}+X(\bar{T})\cap \{\sum_i a_i L_i\mid \sum_{i=1}^p a_i=0,\sum_{i=p+1}^n a_i=0\}.
\]
Now assume
\[
\chi'=\sum_i n'_i L_i\in X(\bar{T})\cap \{\sum_i a_i L_i\mid \sum_{i=1}^p a_i=0,\sum_{i=p+1}^n a_i=0\}.
\]
Then putting $a_i=n'_i-\epsilon$ as above we find $p\epsilon\in \ZZ$, $q\epsilon\in\ZZ$. Hence $\epsilon$ may be written as
$
\epsilon=\frac{k}{d}
$
for $d=\gcd(p,q)$ and $k\in \ZZ$. Furthermore we find
\[
c(\chi')=c((p+q)\epsilon)=\overline{\frac{nk}{d}}
\]
so that ultimately using \eqref{ref-9.2-95} we obtain
\[
c(\chi)=
\overline{\frac{nk}{d}}+
\begin{cases}
\bar{0}&\text{$m$ even}\\
\bar{0}&\text{$m$ odd, $n$ even, $2\mid d$}\\
\overline{n/2}&\text{$m$ odd, $n$ even, $2\nmid d$}
\end{cases}
\]
with $d\mid n$, $d\neq n$.  
 One now verifies the following statements.
\begin{enumerate}
\item If $m$ is even then $c(\chi)\neq \bar{1}$.
\item If $m$ is odd, $n$ is even but $n/2$ is odd then $c(\chi)\neq \bar{2}$.
\item If $m$ is odd, $n$ is even but $n/2$ is even then $c(\chi)\neq \bar{1}$.
\end{enumerate}
So as mentioned $c(-\bar{\rho})+c(\chi)$ always misses certain colors, finishing the proof. 
\end{proof}
\begin{figure}[H]
\begin{tikzpicture}
\begin{scope}
\clip(-4.500000,-2.700000) rectangle (4.500000,2.700000);
\path[draw,fill=yellow] (0.000000,0.000000)-- (3.500000,6.062178)-- (-3.500000,6.062178);
\path[draw] (0.000000,0.000000) -- (-7.000000,0.000000);
\path[draw] (0.000000,0.000000) -- (-3.500000,6.062178);
\path[draw] (0.000000,0.000000) -- (3.500000,6.062178);
\path[draw] (0.000000,0.000000) -- (7.000000,0.000000);
\path[draw] (0.000000,0.000000) -- (3.500000,-6.062178);
\path[draw] (0.000000,0.000000) -- (-3.500000,-6.062178);
\path[draw,fill=blue,color=blue] (4.041452,-7.000000) circle [radius=0.100000];
\path[draw,fill=green,color=green] (4.618802,-6.000000) circle [radius=0.100000];
\path[draw,fill=red,color=red] (5.196152,-5.000000) circle [radius=0.100000];
\path[draw,fill=blue,color=blue] (5.773503,-4.000000) circle [radius=0.100000];
\path[draw,fill=green,color=green] (6.350853,-3.000000) circle [radius=0.100000];
\path[draw,fill=red,color=red] (6.928203,-2.000000) circle [radius=0.100000];
\path[draw,fill=blue,color=blue] (7.505553,-1.000000) circle [radius=0.100000];
\path[draw,fill=green,color=green] (8.082904,0.000000) circle [radius=0.100000];
\path[draw,fill=green,color=green] (2.886751,-7.000000) circle [radius=0.100000];
\path[draw,fill=red,color=red] (3.464102,-6.000000) circle [radius=0.100000];
\path[draw,fill=blue,color=blue] (4.041452,-5.000000) circle [radius=0.100000];
\path[draw,fill=green,color=green] (4.618802,-4.000000) circle [radius=0.100000];
\path[draw,fill=red,color=red] (5.196152,-3.000000) circle [radius=0.100000];
\path[draw,fill=blue,color=blue] (5.773503,-2.000000) circle [radius=0.100000];
\path[draw,fill=green,color=green] (6.350853,-1.000000) circle [radius=0.100000];
\path[draw,fill=red,color=red] (6.928203,0.000000) circle [radius=0.100000];
\path[draw,fill=blue,color=blue] (7.505553,1.000000) circle [radius=0.100000];
\path[draw,fill=red,color=red] (1.732051,-7.000000) circle [radius=0.100000];
\path[draw,fill=blue,color=blue] (2.309401,-6.000000) circle [radius=0.100000];
\path[draw,fill=green,color=green] (2.886751,-5.000000) circle [radius=0.100000];
\path[draw,fill=red,color=red] (3.464102,-4.000000) circle [radius=0.100000];
\path[draw,fill=blue,color=blue] (4.041452,-3.000000) circle [radius=0.100000];
\path[draw,fill=green,color=green] (4.618802,-2.000000) circle [radius=0.100000];
\path[draw,fill=red,color=red] (5.196152,-1.000000) circle [radius=0.100000];
\path[draw,fill=blue,color=blue] (5.773503,0.000000) circle [radius=0.100000];
\path[draw,fill=green,color=green] (6.350853,1.000000) circle [radius=0.100000];
\path[draw,fill=red,color=red] (6.928203,2.000000) circle [radius=0.100000];
\path[draw,fill=blue,color=blue] (0.577350,-7.000000) circle [radius=0.100000];
\path[draw,fill=green,color=green] (1.154701,-6.000000) circle [radius=0.100000];
\path[draw,fill=red,color=red] (1.732051,-5.000000) circle [radius=0.100000];
\path[draw,fill=blue,color=blue] (2.309401,-4.000000) circle [radius=0.100000];
\path[draw,fill=green,color=green] (2.886751,-3.000000) circle [radius=0.100000];
\path[draw,fill=red,color=red] (3.464102,-2.000000) circle [radius=0.100000];
\path[draw,fill=blue,color=blue] (4.041452,-1.000000) circle [radius=0.100000];
\path[draw,fill=green,color=green] (4.618802,0.000000) circle [radius=0.100000];
\path[draw,fill=red,color=red] (5.196152,1.000000) circle [radius=0.100000];
\path[draw,fill=blue,color=blue] (5.773503,2.000000) circle [radius=0.100000];
\path[draw,fill=green,color=green] (6.350853,3.000000) circle [radius=0.100000];
\path[draw,fill=green,color=green] (-0.577350,-7.000000) circle [radius=0.100000];
\path[draw,fill=red,color=red] (0.000000,-6.000000) circle [radius=0.100000];
\path[draw,fill=blue,color=blue] (0.577350,-5.000000) circle [radius=0.100000];
\path[draw,fill=green,color=green] (1.154701,-4.000000) circle [radius=0.100000];
\path[draw,fill=red,color=red] (1.732051,-3.000000) circle [radius=0.100000];
\path[draw,fill=blue,color=blue] (2.309401,-2.000000) circle [radius=0.100000];
\path[draw,fill=green,color=green] (2.886751,-1.000000) circle [radius=0.100000];
\path[draw,fill=red,color=red] (3.464102,0.000000) circle [radius=0.100000];
\path[draw,fill=blue,color=blue] (4.041452,1.000000) circle [radius=0.100000];
\path[draw,fill=green,color=green] (4.618802,2.000000) circle [radius=0.100000];
\path[draw,fill=red,color=red] (5.196152,3.000000) circle [radius=0.100000];
\path[draw,fill=blue,color=blue] (5.773503,4.000000) circle [radius=0.100000];
\path[draw,fill=red,color=red] (-1.732051,-7.000000) circle [radius=0.100000];
\path[draw,fill=blue,color=blue] (-1.154701,-6.000000) circle [radius=0.100000];
\path[draw,fill=green,color=green] (-0.577350,-5.000000) circle [radius=0.100000];
\path[draw,fill=red,color=red] (0.000000,-4.000000) circle [radius=0.100000];
\path[draw,fill=blue,color=blue] (0.577350,-3.000000) circle [radius=0.100000];
\path[draw,fill=green,color=green] (1.154701,-2.000000) circle [radius=0.100000];
\path[draw,fill=red,color=red] (1.732051,-1.000000) circle [radius=0.100000];
\path[draw,fill=blue,color=blue] (2.309401,0.000000) circle [radius=0.100000];
\path[draw,fill=green,color=green] (2.886751,1.000000) circle [radius=0.100000];
\path[draw,fill=red,color=red] (3.464102,2.000000) circle [radius=0.100000];
\path[draw,fill=blue,color=blue] (4.041452,3.000000) circle [radius=0.100000];
\path[draw,fill=green,color=green] (4.618802,4.000000) circle [radius=0.100000];
\path[draw,fill=red,color=red] (5.196152,5.000000) circle [radius=0.100000];
\path[draw,fill=blue,color=blue] (-2.886751,-7.000000) circle [radius=0.100000];
\path[draw,fill=green,color=green] (-2.309401,-6.000000) circle [radius=0.100000];
\path[draw,fill=red,color=red] (-1.732051,-5.000000) circle [radius=0.100000];
\path[draw,fill=blue,color=blue] (-1.154701,-4.000000) circle [radius=0.100000];
\path[draw,fill=green,color=green] (-0.577350,-3.000000) circle [radius=0.100000];
\path[draw,fill=red,color=red] (0.000000,-2.000000) circle [radius=0.100000];
\node[] at (0.000000,-2.400000) {$-\bar{\rho}$};
\path[draw,fill=blue,color=blue] (0.577350,-1.000000) circle [radius=0.100000];
\path[draw,fill=green,color=green] (1.154701,0.000000) circle [radius=0.100000];
\path[draw,fill=red,color=red] (1.732051,1.000000) circle [radius=0.100000];
\node[] at (2.032051,1.300000) {$\alpha_2$};
\path[draw,fill=blue,color=blue] (2.309401,2.000000) circle [radius=0.100000];
\path[draw,fill=green,color=green] (2.886751,3.000000) circle [radius=0.100000];
\path[draw,fill=red,color=red] (3.464102,4.000000) circle [radius=0.100000];
\path[draw,fill=blue,color=blue] (4.041452,5.000000) circle [radius=0.100000];
\path[draw,fill=green,color=green] (4.618802,6.000000) circle [radius=0.100000];
\path[draw,fill=green,color=green] (-4.041452,-7.000000) circle [radius=0.100000];
\path[draw,fill=red,color=red] (-3.464102,-6.000000) circle [radius=0.100000];
\path[draw,fill=blue,color=blue] (-2.886751,-5.000000) circle [radius=0.100000];
\path[draw,fill=green,color=green] (-2.309401,-4.000000) circle [radius=0.100000];
\path[draw,fill=red,color=red] (-1.732051,-3.000000) circle [radius=0.100000];
\path[draw,fill=blue,color=blue] (-1.154701,-2.000000) circle [radius=0.100000];
\path[draw,fill=green,color=green] (-0.577350,-1.000000) circle [radius=0.100000];
\path[draw,fill=red,color=red] (0.000000,0.000000) circle [radius=0.100000];
\path[draw,fill=blue,color=blue] (0.577350,1.000000) circle [radius=0.100000];
\node[] at (0.277350,1.300000) {$\phi_2$};
\path[draw,fill=green,color=green] (1.154701,2.000000) circle [radius=0.100000];
\path[draw,fill=red,color=red] (1.732051,3.000000) circle [radius=0.100000];
\path[draw,fill=blue,color=blue] (2.309401,4.000000) circle [radius=0.100000];
\path[draw,fill=green,color=green] (2.886751,5.000000) circle [radius=0.100000];
\path[draw,fill=red,color=red] (3.464102,6.000000) circle [radius=0.100000];
\path[draw,fill=blue,color=blue] (4.041452,7.000000) circle [radius=0.100000];
\path[draw,fill=blue,color=blue] (-4.618802,-6.000000) circle [radius=0.100000];
\path[draw,fill=green,color=green] (-4.041452,-5.000000) circle [radius=0.100000];
\path[draw,fill=red,color=red] (-3.464102,-4.000000) circle [radius=0.100000];
\path[draw,fill=blue,color=blue] (-2.886751,-3.000000) circle [radius=0.100000];
\path[draw,fill=green,color=green] (-2.309401,-2.000000) circle [radius=0.100000];
\path[draw,fill=red,color=red] (-1.732051,-1.000000) circle [radius=0.100000];
\path[draw,fill=blue,color=blue] (-1.154701,0.000000) circle [radius=0.100000];
\path[draw,fill=green,color=green] (-0.577350,1.000000) circle [radius=0.100000];
\node[] at (-0.277350,1.300000) {$\phi_1$};
\path[draw,fill=red,color=red] (0.000000,2.000000) circle [radius=0.100000];
\path[draw,fill=blue,color=blue] (0.577350,3.000000) circle [radius=0.100000];
\path[draw,fill=green,color=green] (1.154701,4.000000) circle [radius=0.100000];
\path[draw,fill=red,color=red] (1.732051,5.000000) circle [radius=0.100000];
\path[draw,fill=blue,color=blue] (2.309401,6.000000) circle [radius=0.100000];
\path[draw,fill=green,color=green] (2.886751,7.000000) circle [radius=0.100000];
\path[draw,fill=red,color=red] (-5.196152,-5.000000) circle [radius=0.100000];
\path[draw,fill=blue,color=blue] (-4.618802,-4.000000) circle [radius=0.100000];
\path[draw,fill=green,color=green] (-4.041452,-3.000000) circle [radius=0.100000];
\path[draw,fill=red,color=red] (-3.464102,-2.000000) circle [radius=0.100000];
\path[draw,fill=blue,color=blue] (-2.886751,-1.000000) circle [radius=0.100000];
\path[draw,fill=green,color=green] (-2.309401,0.000000) circle [radius=0.100000];
\path[draw,fill=red,color=red] (-1.732051,1.000000) circle [radius=0.100000];
\node[] at (-2.032051,1.300000) {$\alpha_1$};
\path[draw,fill=blue,color=blue] (-1.154701,2.000000) circle [radius=0.100000];
\path[draw,fill=green,color=green] (-0.577350,3.000000) circle [radius=0.100000];
\path[draw,fill=red,color=red] (0.000000,4.000000) circle [radius=0.100000];
\path[draw,fill=blue,color=blue] (0.577350,5.000000) circle [radius=0.100000];
\path[draw,fill=green,color=green] (1.154701,6.000000) circle [radius=0.100000];
\path[draw,fill=red,color=red] (1.732051,7.000000) circle [radius=0.100000];
\path[draw,fill=green,color=green] (-5.773503,-4.000000) circle [radius=0.100000];
\path[draw,fill=red,color=red] (-5.196152,-3.000000) circle [radius=0.100000];
\path[draw,fill=blue,color=blue] (-4.618802,-2.000000) circle [radius=0.100000];
\path[draw,fill=green,color=green] (-4.041452,-1.000000) circle [radius=0.100000];
\path[draw,fill=red,color=red] (-3.464102,0.000000) circle [radius=0.100000];
\path[draw,fill=blue,color=blue] (-2.886751,1.000000) circle [radius=0.100000];
\path[draw,fill=green,color=green] (-2.309401,2.000000) circle [radius=0.100000];
\path[draw,fill=red,color=red] (-1.732051,3.000000) circle [radius=0.100000];
\path[draw,fill=blue,color=blue] (-1.154701,4.000000) circle [radius=0.100000];
\path[draw,fill=green,color=green] (-0.577350,5.000000) circle [radius=0.100000];
\path[draw,fill=red,color=red] (0.000000,6.000000) circle [radius=0.100000];
\path[draw,fill=blue,color=blue] (0.577350,7.000000) circle [radius=0.100000];
\path[draw,fill=blue,color=blue] (-6.350853,-3.000000) circle [radius=0.100000];
\path[draw,fill=green,color=green] (-5.773503,-2.000000) circle [radius=0.100000];
\path[draw,fill=red,color=red] (-5.196152,-1.000000) circle [radius=0.100000];
\path[draw,fill=blue,color=blue] (-4.618802,0.000000) circle [radius=0.100000];
\path[draw,fill=green,color=green] (-4.041452,1.000000) circle [radius=0.100000];
\path[draw,fill=red,color=red] (-3.464102,2.000000) circle [radius=0.100000];
\path[draw,fill=blue,color=blue] (-2.886751,3.000000) circle [radius=0.100000];
\path[draw,fill=green,color=green] (-2.309401,4.000000) circle [radius=0.100000];
\path[draw,fill=red,color=red] (-1.732051,5.000000) circle [radius=0.100000];
\path[draw,fill=blue,color=blue] (-1.154701,6.000000) circle [radius=0.100000];
\path[draw,fill=green,color=green] (-0.577350,7.000000) circle [radius=0.100000];
\path[draw,fill=red,color=red] (-6.928203,-2.000000) circle [radius=0.100000];
\path[draw,fill=blue,color=blue] (-6.350853,-1.000000) circle [radius=0.100000];
\path[draw,fill=green,color=green] (-5.773503,0.000000) circle [radius=0.100000];
\path[draw,fill=red,color=red] (-5.196152,1.000000) circle [radius=0.100000];
\path[draw,fill=blue,color=blue] (-4.618802,2.000000) circle [radius=0.100000];
\path[draw,fill=green,color=green] (-4.041452,3.000000) circle [radius=0.100000];
\path[draw,fill=red,color=red] (-3.464102,4.000000) circle [radius=0.100000];
\path[draw,fill=blue,color=blue] (-2.886751,5.000000) circle [radius=0.100000];
\path[draw,fill=green,color=green] (-2.309401,6.000000) circle [radius=0.100000];
\path[draw,fill=red,color=red] (-1.732051,7.000000) circle [radius=0.100000];
\path[draw,fill=green,color=green] (-7.505553,-1.000000) circle [radius=0.100000];
\path[draw,fill=red,color=red] (-6.928203,0.000000) circle [radius=0.100000];
\path[draw,fill=blue,color=blue] (-6.350853,1.000000) circle [radius=0.100000];
\path[draw,fill=green,color=green] (-5.773503,2.000000) circle [radius=0.100000];
\path[draw,fill=red,color=red] (-5.196152,3.000000) circle [radius=0.100000];
\path[draw,fill=blue,color=blue] (-4.618802,4.000000) circle [radius=0.100000];
\path[draw,fill=green,color=green] (-4.041452,5.000000) circle [radius=0.100000];
\path[draw,fill=red,color=red] (-3.464102,6.000000) circle [radius=0.100000];
\path[draw,fill=blue,color=blue] (-2.886751,7.000000) circle [radius=0.100000];
\path[draw,fill=blue,color=blue] (-8.082904,0.000000) circle [radius=0.100000];
\path[draw,fill=green,color=green] (-7.505553,1.000000) circle [radius=0.100000];
\path[draw,fill=red,color=red] (-6.928203,2.000000) circle [radius=0.100000];
\path[draw,fill=blue,color=blue] (-6.350853,3.000000) circle [radius=0.100000];
\path[draw,fill=green,color=green] (-5.773503,4.000000) circle [radius=0.100000];
\path[draw,fill=red,color=red] (-5.196152,5.000000) circle [radius=0.100000];
\path[draw,fill=blue,color=blue] (-4.618802,6.000000) circle [radius=0.100000];
\path[draw,fill=green,color=green] (-4.041452,7.000000) circle [radius=0.100000];
\path[draw,fill=red,thick] (-1.732051,1.000000) circle [radius=0.130000];
\path[draw,fill=red,thick] (1.732051,1.000000) circle [radius=0.130000];
\path[draw,fill=red,thick] (0.000000,2.000000) circle [radius=0.130000];
\path[draw,fill=red,thick] (1.732051,-1.000000) circle [radius=0.130000];
\path[draw,fill=red,thick] (-1.732051,-1.000000) circle [radius=0.130000];
\path[draw,fill=red,thick] (-0.000000,-2.000000) circle [radius=0.130000];
\path[draw,dashed] (-3.464102,0.000000) -- (0.000000,2.000000);
\path[draw,dashed] (0.000000,2.000000) -- (3.464102,0.000000);
\path[draw,dashed] (3.464102,0.000000) -- (3.464102,-4.000000);
\path[draw,dashed] (3.464102,-4.000000) -- (-0.000000,-6.000000);
\path[draw,dashed] (-0.000000,-6.000000) -- (-3.464102,-4.000000);
\path[draw,dashed] (-3.464102,-4.000000) -- (-3.464102,0.000000);
\end{scope}
\node[,align=left] at (-4.700000,-0.500000) {$-\bar{\rho}+(1/2)\Sigma$};
\node[,align=left] at (-3.700000,2.500000) {$X(\bar{T})$};
\end{tikzpicture}
\caption{Relevant data for $Z_{2,3}$.}
\label{ref-9.1-97}
\end{figure}
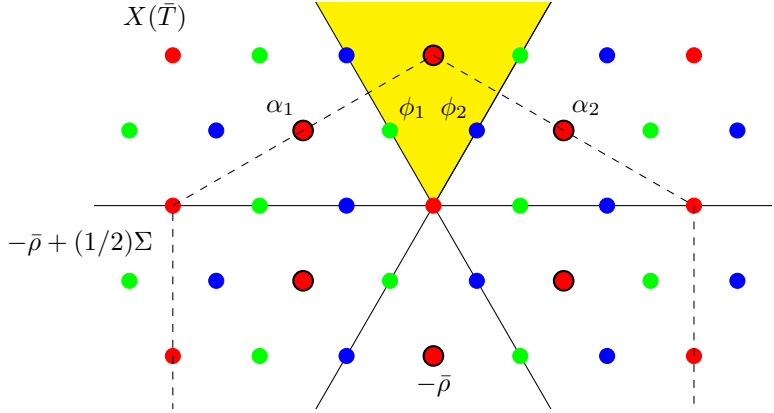
\section{Some counter examples}
\label{ref-10-98}
\subsection{Example}
\label{ref-10.1-99}
We show that the quasi-symmetry hypothesis in Theorem \ref{ref-1.6.2-28} 
cannot be dropped by
giving an example where there is no NCCR given
by modules of covariants. Let $T=G_m^2$ be the standard two
dimensional torus and identify its character group  $X(T)$ with
$\ZZ^2$. Let $W$ be the generic $T$-representation with weights
$(1,1),(-1,0),(0,-1),(-3,-3),(3,0),(0,3)$.  Using \cite[Cor.\
4.1.1]{Vdb1} one verifies that the (indecomposable) modules of covariants
which are Cohen-Macaulay correspond to the weights marked by black
dots in the following picture:
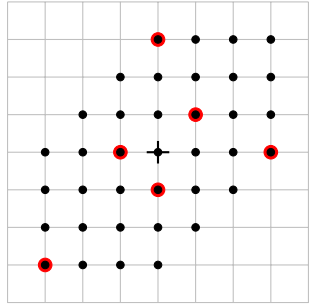
\[
\begin{tikzpicture}[scale=0.500000]
\path[draw,color=lightgray] (-4.000000,-4.000000) grid (4.000000,4.000000);
\path[draw,fill=red,color=red] (1.000000,1.000000) circle [radius=0.180000];
\path[draw,fill=red,color=red] (-1.000000,0.000000) circle [radius=0.180000];
\path[draw,fill=red,color=red] (0.000000,-1.000000) circle [radius=0.180000];
\path[draw,fill=red,color=red] (-3.000000,-3.000000) circle [radius=0.180000];
\path[draw,fill=red,color=red] (3.000000,0.000000) circle [radius=0.180000];
\path[draw,fill=red,color=red] (0.000000,3.000000) circle [radius=0.180000];
\path[draw,color=black,thick] (0.000000,-0.300000) -- (0.000000,0.300000);
\path[draw,color=black,thick] (-0.300000,0.000000) -- (0.300000,0.000000);
\path[draw,fill=black,color=black] (-3.000000,-3.000000) circle [radius=0.100000];
\path[draw,fill=black,color=black] (-3.000000,-2.000000) circle [radius=0.100000];
\path[draw,fill=black,color=black] (-3.000000,-1.000000) circle [radius=0.100000];
\path[draw,fill=black,color=black] (-3.000000,0.000000) circle [radius=0.100000];
\path[draw,fill=black,color=black] (-2.000000,-3.000000) circle [radius=0.100000];
\path[draw,fill=black,color=black] (-2.000000,-2.000000) circle [radius=0.100000];
\path[draw,fill=black,color=black] (-2.000000,-1.000000) circle [radius=0.100000];
\path[draw,fill=black,color=black] (-2.000000,0.000000) circle [radius=0.100000];
\path[draw,fill=black,color=black] (-2.000000,1.000000) circle [radius=0.100000];
\path[draw,fill=black,color=black] (-1.000000,-3.000000) circle [radius=0.100000];
\path[draw,fill=black,color=black] (-1.000000,-2.000000) circle [radius=0.100000];
\path[draw,fill=black,color=black] (-1.000000,-1.000000) circle [radius=0.100000];
\path[draw,fill=black,color=black] (-1.000000,0.000000) circle [radius=0.100000];
\path[draw,fill=black,color=black] (-1.000000,1.000000) circle [radius=0.100000];
\path[draw,fill=black,color=black] (-1.000000,2.000000) circle [radius=0.100000];
\path[draw,fill=black,color=black] (0.000000,-3.000000) circle [radius=0.100000];
\path[draw,fill=black,color=black] (0.000000,-2.000000) circle [radius=0.100000];
\path[draw,fill=black,color=black] (0.000000,-1.000000) circle [radius=0.100000];
\path[draw,fill=black,color=black] (0.000000,0.000000) circle [radius=0.100000];
\path[draw,fill=black,color=black] (0.000000,1.000000) circle [radius=0.100000];
\path[draw,fill=black,color=black] (0.000000,2.000000) circle [radius=0.100000];
\path[draw,fill=black,color=black] (0.000000,3.000000) circle [radius=0.100000];
\path[draw,fill=black,color=black] (1.000000,-2.000000) circle [radius=0.100000];
\path[draw,fill=black,color=black] (1.000000,-1.000000) circle [radius=0.100000];
\path[draw,fill=black,color=black] (1.000000,0.000000) circle [radius=0.100000];
\path[draw,fill=black,color=black] (1.000000,1.000000) circle [radius=0.100000];
\path[draw,fill=black,color=black] (1.000000,2.000000) circle [radius=0.100000];
\path[draw,fill=black,color=black] (1.000000,3.000000) circle [radius=0.100000];
\path[draw,fill=black,color=black] (2.000000,-1.000000) circle [radius=0.100000];
\path[draw,fill=black,color=black] (2.000000,0.000000) circle [radius=0.100000];
\path[draw,fill=black,color=black] (2.000000,1.000000) circle [radius=0.100000];
\path[draw,fill=black,color=black] (2.000000,2.000000) circle [radius=0.100000];
\path[draw,fill=black,color=black] (2.000000,3.000000) circle [radius=0.100000];
\path[draw,fill=black,color=black] (3.000000,0.000000) circle [radius=0.100000];
\path[draw,fill=black,color=black] (3.000000,1.000000) circle [radius=0.100000];
\path[draw,fill=black,color=black] (3.000000,2.000000) circle [radius=0.100000];
\path[draw,fill=black,color=black] (3.000000,3.000000) circle [radius=0.100000];
\end{tikzpicture}
\]
(the red dots are the weights of $W$).
Assume that $M=\bigoplus_{\chi\in S} M(\chi)$ with $S\subset X(T)$ yields a NCCR for $(SW)^T$. Then
by definition  $M(\chi_1\chi_2^{-1})$ is Cohen-Macaulay for all $\chi_1,\chi_2\in S$. 
Moreover by Proposition \ref{ref-3.5-35} $S$ must be maximal with respect to this property.
It is then an easy verification that, up to translation and reflection
around the line $y=x$, there are only two possibilities for such maximal $S$:
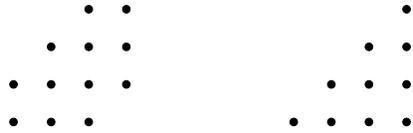
\[
\begin{tikzpicture}[scale=0.500000]
\path[draw,fill=black,color=black] (0.000000,0.000000) circle [radius=0.100000];
\path[draw,fill=black,color=black] (1.000000,0.000000) circle [radius=0.100000];
\path[draw,fill=black,color=black] (2.000000,0.000000) circle [radius=0.100000];
\path[draw,fill=black,color=black] (0.000000,1.000000) circle [radius=0.100000];
\path[draw,fill=black,color=black] (1.000000,1.000000) circle [radius=0.100000];
\path[draw,fill=black,color=black] (2.000000,1.000000) circle [radius=0.100000];
\path[draw,fill=black,color=black] (3.000000,1.000000) circle [radius=0.100000];
\path[draw,fill=black,color=black] (1.000000,2.000000) circle [radius=0.100000];
\path[draw,fill=black,color=black] (2.000000,2.000000) circle [radius=0.100000];
\path[draw,fill=black,color=black] (3.000000,2.000000) circle [radius=0.100000];
\path[draw,fill=black,color=black] (2.000000,3.000000) circle [radius=0.100000];
\path[draw,fill=black,color=black] (3.000000,3.000000) circle [radius=0.100000];
\end{tikzpicture}
 \qquad  \qquad \qquad 
\begin{tikzpicture}[scale=0.500000]
\path[draw,fill=black,color=black] (0.000000,0.000000) circle [radius=0.100000];
\path[draw,fill=black,color=black] (1.000000,0.000000) circle [radius=0.100000];
\path[draw,fill=black,color=black] (2.000000,0.000000) circle [radius=0.100000];
\path[draw,fill=black,color=black] (3.000000,0.000000) circle [radius=0.100000];
\path[draw,fill=black,color=black] (1.000000,1.000000) circle [radius=0.100000];
\path[draw,fill=black,color=black] (2.000000,1.000000) circle [radius=0.100000];
\path[draw,fill=black,color=black] (3.000000,1.000000) circle [radius=0.100000];
\path[draw,fill=black,color=black] (2.000000,2.000000) circle [radius=0.100000];
\path[draw,fill=black,color=black] (3.000000,2.000000) circle [radius=0.100000];
\path[draw,fill=black,color=black] (3.000000,3.000000) circle [radius=0.100000];
\end{tikzpicture}
\]
For a $\ZZ$-graded vector space $V=\oplus_i V_i$ we define its Poincare series as $H(V,t)=\sum_i \dim V_i t^i$.
If $\Lambda=\End(\bigoplus_{\chi\in S} M(\chi))=\bigoplus_{\chi_1,\chi_2\in S} M(\chi_1\chi_2^{-1})$ has finite global dimension then
by considering the minimal projective resolutions of the graded simple $\Lambda$-modules one obtains
that the matrix $H(M(\chi_1\chi_2^{-1}),t)_{\chi_1,\chi_2\in S}\in M_{|S|}(\ZZ[[t]])$ has an inverse with polynomial entries. 
One may check\footnote{We have used Mathematica to do the computations.}
that this property does not hold for the two possible sets $S$ indicated above.
 Thus, the global dimension of the endomorphism rings cannot be finite and we do not obtain a NCCR in this way.

Note that $(SW)^T$ is the hypersurface $k[x_1,x_2,x_3,x_4,x_5]/(x_1^3x_2-x_3x_4x_5)$. It is shown in \cite{SVdB4} that $(SW)^T$ has a NCCR 
given by the direct sum of 12 rank one modules and a single rank two module.
\subsection{Example}
\label{ref-10.2-100}
In this section we give an example of an $\Sl_2$-representation where
the corresponding invariant ring has no NCCR.  We let $G=\Sl_2(k)=\Sl(V)$ with $V=k^2$
and $W=V^4$. Then $(SW)^G\cong k[x_1,\ldots,x_6]/(x_1^2+\cdots+x_6^2)$. Now
we use the fact that an isolated hypersurface singularity of dimension $\ge 4$
never has a NCCR. In the case of even dimension this is proved in
\cite{Dao}. The case of odd dimension (relevant for this example) is treated in Appendix \ref{Dao}. 
\subsection{Example}\label{ref-10.3-101}
We know that $\TT_{3,2}$ is a twisted NCCR for $Z_{3,2}$ (see \S\ref{ref-1.4.5-21}).
We will now show that $Z_{3,2}$ does not admit an ordinary NCCR. 

Let us denote $S=Z_{3,2}$ and let $\ol p\in \Spec S$ be the
(non-closed) point corresponding to simultaneously diagonalizable
matrices.  We claim that the localization $S_{\ol p}$ of $S$ at $\ol
p$ does not have a NCCR. This is clearly sufficient.

Since $S$ is a factorial hypersurface (see e.g.~\cite[Theorem 5.3]{Leb3} and \cite[Theorem 22]{Formanek}), $S_{\ol p}$ is a factorial local hypersurface. It is $3$-dimensional and it has an isolated singularity. 
Therefore $S_{\ol p}$ does not have an NCCR by \cite[Theorem 1.2(1)]{Dao}.
\section{Proofs of the main results: the general case}
\label{ref-11-102}
\subsection{Preliminaries}
We use the notations which were introduced in \S\ref{ref-1.5-23} (in particular
the symbols $T,B,G,G_e,\Wscr,\Phi,\Phi^{\pm },W, (\beta_i)_i,\Sigma$). \emph{Throughout we assume that $G$ is connected} so that $G=G_e$.
 For more unexplained
notation and terminology  regarding root systems we refer to
Appendix \ref{ref-D-147}. 

Let
$\Ascr=\mod(G,SW)$ be the category of finitely generated
$G$-equivariant $SW$-modules. 
By Lemma \ref{ref-4.1.1-36}(4) $\gldim \Ascr=d$. For $\chi\in X(T)^+$ we write $P_\chi=V(\chi)\otimes_k
R$. By Lemma \ref{ref-4.1.1-36}(1) $\Ascr$ has a distinguished set of indecomposable projective
generators given by $P_\chi$ for $\chi\in X(T)^+$, as well as a
distinguished set of simple objects $S_\chi=V(\chi)\otimes_k SW/SW_{>0}$ also with $\chi\in X(T)^+$. The
projectives and simples are dual in the following sense
\[
\Ascr(P_{\chi_1},S_{\chi_2})=\delta_{\chi_1,\chi_2}\cdot k\quad \text{for $\chi_1,\chi_2\in X(T)^+$}.
\]
Note that we have
\[
\Ascr(P_{\chi_1},P_{\chi_2})=M(\Hom_k(V(\chi_1),V(\chi_2))).
\]
Fix a finite subset ${\Lscr}$ of $X(T)^+$
and put 
\[
P_{\Lscr}=\bigoplus_{\chi\in {\Lscr}} P_\chi,
\]
\[
\Lambda_{\Lscr}=\Ascr(P_{\Lscr},P_{\Lscr}).
\]
We want to find conditions on ${\Lscr}$ under which one has $\gldim \Lambda_{\Lscr}<\infty$. For $\chi\in X(T)^+$ put
\[
\tilde{P}_{{\Lscr},\chi}=\Ascr(P_{\Lscr},P_\chi).
\]
This is a right projective $\Lambda_{\Lscr}$-module if $\chi\in \Lscr$. Similarly we put
\[
\tilde{S}_{{\Lscr},\chi}=\Ascr(P_{\Lscr},S_\chi).
\]
The
graded simple right $\Lambda_{\Lscr}$-modules are of the form
$\tilde{S}_{{\Lscr},\chi}$ for $\chi\in {\Lscr}$. Note that if $\chi\not\in {\Lscr}$ then $\tilde{S}_{{\Lscr},\chi}=0$.
\begin{lemmas}
\label{ref-11.1.1-103}
The ring $\Lambda_{\Lscr}$ has finite global dimension if and only for all $\chi\in X(T)^+$ one has
\begin{equation}
\label{ref-11.1-104}
\pdim_{\Lambda_{\Lscr}} \tilde{P}_{{\Lscr},\chi}<\infty.
\end{equation}
\end{lemmas}
\begin{proof} The $\Rightarrow$ direction is trivial so let us
  consider the $\Leftarrow$-direction. So we assume that
  \eqref{ref-11.1-104} holds and we have to prove that $\pdim_{\Lambda_{\Lscr}} \tilde{S}_{{\Lscr},\chi}<\infty$ for $\chi\in {\Lscr}$.

The Koszul complex gives us a resolution of $S_\chi$:
\[
0\r P_d\r \cdots\r P_i \r \cdots \r P_{0}\r S_\chi\r 0
\]
with 
\[
P_i=\bigoplus_{\mu} P_\mu^{m_{\mu,i}}
\]
where $m_{\mu,i}$ is the multiplicity of $V(\mu)$ in $V(\chi)\otimes_k \wedge^i W$.

Applying $\Ascr(P_{\Lscr},-)$ we get an exact sequence in $\mod(\Lambda_{\Lscr}^\circ)$
\[
0\r \tilde{P}_d \r \cdots \r \tilde{P}_i  \r \cdots \r  \tilde{P}_{0}\r \tilde{S}_{{\Lscr},\chi}\r 0
\]
with
\begin{align*}
\tilde{P}_i&=\Ascr(P_{\Lscr},P_i)\\
&=\bigoplus_{\mu} \tilde{P}_{{\Lscr},\mu}^{m_{\mu,i}}.
\end{align*}
Since by \eqref{ref-11.1-104} each of the $\tilde{P}_i$ has finite projective dimension
over $\Lambda_{{\Lscr}}$, the same holds for $\tilde{S}_{{\Lscr},\chi}$.
\end{proof}
\subsection{Creating complexes}
\label{ref-11.2-105}
By $Y(T)$ we denote the group of one parameter subgroups of $T$. 
We let $Y(T)_\RR^-$ be the subset of $Y(T)_\RR$  consisting of all $\lambda$
such that for all $\rho\in \Phi^+$ we have $\langle \lambda ,\rho\rangle \le 0$.
Since $\Phi^-=-\Phi^+$ this  implies $\langle\lambda,\rho\rangle \ge 0$ for all $\rho\in \Phi^-$.
We also put
$Y(T)^-=Y(T)_\RR^-\cap Y(T)$.

For $0\neq\lambda\in Y(T)$  we put
\[
Z_\lambda=\{x\in X\mid \lim_{t\r 0} \lambda(t)x\text{ exists }\},
\]
\[
Q_\lambda=\{g\in G\mid \lim_{t\r 0} \lambda(t)g\lambda(t)^{-1}\text{ exists }\}.
\]
Then $Z_\lambda$ is a linear subspace of $X$ cut out in $X\cong
W^\ast$ by the subspace $K_\lambda$ of~$W$ spanned by the weight
vectors $w_j$ such that $\langle \lambda,\beta_j\rangle> 0$. Moreover
$Q_\lambda$ is the parabolic subgroup of $G$ containing $T$ and having roots
$\rho\in \Phi$ such that $\langle \lambda,\rho\rangle\ge 0$. If $\lambda\in Y(T)^-$ this implies $B\subset Q_\lambda$
and then $Z_\lambda$ is
preserved by $B$.

The 
descriptions of $Q_\lambda,Z_\lambda$ using roots and weights still make sense if $\lambda\in Y(T)_\RR$. 
Note that if $\lambda\in Y(T)^-_\RR$ then there is
always a $\lambda'\in Y(T)^-$  such that $Z_\lambda=Z_{\lambda'}$.
So that if $\lambda\in Y(T)_\RR^-$ it is still true that $Q_\lambda$ contains $B$ and $Z_\lambda$ is preserved by $Q_\lambda$ and hence by $B$.

For $\lambda\in Y(T)_\RR^-$ we consider the usual ``Springer type'' diagram 
\begin{equation}
\label{ref-11.2-106}
\xymatrix{
G\times^{B} Z_{\lambda}\ar[d]\ar@{^(->}[r] & G\times^{B} X\ar[d]\ar[r] & G/B\ar[d]\\
GZ_{\lambda}\ar[r]& X\ar[r] &\bullet
}
\end{equation}
Denote the category equivalence from $B$-representations to
$G$-equivariant bundles on $G/B$ by $\widetilde{?}$. The inverse is
given by taking the fiber in $e:=[B]\in G/B$.

Since $G\times^ {B} X\r G/B$ and $G\times^{B} Z_{\lambda}\r G/B$ are
vector bundles we see that the left most top arrow in \eqref{ref-11.2-106} is obtained by applying
$\underline{\Spec}$ to the sheaves of $\Oscr_{G/B}$-algebras
\[
SW\otimes_k \Oscr_{G/B}\r S_{G/B}((W/K_{\lambda})\,\tilde{}\,).
\]
We have the corresponding Koszul resolution of sheaves of  $\Oscr_{G/B}$-modules:
\[
0\r \wedge^{d_{\lambda}} K_{\lambda}\,\tilde{}\,\otimes_{k} SW
\r \wedge^{d_{\lambda}-1} K_{\lambda}\,\tilde{}\,\otimes_{k} SW
\r
\cdots
\r 
\Oscr_{G/B}\otimes_k SW
\r
S_{G/B}((W/K_{\lambda})\,\tilde{}\,)
\r
0
\]
for $d_{\lambda}=\dim_k K_{\lambda}$. 

Let $\chi\in X(T)^+$. Then we still have an exact sequence
\begin{multline*}
0\r (\chi\otimes_k \wedge^{d_{\lambda}} K_{\lambda})\,\tilde{}\,\otimes_{k} SW
\r (\chi\otimes_k \wedge^{d_{\lambda}-1} K_{\lambda})\,\tilde{}\,\otimes_{k} SW
\r
\cdots\\\cdots
\r 
\tilde{\chi}\otimes_k SW
\r
\tilde{\chi}\otimes_k S_{G/B}((W/K_{\lambda})\,\tilde{}\,)
\r
0.
\end{multline*}
Then by Theorem \ref{ref-A.3-137}  we get a $G$-equivariant quasi-isomorphism. 
\begin{equation}
\label{ref-11.3-107}
C_{\lambda,\chi}\overset{\text{def}}{=}\left(\bigoplus_{p\le 0,q\ge 0} H^q(G/B,(\chi\otimes_k \wedge^{-p} K_{\lambda})\,\tilde{}\,)\otimes_{k} SW[-p-q],d\right)
\cong 
R\Gamma(G/B,\tilde{\chi}\otimes_k S_{G/B}((W/K_{\lambda})\,\tilde{}\,))
\end{equation}
where $d$ is obtained from a ``horizontal twisted differential'' (see 
Appendix \ref{ref-A-130}).
 We have also used here that the $G$-equivariant and non-equivariant
$R\Gamma$ coincide by \cite[Lemma 1.5.9]{VV3}.

As above we fix  a subset ${\Lscr}$ of $X(T)^+$. 
We denote by $C_{{\Lscr},\lambda,\chi}$ the complex of right $\Lambda_{\Lscr}$-modules
given by applying $\Ascr(P_{\Lscr},-)$ to $C_{\lambda,\chi}$.

We say that $\chi\in X(T)^+$ is separated from ${\Lscr}$ by $\lambda\in Y(T)^-_\RR$ if 
\[
\langle \lambda,\chi\rangle <\langle\lambda,\mu\rangle\qquad \text{for every $\mu \in {\Lscr}$.}
\]
If there exists $w\in \Wscr$
such that $\mu=w\ast \chi:=w(\chi+\bar{\rho})-\bar{\rho}$ is dominant (see Appendix \ref{ref-D-147})  then we write $\chi^+=\mu$.

\begin{lemmas}
\label{ref-11.2.1-108}
 Assume that $\chi\in X(T)^+$ is separated from ${\Lscr}$ by $\lambda\in Y(T)_\RR^-$. Then $C_{{\Lscr},\lambda,\chi}$ is acyclic.
Furthermore, forgetting the differential $C_{{\Lscr},\lambda,\chi}$ is a sum of right $\Lambda_{\Lscr}$-modules of the form $\tilde{P}_{{\Lscr},\mu}$ where
the $\mu$ are among the weights
\[
(\chi+\beta_{i_1}+\beta_{i_2}+\cdots+\beta_{i_{-p}})^+
\]
(with each such expression occurring at most once)
where $\{i_1,\ldots,i_{-p}\}\subset\{1,\ldots,d\}$, $i_j\neq i_{j'}$ for
$j\neq j'$ and $\langle\lambda,\beta_{i_j}\rangle>0$. 
\end{lemmas}
\begin{proof} 
The claim about the $\tilde{P}_{{\Lscr},\mu}$ that appear is a straightforward application of Bott's theorem
after filtering the $B$-representations $\chi\otimes_k \wedge^{-p}K_\lambda$ by $T$-representations.
So we only have to prove acyclicity.  In other words we have to prove that the righthand side of \eqref{ref-11.3-107}
is acyclic after applying $\Ascr(P_{{\Lscr}},-)$. This amounts to showing 
that the simple
$G$-representations that appear as summand of 
\[
H^\ast(G/B,\tilde{\chi}\otimes_k S((W/K_{\lambda})\,\tilde{}\,))
\]
are not of the form $V(\mu)$ for $\mu\in {\Lscr}$.

The weights of $\chi\otimes_k S(W/K_{\lambda})$ are of
the form
\[
\chi+\beta_{i_1}+\cdots+\beta_{i_q}
\]
where $i_j\in\{1,\ldots,d\}$ (repetitions are allowed) and $\langle\lambda,\beta_{i_j}\rangle \le 0$.
By Bott's theorem it follows that $H^\ast(G/B,\tilde{\chi}\otimes_k S((W/K_{\lambda})\,\tilde{}\,))$ is a direct sum of representations
of the form
\[
V((\chi+\beta_{i_1}+\cdots+\beta_{i_q})^+).
\]
We have
\[
\langle\lambda ,\chi+\beta_{i_1}+\cdots+\beta_{i_q}\rangle\le \langle \lambda,\chi\rangle.
\]
Thus it suffices to show that if $\mu\in X(T)$ and $\mu^+$ exists then $\langle
\lambda ,\mu^+\rangle\le
\langle\lambda,\mu\rangle$. This follows immediately from Corollary \ref{ref-D.3-151} (with $y=-\lambda$, $x=\chi$).
\end{proof}
\subsection{Proof of Theorem \ref{ref-1.5.1-24}} 
\label{ref-11.3-109}
Put 
$
\Gamma=-\bar{\rho}+\left\{\sum_i a_i \beta_i\mid a_i\le 0\right\}+\Delta
$.
For $\chi\in \Gamma$ let
$r_\chi\ge 0$ be minimal with respect to the property $\chi\in
-\bar{\rho}+r_\chi\bar{\Sigma}+\Delta$. Note that $\chi\not\in -\bar{\rho}+r_\chi\Sigma+\Delta$
for otherwise we could reduce $r_\chi$. 

For $\chi\in \Gamma$ let $p_\chi$ be the minimal number of $a_i$ equal
to $-r_\chi$ among all ways of writing $\chi=-\bar{\rho}+\sum_i a_i\beta_i+\delta$ with
$a_i\in [-r_\chi,0]$, $\delta\in \Delta$.  The following properties follow directly  from the definitions. 
\begin{enumerate}
\item Both $r_{\chi}$ and $p_\chi$ depend only on the $\Wscr$-orbit of $\chi$ for the $\ast$-action.
\item If $\chi\in ]\chi',\chi''[$ and $r_\chi=r_{\chi'}=r_{\chi''}$ then $p_{\chi}\le \min(p_{\chi'},p_{\chi''})$.
\end{enumerate}
Assume $\gldim\Lambda_{\Lscr}=\infty$. Then by Lemma \ref{ref-11.1.1-103} there is some $\chi\in X(T)^+$ such
that $\pdim_{\Lambda_{\Lscr}} \tilde{P}_{{\Lscr},\chi}=\infty$. Then by Lemma \ref{ref-11.3.1-113}
below
$\chi$ must be in $\Gamma$ for 
otherwise $\tilde{P}_{{\Lscr},\chi}=0$. 

We pick $\chi$ such that first $r_\chi$ is minimal and then $p_\chi$ is minimal. We have $r_\chi\ge 1$ (for otherwise $\chi\in {\Lscr}$ and hence $\pdim \tilde{P}_{{\Lscr},\chi}=0$). 
We find by Lemma \ref{ref-C.2-146} below (changing the sign of $\lambda$) that there exists $\lambda$ such that for all $\mu\in -\bar{\rho}+r_\chi\Sigma+\Delta$
we have $\langle \lambda, \chi\rangle>\langle \lambda,\mu\rangle$. Hence also $\langle \lambda, \chi+\bar{\rho}\rangle>\langle \lambda,\mu+\bar{\rho}\rangle$. Let $w\in \Wscr$ be such that $w\lambda$ is dominant. Then since $r_\chi\Sigma+\Delta$
is $\Wscr$-invariant we still have for all  $\mu\in -\bar{\rho}+r_\chi\Sigma+\Delta$: $\langle w\lambda, w(\chi+\bar{\rho})\rangle>\langle w\lambda,\mu+\bar{\rho}\rangle$. Moreover by Corollary \ref{ref-D.3-151} below we also have 
$\langle w\lambda, w(\chi+\bar{\rho})\rangle\le  \langle w\lambda, \chi+\bar{\rho}\rangle$. Replacing $\lambda$ then by 
$-w\lambda$ we find $\lambda\in Y(T)^-_\RR$ such that $\langle \lambda, \chi+\bar{\rho}\rangle<\langle \lambda,\mu+\bar{\rho}\rangle$, so that finally we have shown
\begin{equation}
\label{ref-11.4-110}
\forall \mu\in -\bar{\rho}+r_\chi \Sigma+\Delta:\langle \lambda,\chi\rangle < \langle\lambda,\mu\rangle
\end{equation}
Adding the boundary this becomes
\begin{equation}
\label{ref-11.5-111}
\forall \mu\in -\bar{\rho}+r_\chi \bar{\Sigma}+\Delta:\langle \lambda,\chi\rangle \le \langle\lambda,\mu\rangle
\end{equation}

Since $\Lscr\subset -\bar{\rho}+\Sigma+\Delta\subset -\bar{\rho}+r_\chi \Sigma+\Delta$ we obtain in particular $\chi$ is separated from ${\Lscr}$ by $\lambda$ 
and hence by Lemma \ref{ref-11.2.1-108} we
have an exact sequence $C_{{\Lscr},\lambda,\chi}$.
Let 
\[
\mu=(\chi+\beta_{i_1}+\beta_{i_2}+\cdots+\beta_{i_{-p}})^+
\]
as in Lemma \ref{ref-11.2.1-108}. If $p=0$ then $\chi^+=\chi$ and hence $\tilde{P}_{{\Lscr},\mu}=\tilde{P}_{{\Lscr},\chi}$.
If $p<0$ then we claim that either $r_\mu<r_\chi$, or else $p_\mu<p_\chi$. To start put
\begin{equation}
\label{ref-11.6-112}
\mu'=\chi+\beta_{i_1}+\beta_{i_2}+\cdots+\beta_{i_{-p}}.
\end{equation}
By Claim (1) above it is sufficient to prove that either $r_{\mu'}<r_\chi$ or else $p_{\mu'}<p_\chi$. This follows easily from
the following observation
\begin{enumerate}
\setcounter{enumi}{2}
\item Write $\chi=-\bar{\rho}+\sum_i a_i\beta_i+\delta$ with $a_i\in [-r_\chi,0]$, $\delta\in \Delta$. If
  $\langle \lambda,\beta_i\rangle>0$ then $a_i=-r_\chi$.
\end{enumerate}
If this claim is false then there is an $\epsilon>0$ such that
$\chi-\epsilon\beta_i\in -\bar{\rho}+r_\chi\bar{\Sigma}+\Delta$ but
$\langle\lambda,\chi-\epsilon\beta_i\rangle=\langle\lambda,\chi\rangle-\epsilon\langle
\lambda,\beta_i\rangle<\langle\lambda,\chi\rangle$ which  contradicts \eqref{ref-11.5-111}.

So we conclude that the indecomposable projective right $\Lambda_{\Lscr}$-modules $\tilde{P}_{{\Lscr},\mu}$ occurring in $C_{{\Lscr},\lambda,\chi}$, which are different from the single copy of $\tilde{P}_{{\Lscr},\chi}$,
satisfy either $r_{\mu}<r_{\chi}$ or else $p_{\mu}<p_\chi$. By the minimality assumptions on $\chi$ we  have $\pdim \tilde{P}_{{\Lscr},\mu}<\infty$. This implies that $\pdim \tilde{P}_{{\Lscr},\chi}<\infty$
as well, which is a contradiction.
\begin{lemmas}
\label{ref-11.3.1-113}
Assume that $\Ascr(P_{\Lscr},P_\chi)\neq 0$. Then $\chi\in \Gamma$.
\end{lemmas}
\begin{proof}
If $\Ascr(P_{\Lscr},P_\chi)\neq 0$ then some $V(\mu)$ for $\mu\in {\Lscr}$ is a summand of $V(\chi)\otimes S^dW$ for some $d$. This is equivalent to $V(\chi)^\ast$ being a summand of
some $V(\mu)^\ast\otimes S^dW$. We have $V(\mu)^\ast=V(\mu^\ast)$ with $\mu^\ast=-w_0\mu$ where $w_0$ is the longest element of $\Wscr$. So in particular
$\mu^\ast\in -\Gamma-2\bar{\rho}$.

Since the weights of $S^dW$ are in $\{\sum_i a_i\beta_i\mid a_i\ge
0\}$ we conclude by \cite[Ex.\ 25.33]{FH} that $V(\mu^\ast)\otimes
S^dW$ is a sum of representations of the form $V(\theta)$ with
$\theta\in(-\Gamma-2\bar{\rho})\cap X(T)^+$.  Hence if
$V(\chi)^\ast=V(-w_0\chi)$ appears as a summand of $V(\mu^\ast)\otimes
S^dW$ then $-w_0\chi\in -\Gamma-2\bar{\rho}$ and hence $\chi\in \Gamma$.
\end{proof}
\begin{remarks}
\label{ref-11.3.2-114}
The proof Theorem \ref{ref-1.5.1-24} may be converted into a kind of
algorithm for recognizing algebras of covariants $\Lambda_{\Lscr}$ that
have finite global dimension. It is only a pseudo-algorithm in the
sense that \emph{if} it gives a positive answer then definitely
$\gldim \Lambda_{\Lscr}<\infty$ but the algorithm is not guaranteed to
detect every $\Lambda_{\Lscr}$ of finite global dimension.
Nonetheless the algorithm appears to be quite effective in practice
and for small examples it can even be carried out manually.
Furthermore we have programmed it for two-dimensional tori and in that
case it has been very useful in our investigations.

The basis of the algorithm is to verify  \eqref{ref-11.1-104} for all $\chi\in X(T)^+$. 
If we have verified \eqref{ref-11.1-104} for a certain finite set of weights $\chi\in {\Lscr}'$ 
(initially ${\Lscr}={\Lscr}'$) then
we attempt to enlarge this set using Lemma \ref{ref-11.2.1-108} (as in the proof of Theorem \ref{ref-1.5.1-24}). If this turns out to be impossible then the algorithm returns no answer. Otherwise we attempt to continue enlarging ${\Lscr}'$ until
we arrive at a situation where 
\begin{equation}
\label{ref-11.7-115}
{\Lscr}'\supset X(T)^+\cap (-\bar{\rho} +r\Sigma+\Delta)\supset {\Lscr}
\end{equation}
for a suitable
$r\ge 1$ and a suitable bounded closed convex $\Wscr$-invariant set $\Delta$ (usually we may take $\Delta=\emptyset$).
In that case we return a positive answer.
We may really stop at this stage since now we may proceed as in the proof of Theorem \ref{ref-1.5.1-24}
to enlarge ${\Lscr}'$ to $X(T)^+$.

There are situations where we can keep enlarging ${\Lscr}'$ without
\eqref{ref-11.7-115} ever becoming true.  It is not so clear how
to recognize this situation algorithmically. So we simply set a bound on
the running time  and return no answer if that bound
is reached.
\end{remarks}
\subsection{Proof of Theorem \ref{ref-1.3.1-5}}
\label{ref-11.4-116}
 By
embedding $X$ as a closed subvariety in a $G$-representation we may by Theorem \ref{ref-4.3.1-48}
reduce to the case that $X$ is itself a representation. 

We now invoke Theorem \ref{ref-1.5.1-24} to obtain a $G_e$-representation 
$U_e$ containing a trivial direct summand such that $M_{G_e,X}(\End(U_e))$ has finite global dimension.
Using Lemmas \ref{ref-4.5.1-63},\ref{ref-4.5.2-67} one obtains a
$G$-representation $U=\Ind^G_{G_e} U_e$ from $U_e$ 
such that $M_{G,X}(\End(U))$
has finite global dimension. It is clear that $U$ still contains a trivial direct summand. 
\subsection{Proof of Theorem \ref{ref-1.3.3-7}}
\label{ref-11.5-117}
Let $T$ be the homogeneous coordinate ring of $(X,\Lscr)$. Then $X=\Proj T$ and $X^{ss}\quot G=\Proj T^G$.  
Since $T$ and $T^G$ are finitely generated commutative rings, they have Veronese subrings generated in degree one. 
Hence replacing $\Lscr$ by a suitable power we may assume that $T$ and $T^G$ are both generated in degree one. We choose a $G$-representation $W$ together with a $G$-equivariant graded surjective map $SW\r T$ ($W$ in degree one). 
As in the proof of Theorem \ref{ref-1.3.1-5} above we may find a~$G$-representation $U$ containing the trivial 
representation such that $M_{SW}(\End(U))$ has finite global dimension. 

By construction $X^{ss}$ has a covering by $G$-invariant affine
open sets $\Spec (T_f)_0$ where $f$ runs through the homogeneous
elements of $T^G$ of strictly positive degree.  For each such $f$ let
$\tilde{f}$ be a lift of $f$ in $(SW)^G$. Then
$M_{SW_{\tilde{f}}}(\End(U))=M_{SW}(\End(U))_{\tilde{f}}$ also has
finite global dimension. Now $\Spec (T_f)_0$ is an affine open subset of  $X$ and
since $X$ is smooth the same holds for $\Spec (T_f)_0$. Since $T$ is
generated in degree one, $T_f$ is strongly graded (see the proof of Lemma \ref{ref-4.5.1-63}) and hence $\Spec
T_f$ is smooth as well. So $\Spec T_f$ is a smooth, closed subvariety
of $\Spec (SW)_{\tilde{f}}$ and moreover since $\tilde{f}$ is an
invariant function 
it follows that closed orbits
in $\Spec T_f$ are also closed in $\Spec SW$. Hence by Theorem
\ref{ref-4.3.1-48} we obtain that $M_{T_f}(\End(U))$ has finite global
dimension.

Now the fact that $T_f^G$ is strongly graded implies that $M_{T_f}(\End(U))$ is strongly graded and hence 
$M_{T_f}(\End(U))_0$ has finite global dimension \cite[Thm I.3.4]{NVO}. Since $M_{T_f}(\End(U))_0$ are the sections of $M^{ss}(\End(U))$ 
on $\Spec (T^G_f)_0$, we are done.

\subsection{Proof of Proposition \ref{ref-1.3.6-10}}
\label{ref-11.6-118}
Let $S$ be as in the statement of the Proposition. 
\subsubsection{Preliminaries} We recall some ingredients
from the theory of affine toric varieties. Let $M$ be the
quotient group of $S$. Then following e.g.\ \cite{BrGu} there exists a
finite irredundant set of primitive $(\sigma_i)_{i=1}^d\in M^\ast$ such that
$S=M\cap \sigma^\vee$ where $\sigma^\vee$ is the cone $\{x\in M_\RR\mid
\forall i:\sigma_i(x)\ge 0\}$. One knows that $\sigma^\vee$ contains 
an interior point and no linear subspace. 

For $\gamma\in M_\RR$ define
$
S_\gamma=\{z\in M\mid \sigma_i(z)\ge \sigma_i(\gamma)\}
$. If $\gamma'=\gamma+m$ with $m\in M$ then
$S_{\gamma'}=m+ S_\gamma$. Thus $S_\gamma$ as an $S$-module depends up to isomorphism only
on $\bar{\gamma}\in M_{\RR}/M$. Following \cite{BrGu} we call 
$S_\gamma$ a \emph{conic} fractional $S$-ideal. Note that since $\sigma_i$ takes integral values on $M$ there exists
${\bar{n}}\in \NN$ such that if $n\ge {\bar{n}}$ then for every $\gamma\in M_{\RR}$ there exists
$\gamma'\in (1/n)M$ such that $S_\gamma=S_{\gamma'}$.
The following result is stated in \cite{BrGu} (see also \cite{SmithVdB,Yasuda}).
\begin{lemmas}
\label{ref-11.6.1-119} Let $n>0$.  The $S$-module $(1/n)S\subset M_\RR$ 
  is a union of translates (in $M_\RR$) of conic fractional ideals. Moreover
every conic fractional ideal occurs in this way if $n\ge {\bar{n}}$.
\end{lemmas}
\begin{proof}
 Choose representatives $(\gamma_i)_i$ for $(1/n)M/M$.  One verifies
$
(1/n)S=\coprod_i(-\gamma_i+S_{\gamma_i})
$.
\end{proof}
According to \cite{chouinard} there is an exact sequence
\begin{equation}
\label{ref-11.8-120}
0\r M\xrightarrow{\phi} \ZZ^d\r \Cl(k[S])\r 0
\end{equation}
where $\phi(m)=(\sigma_i(m)_i)$ such that $S=\phi^{-1}(\NN^d)$. According to \cite{BrGu} this makes $R\overset{\text{def}}{=}k[\NN^d]$ into a $\Cl(k[S])$-graded ring such that its part of degree zero is $k[S]$.  Put $G=\Hom(\Cl(k[S]),G_m)$. This
is an abelian algebraic group which is (non-canonically) the sum of a torus and a finite group.
The $\Cl(k[S])$-grading may be transformed into a $G$-action such that one has
$R^G=k[S]$ and $X(G)=\Cl(k[S])$.

 Let $w_i$ be the $i$'th generator of $\NN^d$ considered as an element
of $R$ and let us define the degree of $w_i$ to be its image $\beta_i\in \Cl(k[S])$. According to
the equivalence between (1)(3) in
\cite[Thm 2]{chouinard} one obtains that for all $1\le i\le d$: 
\[
\NN\beta_1+\cdots+\widehat{\NN\beta_i}+\cdots+\NN\beta_d=\Cl(k[S]).
\]
It is easy to see that this implies that $W:=\sum_i kw_i$ is generic. 

Let $T=G_e$ be the connected component of the identity. This is a torus.
\begin{definitions} An element $\chi\in \Cl(k[S])=X(G)$ is strongly critical if
it is strongly critical for $T$ (cfr \S\ref{ref-4.4.1-58}).
\end{definitions}
The following result is usually stated in the case that $G=T$.
\begin{lemmas} 
\label{ref-11.6.3-121} For every $\gamma\in M_\RR$ there exists a strongly critical weight $\chi$ such that $k[S_\gamma]=M(-\chi)$ and vice versa if $M(-\chi)\neq 0$.
\end{lemmas}
\begin{proof}
Using the map $\phi$ above we have the following identifications
\begin{align*}
M&=\{(a_i)_i\in \ZZ^d\mid \sum_i a_i\beta_i=0\text{ in $\Cl(k[S])$}\},\\
M_\RR&=\{(a_i)_i\in \RR^d\mid \sum_i a_i\beta_i=0\text{ in $\Cl(k[S])_\RR$}\}
\end{align*}
with $\sigma_i((a_i)_i)=a_i$.
Hence $\gamma\in M_\RR$ may be identified with  $(\gamma_i)_i\in \RR^d$ such that $\sum_i\gamma_i\beta_i=0$
in $\Cl(k[S])_\RR$. We obtain
\[
S_\gamma=\{(a_i)_i\in \ZZ^d\mid \sum_i a_i\beta_i=0, a_i\ge \gamma_i\}.
\]
Write $\gamma_i=n_i+\delta_i$ with $n_i\in \ZZ$ and $\delta_i\in ]-1,0]$. 
$b_i=a_i-n_i$, $\chi=-\sum_i n_i\beta_i\in\Cl(k[S])$.
Then we find that as $S$-fractional ideal we have
\[
S_\gamma=(n_i)_i+ \{(b_i)_i\in \NN^d\mid \sum_i b_i\beta_i=\chi\}.
\]
In other words
$
k[S_\gamma]\cong M(-\chi)
$.
Note that $\chi=\sum_i \delta_i \beta_i$ in $\Cl(k[S])_\RR=X(T)_\RR$ so $\chi$ is
strongly critical. It is easy to see that this procedure is reversible.
\end{proof}
\subsubsection{Proof of Proposition \ref{ref-1.3.6-10}}
We let $(G,T,W)$, etc\dots, be as in the previous section. Recall that
$W$ is generic.
We may write $G=T\times A$ where $A$ is a finite abelian group.
Combining Corollary \ref{ref-1.5.2-25} for $G=G_e=T$, $\Delta=0$ with Lemma \ref{ref-4.5.1-63}  we see that if we put
\begin{align*}
{\Lscr}&=\{\mu_1\otimes \mu_2\mid \mu_1\in X(A),\mu_2\in \Sigma\cap X(T)\},\\
M&=\bigoplus_{\chi\in {\Lscr}} M(\chi),\\
\Lambda&=\End_{R^G}(M)
\end{align*}
then $\Lambda$ is a NCR for $R^G=k[S]$. However by Lemma \ref{ref-11.6.3-121} and
Lemma \ref{ref-11.6.1-119} 
$M$ has the same summands as $k[(1/n)S]$ for $n$ large. This finishes the proof.

\section{Proofs of the main results: the quasi-symmetric case}
\label{ref-12-122}
\subsection{Proof of Theorem \ref{ref-1.6.1-27}}
\label{ref-12.1-123}
The proof runs parallel with the one of Theorem \ref{ref-1.5.1-24}. We only highlight  the differences. Instead of
$r_\chi\ge 1$, we now have $r_{\chi}>1/2$. Under this condition we have to show that
if
\[
\mu'=\chi+\sum_{i\in S} \beta_i,
\]
with $S=\{i_1,\ldots,i_{-p}\}\neq\emptyset$
as in \eqref{ref-11.6-112} 
then $r_{\mu'}<r_\chi$ or $p_{\mu'}<p_\mu$ where the following additional conditions are satisfied
\begin{enumerate}
\item
$\chi=-\bar{\rho}+\sum_i a_i\beta_i+\delta$ with $a_i\in [-r_\chi,0]$ and $\delta\in \Delta$
with the number of $a_i$ satisfying $a_i=-r_{\chi}$ being minimal.
\item All $\beta_i$ for $i\in S$ are in an open half space $\langle \lambda,-\rangle>0$.
\item $a_i= -r_\chi$ for $i\in T_\lambda:=\{i\mid \langle\lambda,\beta_i\rangle>0\}\supset S$ (by ``observation (3)'' in the proof of Theorem of \ref{ref-1.5.1-24}).
\end{enumerate}
In addition we may and we will assume
\begin{enumerate}
\setcounter{enumi}{3}
\item For every line $\ell\subset X(T)_\RR$
through the origin the $\beta_i\in\ell$ with $a_i\neq 0$ are all
\emph{on the same side of the origin}.
\end{enumerate}
This uses the fact that the set
$\{\beta_i\in \ell\}$ contains only the zero weight, or else contains weights $\beta_i$ on
both sides of the origin (since otherwise $\sum_{\beta_i\in \ell}\beta_i\neq
0$).

We have
\[
\mu'=-\bar{\rho}+\sum_i a'_i\beta_i+\delta,
\]
\[
a'_i=
\begin{cases}
a_i&\text{$i\not \in S$}\\
a_i+1&\text{$i \in S$}
\end{cases}
\]
We now write
\begin{equation}
  \label{ref-12.1-124}
\mu'=-\bar{\rho}+\sum_\ell\sum_{\beta_i\in\ell} a_i'\beta_i+\delta
\end{equation}
where the sum is over the lines $\ell\subset X(T)_\RR$ through the origin. 
 Fix such a line.
If $\ell\cap \{\beta_i\mid i\in S\}=\emptyset$ then $a'_i=a_i$ for all $i$ such that $\beta_i\in \ell$
and hence $a'_i\in [-r_\chi,0]$.

We assume now that $\ell\cap \{\beta_i\mid i\in S\}\neq\emptyset$
(note that it is clear that there are $\ell$ for which this holds). In
particular $\langle \lambda,-\rangle$ is non-zero on $\ell$. Let
$\gamma_u$ be a unit vector on $\ell$ such that
$\langle\lambda,\gamma_u\rangle>0$.

Put $T_\lambda^\ell=\{i\in T_\lambda\mid \beta_i\in \ell\}$, $S^\ell=S\cap T_\lambda^\ell$. For $\beta_i\in\ell$ put
 $\beta_i=c_i\gamma_u$. Then $c_i>0$ for all $i\in T_\lambda^\ell$. 
If $i\in S^\ell$ then by (3) $a_i=-r_{\chi}\neq 0$. By (4)
we deduce from this that if $\beta_i\in \ell-\{0\}$ is such that $a_i\neq 0$ then  $i\in T^\ell_\lambda$. We then
compute
\begin{align*}
\sum_{\beta_i\in \ell}a'_i\beta_i
&=\sum_{i\in S^\ell}(1-r_{\chi})\beta_i
+
\sum_{i\in T_\lambda^\ell\setminus S^\ell}(-r_{\chi})\beta_i
+
\sum_{\beta_i\in\ell, i\not\in T_\lambda^\ell}a_i\beta_i\\
&=\sum_{i\in S^\ell}(1-r_{\chi})\beta_i
+
\sum_{i\in T_\lambda^\ell\setminus S^\ell}(-r_{\chi})\beta_i\\
&=\pi\gamma_u
\end{align*}
with
\[
\pi=\sum_{i\in S^\ell}(1-r_{\chi})c_i
+
\sum_{i\in T_\lambda^\ell\setminus S^\ell}(-r_{\chi})c_i
\]
where there is at least one term of the form $(1-r_{\chi})c_i$.
Set $c=\sum_{i\in T_\lambda^\ell}c_i$.  Since $c_i>0$ for $i\in T^\ell_\lambda$ we have
\[
\pi>\sum_{i\in S^\ell}(-r_{\chi})c_i
+
\sum_{i\in T_\lambda^\ell\setminus S^\ell}(-r_{\chi})c_i=-r_{\chi}c.
\]
On the other hand as $ r_{\chi}>1/2$ we have $1- r_{\chi}<r_{\chi}$, $-r_{\chi}<r_{\chi}$ and hence
\[
\pi
<\sum_{i\in S^\ell}r_{\chi}c_i
+
\sum_{i\in T_\lambda^\ell\setminus S^\ell}r_{\chi}c_i=r_{\chi}c.
\]
Assume $\pi< 0$ and put $a'=\pi/c$. Then we have $a'\in ]-r_{\chi},0]$ and 
\begin{equation}
\label{ref-12.2-125}
\sum_{\beta_i\in \ell}a'_i\beta_i=\pi\gamma_u=ca'\gamma_u=\sum_{i\in T^\ell_\lambda} a'\beta_i.
\end{equation}
Similarly assume $\pi\ge 0$ and put $a'=-\pi/c$. Then again we have $a'\in ]-r_\chi,0]$ and 
\begin{equation}
\label{ref-12.3-126}
\sum_{\beta_i\in \ell}a'_i\beta_i=\pi\gamma_u=-ca'\gamma_u=-\sum_{i\in T^\ell_\lambda} a'\beta_i=
\sum_{\beta_i\in\ell,i\not\in T^\ell_\lambda} a'\beta_i.
\end{equation}
Note that in the last equality we finally used the full force of the
hypothesis $\sum_{\beta_i\in\ell}\beta_i=0$.  Plugging the righthand
sides of \eqref{ref-12.2-125}\eqref{ref-12.3-126} into \eqref{ref-12.1-124} we
conclude that either $r_{\chi'}<r_{\chi}$ or
else $p_{\chi'}<p_{\chi}$, contradicting the
minimality of $\chi$.

\subsection{Proof of Theorem \ref{ref-1.6.3-29}}
\label{ref-12.2-127}
We first prove that $\gldim \Lambda<\infty$.
For small strictly positive $r$ we have
\begin{align*}
X(T)_{\bar{\mu}}^+\cap (-\bar{\rho}+(1/2)\bar{\Sigma}_\varepsilon)
&=X(T)_{\bar{\mu}}^+\cap (-\bar{\rho}+(1/2)(\bar{\Sigma}\cap (r\varepsilon+\bar{\Sigma})))\\
&=X(T)_{\bar{\mu}}^+\cap (-\bar{\rho}+(1/2)(r\varepsilon+ \bar{\Sigma}))\\
&=X(T)_{\bar{\mu}}^+\cap (-\bar{\rho}+r\varepsilon/2+ (1/2)\bar{\Sigma})
\end{align*}
since for such small $r$ one has that 
$
X(T)_{\bar{\mu}}^+\cap (-\bar{\rho}+(1/2)((r\varepsilon+\bar{\Sigma})\setminus \bar{\Sigma}))$ is empty.
We now apply Theorem \ref{ref-1.6.1-27} with $\Delta=\{r\epsilon/2\}$.

\medskip

The statement about the index of $\Lambda$ in the generic case follows from Proposition \ref{ref-4.1.6-40}.

\subsection{Proof of Theorem \ref{ref-1.6.4-30}}
\label{ref-12.3-128}
Using Theorem \ref{ref-1.6.3-29} it is sufficient to prove that  $\Lambda$ is Cohen-Macaulay. I.e. if 
$\chi_1,\chi_2\in X(T)^+_{\bar{\mu}}\cap (-\bar{\rho}+(1/2)\bar{\Sigma}_\varepsilon)$ then $M(V(\chi_1)^\ast\otimes V(\chi_2))=
M(V(\chi_1)^\ast\otimes V(-w_0\chi_2)^\ast)$ is Cohen-Macaulay. By Proposition \ref{ref-4.4.4-61} it is sufficient to 
prove that $\chi_1-w_0\chi_2$ is strongly critical.

We have $-w_0\chi_2\in -\bar{\rho}+(1/2)(-\bar{\Sigma})$. Since $\sum_i\beta_i=0$ it easy to see that  
$-\bar{\Sigma}=\bar\Sigma$. Thus $\chi_1-w_0\chi_2\in -2\bar{\rho}+
(1/2)(\bar{\Sigma}+\bar{\Sigma}))=-2\bar{\rho}+\bar{\Sigma}$.

If $\chi=\sum_i a_i\beta_i$ with $a_i\in ]-1,0]$ then by subtracting a
small multiple of the identity $\sum_i \beta_i=0$ we may assume that
$a_i\in ]-1,0[$. Hence $\Sigma$ is relatively open and thus it is
equal to $\bar{\Sigma}-\partial\bar{\Sigma}$. Assume that $\chi_1-w_0\chi_2\not\in -2\bar{\rho}+\Sigma$. This
is only possible if $\bar{\rho}+\chi_1$ and $\bar{\rho}-w_0\chi_2$ are elements of the same boundary face $F$ of 
$(1/2)\bar{\Sigma}$.
Then $F$ must be a boundary face of $(1/2)\Sigma_\varepsilon$ and a boundary face of $(1/2)(-w_0)(\Sigma_\varepsilon)=
(1/2)\Sigma_{-\varepsilon}$. So $F$ is in fact a boundary face of $(1/2)\Sigma_{\pm\varepsilon}$.  
Now $\chi_1\in X(T)_{\bar{\mu}}\cap(-\bar{\rho}+F)$ which is empty by the hypothesis \eqref{ref-1.5-31}. This is a contradiction.
\subsection{Proof of Theorem \ref{ref-1.6.2-28}}
\label{ref-12.4-129}
This result is an immediate consequence of Theorem \ref{ref-1.6.4-30}, taking $\bar{G}=G=T$ and $\varepsilon$ generic.
\appendix

\section{A refinement of the  $E_1$-hypercohomology spectral sequence}
\label{ref-A-130}
Below $\Cscr$ is an abelian category. 
\begin{lemma} 
\label{ref-A.1-131} Assume that $I^\bullet$ is a complex over $\Cscr$ with projective homology. Then
there is a quasi-isomorphism 
\[
j:\left(\bigoplus_n H^n(I^\bullet)[-n],0\right)\r I^\bullet.
\]
\end{lemma}
\begin{proof}
For each $i$ choose a splitting $\beta_i$ for the projection $Z^i(I^\bullet)\r H^i(I^\bullet)$
and let  $j_i$ be the composition
\[
H^i(I^\bullet)\xrightarrow{\beta_i} Z^i(I^\bullet)\r I^i.
\]
It now suffices to take $j=\oplus_i j_i$.
\end{proof}
We now discuss a two-dimensional variant of this result.  If $A^{\bullet\bullet}$ is a bigraded object in $\Cscr$ then 
 $\Tot_{\oplus}(A^{\bullet\bullet})$ is the graded object in $\Cscr$ given by
\[
\Tot_{\oplus}(A^{\bullet\bullet})^m\overset{\text{def}}{=}\bigoplus_{p+q=m}  A^{pq}
\]
provided this coproduct is finite.
We will usually write $\Tot_{\oplus}(A)=\bigoplus_{pq}A^{pq}[-p-q]$.

A twisted differential on $A^{pq}$ is a collection of maps
$d_{n}^{pq}: A^{pq}\r A^{p+n,q-n+1}$ such that
$d=\sum_{pqn}d^{pq}_n$ induces a differential on
$\Tot_{\oplus}(A^{\bullet\bullet})$. In other word we require for all
$p,q,p'$
\begin{equation}
\label{ref-A.1-132}
\sum_{n+n'=p'-p} d^{p+n,q-n+1}_{n'}  d^{pq}_n=0.
\end{equation}
Below we will make the adhoc definition
that a twisted differential is horizontal (htd) if  $d^{pq}_n=0$ for $n\le 0$.
Note that in that case $d_1^2=0$ for $d_1=\sum_1 d^{pq}_1$.
\begin{lemma}
\label{ref-A.2-133}
Let $I^{\bullet\bullet}$ be a double complex over $\Cscr$. Assume
\begin{enumerate}
\item There are some $p_0,p_1$ such that $I^{pq}=0$ for all $q$ and all $p\not\in [p_0,p_1]$.
\item For each $p$, $I^{p,\bullet}$ has bounded cohomology.
\item For each $p$ the cohomology of $I^{p,\bullet}$ is projective.
\end{enumerate}
Then there exists a quasi-isomorphism
\[
\left(\bigoplus_{pq} H^q(I^{p,\bullet})[-p-q],d\right)\rightarrow \Tot_{\oplus} I^{\bullet\bullet}
\]
where the differential on the left is obtained from a htd on $(H^q(I^{p,\bullet}))_{pq}$ with
$d_1^{pq}:H^q(I^{p,\bullet})\r H^q(I^{p+1,\bullet})$ being obtained from the differential
$I^{pq}\r I^{p+1,q}$ in $I^{\bullet\bullet}$.
\end{lemma}
\begin{proof}
We have an exact sequence of bicomplexes
\[
0\r I^{\ge p_0+1,\bullet}\r I^{\ge p_0,\bullet}\r I^{p_0,\bullet}\r 0,
\]
which is split if we ignore the horizontal differential. Hence the total complexes form a distinguished
triangle in the homotopy category of complexes over $\Cscr$. After rotating this distinguished triangle
we obtain a quasi-isomorphism
\begin{equation}
\label{ref-A.2-134}
\Tot_{\oplus} I^{\ge p_0\bullet} \cong \cone(\Tot_{\oplus} I^{p_0,\bullet}[-1] \xrightarrow{\theta} \Tot_{\oplus} I^{\ge p_0+1,\bullet}).
\end{equation}
for suitable $\theta$.
Using Lemma \ref{ref-A.1-131} we have in $D(\Cscr)$
\begin{equation}
\label{ref-A.3-135}
\Tot_{\oplus} I^{p_0,\bullet}\cong \left(\bigoplus_q H^{q}(I^{p_0,\bullet})[-p_0-q],0\right)
\end{equation}
and by induction we also have
\begin{equation}
\label{ref-A.4-136}
\Tot_{\oplus} I^{\ge p_0+1,\bullet} \cong \left(\bigoplus_{p\ge p_0+1,q} H^q(I^{p,\bullet})[-p-q],d'\right)
\end{equation}
where $d'$ is obtained from a htd. 
Substituting \eqref{ref-A.3-135},\eqref{ref-A.4-136} in \eqref{ref-A.2-134}  and noting that now~$\theta$ becomes 
a map between bounded projective complexes, we find that $\theta$ is represented by an actual map of complexes in $\Cscr$. The lemma now follows
using the standard construction of the cone. 
\end{proof}
The following is a refined version of \cite[Proposition 4.4]{VdB34}.
\begin{theorem} \label{ref-A.3-137} Let $F:\Ascr\r \Bscr$ be a left exact functor between abelian categories and assume in addition
that $\Ascr$  has enough injectives. Let $A^\bullet$ be a (literally) bounded complex in $\Ascr$ and assume in addition that
for each $p$, $R^q F(A^p)\in \Ob(\Bscr)$ is projective for all $q$, and zero for $q\gg 0$. Then there is an isomorphism
in $D(\Bscr)$
\begin{equation}
\label{ref-A.5-138}
\left(\bigoplus_{p,q} R^q F(A^p)[-p-q],d\right) \cong RF(A^\bullet)
\end{equation}
where $d$ on the left is obtained from a htd such that 
$d_{1}^{pq}: RF^q(A^p)\r RF^{q}(A^{p+1})$ is equal to $RF^q(d^p_{A^\bullet})$.
\end{theorem}
\begin{proof}
  Let $A^\bullet\r I^{\bullet\bullet}$ be an injective
  Cartan-Eilenberg resolution of $A^\bullet$ (with $I^{p,\bullet}$
  resolving $A^p$). Then $RFA$ is computed by
  $F(\Tot_{\oplus}(I^{\bullet\bullet}))$. It now suffices to use Lemma \ref{ref-A.2-133}
  (with $F(I^{\bullet\bullet})$ playing the role of
  $I^{\bullet\bullet}$).
\end{proof}
\begin{remark} Note that \eqref{ref-A.5-138} is a refined version  of the standard hypercohomology spectral sequences
\[
E^{pq}_1=RF^q(A^p)\Rightarrow R^{p+q}F(A^\bullet).
\]
\end{remark}
\section{Faces of some polygons}
\label{ref-B-139}
Let $E$ be a Euclidean space, $(\beta_i)_{i=1,\ldots,d}\subset E$ a collection of points and let $a_i<b_i$, $i=1,\ldots,d$ be a collection of real numbers. Consider the
closed polygon 
\[
\nabla=\left\{\sum_i g_i\beta_i\mid g_i\in [a_i,b_i]\right\}.
\]
For $\lambda\in E^\ast$  write $\nabla_\lambda$ for the set of $\beta=\sum_i g_i\beta_i\in \nabla$
  that satisfy 
\begin{align}
\langle \lambda,\beta_i\rangle>0&\Rightarrow g_i=a_i\label{ref-B.1-140},\\
\langle \lambda,\beta_i\rangle<0&\Rightarrow g_i=b_i\label{ref-B.2-141}.
\end{align}
\begin{lemma}  \label{ref-B.1-142}
$\nabla_\lambda$ is a face of $\nabla$ whose linear span is 
\[
H_\lambda=C_\lambda+\sum_{\langle \lambda,\beta_i\rangle=0} \RR\beta_i
\]
where
\begin{equation}
\label{ref-B.3-143}
C_\lambda=\sum_{\langle\lambda,\beta_i\rangle>0} a_i\beta_i+\sum_{\langle\lambda,\beta_i\rangle<0} b_i\beta_i\in \nabla.
\end{equation}
\end{lemma}
\begin{proof} If $\lambda=0$ then $\nabla_\lambda=\nabla$. Assume $\lambda\neq 0$.
Put  
$
c_\lambda=\langle \lambda,C_\lambda\rangle$.
It is an easy verification that $ \nabla$ is contained in the half space
$
H_\lambda=\{c_\lambda\le
\langle\lambda,-\rangle\}
$
and moreover $\nabla_\lambda$ is the intersection $\nabla$ with $\partial H_\lambda=\{c_\lambda=
\langle\lambda,-\rangle\}$.
This proves that $\nabla_\lambda$ is a face. The claim about the linear span is clear.
\end{proof}
\begin{lemma} If $F$ is a face in $\nabla$ then $F=\nabla_\lambda$
for suitable $\lambda$.
\end{lemma}
\begin{proof}
If $F=\nabla$ then $F=\nabla_0$. If $F\neq\nabla$ then there exists 
$\lambda\in E^\ast$, $c\in\RR$ such that $\nabla\subset 
H_\lambda=\{\langle \lambda,-\rangle\ge c\}$
and $F=\nabla\cap \partial H_\lambda$.
Let $c_\lambda=\langle\lambda,C_\lambda\rangle$ be as above. Then the minimum of $\langle\lambda,-\rangle$
attained on $\nabla$ is both $c$ and $c_\lambda$ and these minima are achieved on
$F$ and $\nabla_\lambda$ respectively. Hence $c_\lambda=c$ and $F=\nabla_\lambda$.
\end{proof}
\begin{corollary} 
\label{ref-B.3-144} $\nabla$ is the convex hull of $C_\lambda$ where $\lambda$ runs through those
  elements of $E^\ast$ such that $\langle \lambda,\beta_i\rangle\neq 0$ for all $i$.
\end{corollary}
\begin{proof} By the above discussion the set of $C_\lambda$ we have described is precisely the set
of vertices of $\nabla$.
\end{proof}
\section{Supporting hyperplanes of Minkowski sums}
This section is related to Appendix \ref{ref-B-139}. Presumably the following result is standard.
\begin{lemma}
Let $({\Pi}_i)_{i=1,\ldots,n}$ be closed convex sets in a finite dimensional vector space $E$ over $\RR$ and let
${\Pi}=\{\sum_i x_i\mid x_i\in{\Pi}_i\}$ be their Minkowski sum. Let $x\in  \Pi$. Then there 
exists $\lambda\in E^\ast$ such that $\Pi$ is contained in the set $\langle \lambda,-\rangle\ge \langle \lambda,x\rangle$ and such that 
 $x$ can be written as $\sum_i x_i$ with $x_i\in {\Pi}_i$ in such
a way that $\langle \lambda,-\rangle$ is constant on ${\Pi}_i$ if and only if $x_i\not\in\partial{\Pi}_i$.
\end{lemma}
\begin{proof} For $z\in {\Pi}$ let $p(z)$ be the minimal number of $z_i\in \partial {\Pi}_i$ among all ways
of writing $z=\sum_i z_i$ with $z_i\in {\Pi}_i$.  Note that for $z_1,z_2\in {\Pi}$ and $z\in ]z_1,z_2[$
we have 
\begin{equation}
\label{ref-C.1-145}
p(z)\le \min (p(z_1),p(z_2)).
\end{equation}

With $x$ as in the statement of the lemma, write $x=\sum_i x_i$ such that the number of $x_i$ in $\partial {\Pi}_i$ is mininal.
Let $\Gamma$ be the polyhedral cone spanned by all $y-x_i$ with $i\in\{1,\ldots,n\}$ and 
$y\in {\Pi}_i$. Note that if $\sigma\in \Gamma$ then for $\epsilon>0$ small enough we have
$x+\epsilon\sigma\in {\Pi}$ and moreover $p(x+\epsilon \sigma)\le p(x)$.

Let $L$ be the maximal linear subspace in $\Gamma$ (which could be $E$) and let $\lambda\in E^\ast$ be such
that $\langle \lambda,-\rangle > 0$ on $\Gamma-L$ and $\langle \lambda,-\rangle=0$ on $L$.

If $x_i \not\in\partial{\Pi}_i$ then $y-x_i$ with $y\in {\Pi}_i$ spans a linear subspace of $\Gamma$ and hence
$\langle \lambda,y-x_i\rangle =0$ for $y\in {\Pi}_i$.

Assume $x_i\in \partial {\Pi}_i$. It is sufficient to prove that $y-x_i\in \Gamma-L$
for $y\in \relint {\Pi}_i$. If $y-x_i\in L$ then
$-(y-x_i)\in \Gamma$ and hence  as noted above, for $\epsilon>0$ small enough we have $x'=x-\epsilon(y-x_i)\in {\Pi}$ as well as $p(x')\le p(x)$. On the other hand 
we have for $x''=x+(y-x_i)\in {\Pi}$:
$p(x'')
<p(x)$. Since $x\in ]x_1,x_2[$ it now suffices to invoke \eqref{ref-C.1-145} to obtain the
contradiction $p(x)<p(x)$.
\end{proof}
Now we use the notations of Appendix \ref{ref-B-139}. We define in addition
\[
\Sigma=\left\{\sum_i g_i\beta_i\mid g_i\in ]a_i,b_i]\right\}
\]
such that $\nabla=\bar{\Sigma}$.
\begin{lemma} \label{ref-C.2-146} Let $\Delta$ be a closed convex subset of $X(T)_\RR$. If $x\in( \bar{\Sigma}+\Delta)-(\Sigma+\Delta)$
then there exists $\lambda\in Y(T)_\RR$ such that for all $z\in \Sigma+\Delta$ we have $\langle \lambda,z\rangle>\langle\lambda,x\rangle$.
\end{lemma}
\begin{proof} We apply the previous lemma (changing the indexing) with ${\Pi}_i=[a_i,b_i]\beta_i$ and ${\Pi}_0=\Delta$. There exists
$\lambda\in Y(T)_\RR$ such that we may write $x$ as
\[
x=\sum_i g_i\beta_i+\delta
\]
such that  for 
all $\delta'\in\Delta$ we have $\langle\lambda,\delta\rangle \le \langle\lambda,\delta'\rangle$ and moreover
\begin{align*}
\langle \lambda,\beta_i\rangle >0 &\iff g_i=a_i,\\
\langle \lambda,\beta_i\rangle <0 &\iff g_i=b_i,\\
\langle \lambda,\beta_i\rangle =0 &\iff g_i\neq a_i,b_i.
\end{align*}
Since $x\not \in \Sigma+\Delta$ there is at least one $g_i$ equal to $a_i$. From this one easily deduces the
claim in the statement of the lemma. 
\end{proof}
\section{Some elementary facts about root systems}
\label{ref-D-147}
Let $E$ be a finite dimensional real vector space equipped with a positive definite quadratic
form $(-,-)$. Let $\Phi$ be a root system in $E$ spanning some subspace $E'\subset E$. Let $\Phi^+\subset \Phi$ be a set of positive roots and let $S\subset \Phi^+$
be the corresponding simple roots. For $\rho$ a root let $\check{\rho}$ be the corresponding
coroot, given by $\check{\rho}=2\rho/(\rho,\rho)$. 

We say that $x\in E$ is dominant if $(\rho,x)\ge 0$
for all $\rho\in \Phi^+$. The reflection associated to a simple root $\alpha$ is
\[
s_\alpha(x)=x-(\check{\alpha},x)\alpha
\]
By definition the reflections generate the Weyl group $\Wscr$ of $\Phi$. $\Wscr$ preserves $E,(-,-),\Phi$ 

\begin{lemma}
\label{ref-D.1-148}
Let $\Delta$ be a $\Wscr$-invariant convex subset of $E$. Let $x,y\in E$ be dominant.
If $x+y\in \Delta$
then for all $v,w\in \Wscr$ we have $vx+wy\in\Delta$.
\end{lemma}
\begin{proof}
Without loss of generality we may assume $v=1$. Since $y$ is dominant there exist simple roots
$\alpha_1,\ldots,\alpha_n$ such that $w=s_{\alpha_n}\cdots s_{\alpha_1}$ and such that if we put
$y_i=s_{\alpha_i}\cdots s_{\alpha_1}y$ then $(\alpha_{i+1},y_i)>0$. By induction on $i$ we may assume
$x+y_i\in \Delta$. Then we have $s_{\alpha_{i+1}}(x+y_i)\in \Delta$ and hence
\begin{align*}
x+y_i-0\cdot {\alpha}_{i+1}&\in \Delta,\\
x+y_i-(\check{\alpha}_{i+1},x+y_i){\alpha}_{i+1}&\in \Delta.
\end{align*}
 Now note
\[
x+y_{i+1}=x+y_i-(\check{\alpha}_{i+1},y_i){\alpha}_{i+1}
\]
and hence since $x$ is dominant and therefore $(\check{\alpha}_{i+1},x)\ge 0$, we have
\[
0\le (\check{\alpha}_{i+1},y_i)\le (\check{\alpha}_{i+1},x+y_i).
\]
Thus $x+y_{i+1}$ is in the convex hull of $x+y_i$ and $s_{\alpha_{i+1}}(x+y_i)$ and therefore $x+y_{i+1}\in \Delta$.
\end{proof}
Let $\bar{\rho}=(1/2)\sum_{\rho\in \Phi^+}\rho$. Then it is well known that for each simple root we have
$s_\alpha(\bar{\rho})=\bar{\rho}-\alpha$. 
Put
\[
E_\ZZ=\{x\in E\mid \forall \rho\in \Phi:(\hat{\rho},x)\in \ZZ\}.
\]
For $w\in \Wscr$, $x\in E$ put $w{\ast} x=w(x+\bar{\rho})-\bar{\rho}$. This defines an (affine) action of
$\Wscr$ on $E$, which preserves $E_\ZZ$.
\begin{lemma}
Assume that $y\in E$ is dominant.
If $x\in E$ and $\alpha\in S$ is such that $(\alpha,x)<0$ then 
\begin{equation}
\label{ref-D.1-149}
(s_\alpha x,y)\ge (x,y).
\end{equation}
Moreover if $x\in E_\ZZ$ then also
\begin{equation}
\label{ref-D.2-150}
(s_\alpha {\ast} x,y)\ge (x,y).
\end{equation}
\end{lemma}
\begin{proof} We compute
\[
(s_\alpha x,y)=(x-(\check{\alpha},x)\alpha,y)=(x,y)-(\check{\alpha},x)(\alpha,y)\ge (x,y)
\]
using the fact that $y$ is dominant and hence $(\alpha,y)\ge 0$.
Similarly
\[
(s_\alpha {\ast} x,y)=(x-((\check{\alpha},x)+1)\alpha,y)=(x,y)-((\check{\alpha},x)+1)(\alpha,y)\ge (x,y)
\]
using now in addition that $(\check{\alpha},x)\in \ZZ$.
\end{proof}
\begin{corollary} \label{ref-D.3-151}
Assume that $y\in E$ is dominant and $x\in E$ is arbitrary. If $w\in \Wscr$ is such
that $wx$ is dominant then 
\[
(wx,y)\ge (x,y).
\]
If $x\in E_\ZZ$ and $w{\ast} x$ is dominant then
\[
(w{\ast} x,y)\ge (x,y).
\]
\end{corollary}
\begin{proof} For the first inequality note that if $wx$ is dominant then it can be written as $s_{\alpha_n}\cdots s_{\alpha_1} x$
such that for each $x_i=s_{\alpha_i}\cdots s_{\alpha_1} x$ the inequality $(\alpha_{i+1},x_i)<0$ is satisfied. It now suffices to invoke \eqref{ref-D.1-149}. The 
argument for $w{\ast} x$ is similar, now using \eqref{ref-D.2-150}.
\end{proof}
\section{NCCRs for hypersurface singularities}
\label{Dao}
The following result is an analogue of \cite[Thm 2.7]{Dao} for odd dimensional hypersurface singularities. Combining Dao's result with ours\footnote{Hailong Dao was also aware of our result.} we obtain that isolated hypersurface singularities 
of dimension $\ge 4$ do not admit NCCRs.
\begin{proposition} Let $R$ be a commutative noetherian local $k$-algebra over an infinite field $k$ with an isolated singularity. Assume $\hat{R}\cong S/(f)$ where $S$ is an 
(equicharacteristic)  regular local ring and $f$ is a regular element. 
Assume $n:=\dim R$ is odd and $\ge 5$ and let $M$ be a reflexive $R$ module. If $\End_R(M)$ satisfies $S_4$ then $M$ is free.
\end{proposition}
\begin{proof} 
Taking a generic hyperplane section of $R$ we will reduce to the even dimensional case and then we invoke  \cite[Thm 2.7]{Dao}.
Put $\Lambda=\End_R(M)$ and let $P\in \Spec R$ have codimension $d\le 4$. Since $\Lambda$ satisfies $S_4$, $\Lambda_P$ has depth $d$ and hence is Cohen-Macaulay. Since
$\dim R\ge 5$ and $R$ has an isolated singularity we have that $R_P$ is regular local. Hence $\Lambda_P=\End_{R_P}(M_P)$ is projective and therefore by \cite[Thm.\ 4.4]{AG1} 
$M_P$ is projective. Hence $M$ is locally free in codimension $4$. It now follows from\footnote{We are grateful to Hailong Dao for suggesting this simplification of our original argument.}  \cite[Lemma 2.3]{Dao} (applied with $M=N$, $n=2$) 
that $\Ext^1_R(M,M)=0$.
%
%

Let $m$ be the maximal ideal of $R$. 
By \cite[Thm.\ 4.1]{Flenner} 
there exists a non-zero divisor $x\in m-m^2$ such that $\bar{R}=R/xR$ also has an isolated singularity. 
Note that if $\tilde{x}$ is a lift of $x$ in $S$ then $\bar{S}=S/\tilde{x}S$ is regular local and 
the completion of $\bar{R}$ is of the form $\bar{S}/(\bar{f})$.

Below we write $\overline{?}$ for $\bar{R}\otimes_R-$. We have a short exact sequence
\[
0\r M\xrightarrow{x} M\r \bar{M}\r 0
\]
which using $\Ext_R^1(M,M)=0$ implies $\bar{\Lambda}\cong \End_{\bar{R}}(\bar{M})$.
Since $\dim \bar{R}=n-1$ is even and $\ge 4$ and $\bar{\Lambda}$ satisfies $S_3$ it follows from \cite[Thm 2.7]{Dao} 
that $\bar{M}$ is a free $\bar{R}$-module. Hence  $\End_{\bar{R}}(\bar{M})$ is a matrix ring. Lifting idempotents we see that $\hat{\Lambda}=\End_{\hat{R}}(\hat{M})$ is
a matrix ring over $\hat{R}$. Hence $\hat{M}$ is isomorphic to $I^{\oplus m}$ for a reflexive $I$ ideal in $\hat{R}$. Since $\hat{R}$ is a complete intersection of dimension $\ge 4$ it is 
parafactorial \cite[Thm.\ 18.13]{Fossum} and since $\hat{R}$ has an isolated singularity 
 it follows that $\hat{R}$ is factorial.
Therefore $\hat{M}$ is free and hence projective. By faithfully flat descent for $\hat{R}/R$ it follows that $M$ is projective and hence free.
\end{proof}

\def\Spenko{\v{S}penko}\def\cprime{$'$} \def\cprime{$'$} \def\cprime{$'$}
\providecommand{\bysame}{\leavevmode\hbox to3em{\hrulefill}\thinspace}
\providecommand{\MR}{\relax\ifhmode\unskip\space\fi MR }
\providecommand{\MRhref}[2]{%
  \href{http://www.ams.org/mathscinet-getitem?mr=#1}{#2}
}
\providecommand{\href}[2]{#2}

\end{document}